\documentclass[11pt,a4paper,twoside,titlepage]{report}

\usepackage{a4wide}    
\usepackage{graphicx, subfigure}

\usepackage[pdftex]{hyperref}

\usepackage[english]{babel} 

\usepackage{amsmath,amssymb}  
\usepackage{amsthm}         

\usepackage{tikz}

\usetikzlibrary{arrows,chains,matrix,positioning,scopes}

\makeatletter
\tikzset{join/.code=\tikzset{after node path={%
\ifx\tikzchainprevious\pgfutil@empty\else(\tikzchainprevious)%
edge[every join]#1(\tikzchaincurrent)\fi}}}
\makeatother

\tikzset{>=stealth',every on chain/.append style={join},
         every join/.style={->}}

\theoremstyle{plain}                       
\newtheorem{theorem}{Theorem}[chapter]    
\newtheorem*{main}{Main Problem}      
  
\newtheorem{prop}[theorem]{Proposition}	

\theoremstyle{definition}    
\newtheorem{definition}[theorem]{Definition}
\newtheorem{example}[theorem]{Example}

\newtheorem{remark}[theorem]{Remark}      
\newtheorem{corollary}[theorem]{Corollary}

\newcommand{\field}[1]{\mathbb{#1}} 
\newcommand{\R}{\field{R}} 
\newcommand{\N}{\field{N}} 
\newcommand{\C}{\field{C}} 
\newcommand{\Z}{\field{Z}} 
\newcommand{\mP}{\mathcal{P}}
\newcommand{\mI}{\mathcal{I}}
\newcommand{\mQ}{\mathcal{Q}}
\newcommand{\mE}{\mathcal{E}}
\newcommand{\mAss}{\mathcal{A}ss}
\newcommand{\mCom}{\mathcal{C}om}
\newcommand{\mComm}{\mathcal{C}om}
\newcommand{\mLie}{\mathcal{L}ie}
\newcommand{\mPoiss}{\mathcal{P}oiss}
\newcommand{\wh}{^\circlearrowright}
\newcommand{\mAssw}{\mathcal{A}ss\wh}
\newcommand{\mComw}{\mathcal{C}om\wh}
\newcommand{\mCommw}{\mathcal{C}om\wh}
\newcommand{\mLiew}{\mathcal{L}ie\wh}
\newcommand{\mPoissw}{\mathcal{P}oiss\wh}
\newcommand{\tree}{\begin{tikzpicture}
\draw (0,0.25) -- (0,0);
\draw (0,0) -- (-0.2,-0.2);
\draw (0,0) -- (0.2,-0.2);
\end{tikzpicture}}
\newcommand{\comtree}{\begin{tikzpicture}
\draw (0,0.25) -- (0,0);
\draw (0,0) -- (-0.2,-0.2);
\draw (0,0) -- (0.2,-0.2);
\draw (0,0) circle (1pt) [fill=white];
\end{tikzpicture}}
\newcommand{\commtree}{\begin{tikzpicture}
\draw (0,0.25) -- (0,0);
\draw (0,0) -- (-0.2,-0.2);
\draw (0,0) -- (0.2,-0.2);
\draw (0,0) circle (1pt) [fill=white];
\end{tikzpicture}}
\newcommand{\twisttree}{\begin{tikzpicture}
\draw (0,0.25) -- (0,0);
\draw (0,0) -- (-0.06,-0.06) -- (-0.02,-0.1);
\draw (0,0) -- (0.06,-0.06) -- (-0.08,-0.2);
\draw (0.02,-0.14) -- (0.08,-0.2);
\end{tikzpicture}}
\newcommand{\sigmaltree}{\begin{tikzpicture}
\draw (0,0.25) -- (0,0);
\draw (0,0) -- (-0.2,-0.2);
\draw (0,0) -- (0.2,-0.2) node [below right=-3pt and -8pt] {\tiny{$\sigma(3)$}};
\draw (-0.2,-0.2) -- (-0.4,-0.4) node [below left=-2pt and -8pt] {\tiny{$\sigma(1)$}};
\draw (-0.2,-0.2) -- (0,-0.4) node [below right=-2pt and -8pt] {\tiny{$\sigma(2)$}};
\end{tikzpicture}}
\newcommand{\sigmartree}{\begin{tikzpicture}
\draw (0,0.25) -- (0,0);
\draw (0,0) -- (-0.2,-0.2) node [below left=-3pt and -8pt] {\tiny{$\sigma(1)$}};
\draw (0,0) -- (0.2,-0.2);
\draw (0.2,-0.2) -- (0,-0.4) node [below left=-2pt and -8pt] {\tiny{$\sigma(2)$}};
\draw (0.2,-0.2) -- (0.4,-0.4) node [below right=-2pt and -8pt] {\tiny{$\sigma(3)$}};
\end{tikzpicture}}

\newcommand{\afkorting}[2]{afkorting, waarin we het eerste
                           argument #1 en het tweede argument #2 kunnen varieren}

\renewcommand{\vec}[1]{\boldsymbol{#1}}
\newcommand{\convolution}{\ast} 
\DeclareMathOperator{\diag}{diag}    
\DeclareMathOperator{\beeld}{Im}

\newcommand{\RefVgl}[1]{~\textup{(\ref{#1})}} 
\newcommand{\RefFig}[1]{Figuur~\ref{#1}}
\newcommand{\RefTab}[1]{Tabel~\ref{#1}}
\newcommand{\RefSec}[1]{Sectie~\ref{#1}}
\newcommand{\RefHfdst}[1]{Hoofdstuk~\ref{#1}}
\newcommand{\RefStel}[1]{Stelling~\ref{#1}}
\newcommand{\RefLem}[1]{Lemma~\ref{#1}}
\newcommand{\RefGev}[1]{Gevolg~\ref{#1}}
\newcommand{\RefDef}[1]{Definitie~\ref{#1}}
\newcommand{\RefVb}[1]{Voorbeeld~\ref{#1}}
\newcommand{\RefAlg}[1]{Algoritme~\ref{#1}}

\renewcommand{\title}{On Cohomology of Graph Complexes}
\newcommand{\authorA}{Simen Bruinsma}

\newcommand{\mailadresA}{simen.bruinsma@student.uva.nl}
\newcommand{\studnummerA}{0513644}
\newcommand{\begeleider}{dr. Sergey Shadrin}
\newcommand{\tweedebeoordelaar}{dr. Hessel Posthuma}
\newcommand{\datum}{28 augustus, 2010}
\newcommand{\content}{Bachelorscriptie}
\newcommand{\watisdit}{Bachelorscriptie}

\begin{document}

\begin{titlepage}
\null\vfill
\begin{center}
\normalfont {\LARGE \title\par}\bigskip {\Large \authorA
\par
}\bigskip {\Large \datum\par}\end{center} \vfill \vfill
\begin{center}
{\Large \watisdit}\par\bigskip {Begeleiding:
\begeleider}\par\medskip
\end{center}
\vfill
\begin{center}
\end{center}
\vfill
\begin{center}
\leavevmode\normalfont Korteweg de Vries Instituut voor Wiskunde\par\smallskip
Faculteit der Natuurwetenschappen, Wiskunde en Informatica\par\smallskip 
Universiteit van Amsterdam\par\smallskip\medskip
\end{center}
\null
\end{titlepage}

\null\vfill \noindent
\begin{minipage}{0.85\hsize}
\textbf{\Large Abstract}\par\smallskip\noindent
In this thesis we prove that the wheeled Poisson operad is not a wheeled Koszul operad. Chapter 1 introduces operads, the subclass of quadratic operads, and their Koszulness. In chapter 2, all these notions are extended to the case of wheeled operads, and we pose the main question. In chapter 3, some techniques for computing cohomology are introduced, and the first cycle that obstructs the Koszulness of the wheeled Poisson operad is found. 
\end{minipage}

\vfill \vfill \vfill \vfill \noindent
\textbf{\Large Gegevens}\par\smallskip\noindent
Titel: \title\\
Auteur: \authorA, \mailadresA, \studnummerA\\
Begeleider: \begeleider\\
Tweede beoordelaar: \tweedebeoordelaar\\
Einddatum: \datum\par
\bigskip\noindent
Korteweg de Vries Instituut voor Wiskunde\\
Universiteit van Amsterdam\\
Science Park 904, 1098 XH Amsterdam\\
\url{http://www.science.uva.nl/math}{}

\tableofcontents

\chapter*{Introduction}
In the fall of 2008, dr. Sergey Shadrin gave me three articles that he felt could be the basis of a bachelor thesis. Because of the high level of these articles, it was decided that I would focus on the easiest subject: wheeled operads and their Koszulness. In the months thereafter, we followed a path remarkably similar to the one in this thesis: I was first introduced to the concept of operads, and how to visualize their elements as trees. I learned about the free operad, quadratic operads and the Cobar complex, and then about the concept of Koszulness. Then, all these notions were extended to the case of wheeled operads, and the main question was posed: is the wheeled completion of an ordinary Koszul operad always a wheeled Koszul operad? Lastly, I got to calculating, and found the first cycle that breaks the Koszulness of the wheeled Poisson operad. 

I'd like to thank my supervisor, dr. Sergey Shadrin, for introducing me to this very interesting subject, and for his patience, both in explaining the subject and in his persisting help, all the way to the end of this project.

\chapter{Operads}

\section{Operads and trees}

\begin{definition} An \textbf{operad} $\mP$ over a field $k$ is a collection $\{ \mP (n)|n\geq 1 \}$ of vector spaces over $k$, with the following extra structure: 
\begin{enumerate}
	\item[i)] An action of the symmetric group $S_n$ on each $\mP(n)$; 
	\item[ii)] For all $n\in\N$, and all $m_1,...,m_n\in\N$, a linear map
$$\gamma_{m_1,...,m_n}:\mP(n)\otimes\mP(m_1)\otimes ... \otimes\mP(m_n) \rightarrow \mP(m_1+...+m_n)$$
We call $\gamma_{m_1,...,m_n}$ a \textbf{composition}, and we write $\nu(\mu_1,...,\mu_n)$ for $\gamma_{m_1,...,m_n}(\nu,\mu_1,...,\mu_n)$
	\item[iii)] A unit $1\in\mP(1)$, such that $\nu(1,...,1)=1(\nu)=\nu$ for any $\nu\in\mP(n)$, and for any $n$.
\end{enumerate}
We demand that the composition is associative and equivariant with respect to the action of the symmetric group. Both of these notions will be described in the language of trees, which we define below. 
\end{definition}

A natural way to think about operads, and their composition, is in terms of directed trees. We use the following definition here:

\begin{definition} A \textbf{directed tree} is a connected, directed graph without cycles, such that every vertex has at least one incoming edge, and exactly one outgoing edge.
\end{definition}

From now on, we will call these directed trees \textbf{trees}. Using the language of trees, we can graph each element $\nu \in \mP(n)$ of an operad $\mP$ as a tree with one vertex, one outgoing edge, and $n$ incoming edges:
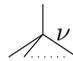
\begin{center}
\begin{tikzpicture}
\draw (0,0.4) -- (0,0) node [right=1pt] {\small{$\nu$}};
\draw (0,0) -- (-0.45,-0.3);
\draw (0,0) -- (-0.25,-0.3);
\draw (-0.15,-0.3) -- (0.35,-0.3) [dotted];
\draw (0,0) -- (0.45,-0.3);
\end{tikzpicture}
\end{center}
(we use the convention that all trees are directed upwards) \\
We view composition as contraction of the internal edges of a tree: 
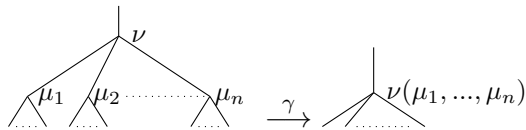
\begin{center}
\begin{tikzpicture}
\draw (0,0.4) -- (0,0) node [right=1pt] {\small{$\nu$}};
\draw (0,0) -- (-1.20,-0.80) node [right=-0pt] {\small{$\mu_1$}};
\draw (0,0) -- (-0.40,-0.80) node [right=-2pt] {\small{$\mu_2$}};
\draw (0.10,-0.80) -- (1.15,-0.80) [dotted];
\draw (0,0) -- (1.20,-0.80) node [right=-2pt] {\small{$\mu_n$}};
\draw (-1.20,-0.80) -- (-1.45,-1.20);
\draw (-1.20,-0.80) -- (-0.95,-1.20);
\draw (-1.35,-1.20) -- (-1.05,-1.20) [dotted];
\draw (-0.40,-0.80) -- (-0.65,-1.20);
\draw (-0.40,-0.80) -- (-0.15,-1.20);
\draw (-0.55,-1.20) -- (-0.25,-1.20) [dotted];
\draw (1.20,-0.80) -- (0.95,-1.20);
\draw (1.20,-0.80) -- (1.45,-1.20);
\draw (1.05,-1.20) -- (1.35,-1.20) [dotted];
\end{tikzpicture} 
$\stackrel{\gamma}{\longrightarrow}$
\begin{tikzpicture} [scale=1.5]
\draw (0,0.4) -- (0,0) node [right=1pt] {\small{$\nu(\mu_1,...,\mu_n)$}};
\draw (0,0) -- (-0.45,-0.3);
\draw (0,0) -- (-0.25,-0.3);
\draw (-0.15,-0.3) -- (0.35,-0.3) [dotted];
\draw (0,0) -- (0.45,-0.3);
\end{tikzpicture}
\end{center}
We think of the unit as a tree with one vertex, one incoming edge, and one outgoing edge:
\begin{center}
\begin{tikzpicture}
\draw (0,0) -- (0,0.4) node [right=-1pt] {\footnotesize{1}};
\draw (0,0.8) -- (0,0.4) node [shape=circle,draw,fill=black,inner sep=0pt,minimum size=1pt] {};
\end{tikzpicture}
\end{center}

So, an arbitrary composition of elements from operads can be thought of as a contraction of a tree to its basic form, \begin{tikzpicture}
\draw (0,0.25) -- (0,-0.00);
\draw (0,0) -- (-0.3,-0.2);
\draw (0,0) -- (-0.15,-0.2);
\draw (-0.10,-0.2) -- (0.25,-0.2) [dotted];
\draw (0,0) -- (0.3,-0.2);
\end{tikzpicture}. 
We will call the incoming edges of a tree or a vertex its inputs, and its outgoing edge its output. Note that the number of inputs and outputs doesn't change when contracting internal edges. 
Because in general, $\nu(\mu_1,\mu_2)\neq\nu(\mu_2,\mu_1)$, the order of inputs of a tree matters. This means that our trees are planar. We will either label the inputs of a tree, or we will assume that the tree is labeled canonnically: 
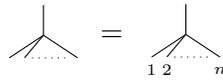
$$^{^{^{\begin{tikzpicture}
\draw (0,0.4) -- (0,0) node [right=1pt] {\small{}};
\draw (0,0) -- (-0.45,-0.3);
\draw (0,0) -- (-0.25,-0.3);
\draw (-0.15,-0.3) -- (0.35,-0.3) [dotted];
\draw (0,0) -- (0.45,-0.3);
\draw (0.9,0) node {=};
\end{tikzpicture}}}} {\begin{tikzpicture}
\draw (0,0.4) -- (0,0) node [right=1pt] {\small{}};
\draw (0,0) -- (-0.45,-0.3) node [below=-1pt] {\tiny{1}};
\draw (0,0) -- (-0.25,-0.3) node [below=-1pt] {\tiny{2}};
\draw (-0.15,-0.3) -- (0.35,-0.3) [dotted];
\draw (0,0) -- (0.45,-0.3) node [below=0pt] {\tiny{$n$}};
\end{tikzpicture}}$$

We can now express the associativity and equivariance conditions. For the associativity, we note that there are multiple ways to order the contracting of edges of a tree, and the associativity condition prescribes that this order of contracting doesn't matter. 
The equivariance tells us that the group action and composition behave ``nicely" with each other: the $\sigma$ in $(\sigma\circ \nu) (\mu_1, ... , \mu_n)$ permutes the ``incoming" elements $\mu_i$, and they ``take along" their inputs. So for example, 
$$ 
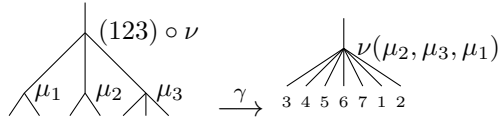
\begin{tikzpicture}
\draw (0,0.4) -- (0,0) node [right=1pt] {\small{$(123)\circ\nu$}};
\draw (0,0) -- (-0.8,-0.8) node [right] {\small{$\mu_1$}};
\draw (0,0) -- (0,-0.8) node [right] {\small{$\mu_2$}};
\draw (0,0) -- (0.8,-0.8) node [right] {\small{$\mu_3$}};
\draw (-0.8,-0.8) -- (-1.0,-1.1);
\draw (-0.8,-0.8) -- (-0.6,-1.1);
\draw (0,-0.8) -- (-0.2,-1.1);
\draw (0,-0.8) -- (0.2,-1.1);
\draw (0.8,-0.8) -- (0.5,-1.1);
\draw (0.8,-0.8) -- (0.8,-1.1);
\draw (0.8,-0.8) -- (1.1,-1.1);
\end{tikzpicture} \stackrel{\gamma}{\longrightarrow} \begin{tikzpicture}
\draw (0,0.4) -- (0,0) node [right=1pt] {\small{$\nu(\mu_2,\mu_3,\mu_1)$}};
\draw (0,0) -- (-0.75,-0.5) node [below] {\tiny{3}};
\draw (0,0) -- (-0.5,-0.5) node [below] {\tiny{4}};
\draw (0,0) -- (-0.25,-0.5) node [below] {\tiny{5}};
\draw (0,0) -- (0,-0.5) node [below] {\tiny{6}};
\draw (0,0) -- (0.25,-0.5) node [below] {\tiny{7}};
\draw (0,0) -- (0.5,-0.5) node [below] {\tiny{1}};
\draw (0,0) -- (0.75,-0.5) node [below] {\tiny{2}};
\end{tikzpicture} $$
Remark that if we say $\mu_i = 1$ for all $i$, we see that
$$ \sigma \circ 
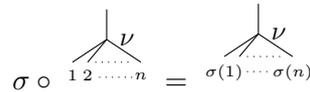
\begin{tikzpicture}
\draw (0,0.4) -- (0,0) node [right=1pt] {\small{$\nu$}};
\draw (0,0) -- (-0.45,-0.3) node [below=-1pt] {\tiny{1}};
\draw (0,0) -- (-0.25,-0.3) node [below=-1pt] {\tiny{2}};
\draw (-0.15,-0.3) -- (0.35,-0.3) [dotted];
\draw (0,0) -- (0.45,-0.3) node [below=0pt] {\tiny{$n$}};
\draw (-0.1,-0.5) -- (0.35,-0.5) [dotted];
\end{tikzpicture} = \begin{tikzpicture}
\draw (0,0.4) -- (0,0) node [right=1pt] {\small{$\nu$}};
\draw (0,0) -- (-0.45,-0.3) node [below=-1pt] {\tiny{$\sigma(1)$}};
\draw (0,0) -- (-0.25,-0.3) node [below=-1pt] {};
\draw (-0.15,-0.3) -- (0.35,-0.3) [dotted];
\draw (0,0) -- (0.45,-0.3) node [below=-1pt] {\tiny{$\sigma(n)$}};
\draw (-0.15,-0.5) -- (0.15,-0.5) [dotted];
\end{tikzpicture} $$

\begin{remark} With the demanded structure, $\mP(1)$ becomes an associative algebra with unit. And in this way, every associative algebra with unit $K$ is an operad $\{\mP(1)=K,\mP(n)=0$ if $n\geq 2\}$.
\end{remark}

\begin{remark} Because of the associativity condition and the existence of $1\in\mP(1)$, we can break our defined composition down into ``smaller", more elementary compositions. For $i \in \{1,...,n\}$, we define 
\begin{align*}
\circ_i: \ \mP(n)\otimes\mP(m) &\rightarrow \mP(n+m-1) \\
(\nu,\mu) \ \ \ \ \ \ &\mapsto \nu(1,...,\mu,...,1)
\end{align*}
with $\mu$ on the $i$th position. We write $\nu\circ_i\mu$ instead of $\circ_i(\nu,\mu)$. 
Now, our original composition can be written in terms of these more basic compositions: 
$$\nu(\mu_1,...,\mu_n)=\nu(1,\mu_2,...,\mu_n)\circ_1 \mu_1=...=(...((\nu\circ_n\mu_n)\circ_{n-1}\mu_{n-1})...)\circ_1\mu_1$$
In terms of trees, this new composition is the contraction of one edge: 
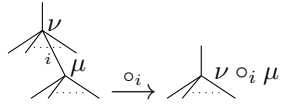
\begin{center}
\begin{tikzpicture} [scale=1.5]
\draw (0,0.25) -- (0,0) node [above right= -4pt and -2pt] {\small{$\nu$}};
\draw (0,0) -- (-0.3,-0.2);
\draw (0,0) -- (-0.15,-0.2);
\draw (0,0) -- (0.2,-0.4) node [above right= -4pt and -2pt] {\small{$\mu$}};
\path (0.125,-0.25) node [left=-2pt] {\tiny{$i$}};
\draw (-0.075,-0.15) -- (0.1875,-0.15) [dotted];
\draw (0,0) -- (0.3,-0.2);
\draw (0.2,-0.4) -- (-0.1,-0.6);
\draw (0.2,-0.4) -- (0.05,-0.6);
\draw (0.125,-0.55) -- (0.3875,-0.55) [dotted];
\draw (0.2,-0.4) -- (0.5,-0.6);
\end{tikzpicture} {$\stackrel{\circ_i}{\longrightarrow}$} \begin{tikzpicture}
\draw (0,0.4) -- (0,0) node [right=1pt] {\small{$\nu \circ_i \mu$}};
\draw (0,0) -- (-0.45,-0.3);
\draw (0,0) -- (-0.25,-0.3);
\draw [scale=0.75](-0.15,-0.3) -- (0.35,-0.3) [dotted];
\draw (0,0) -- (0.45,-0.3);
\end{tikzpicture}
\end{center}
\end{remark}

\begin{definition} A \textbf{homomorphism} of operads $\zeta: \mP \rightarrow \mQ$ is a series of linear maps $\zeta_n: \mP(n) \rightarrow \mQ(n)$ that maps the unit of $\mP$ to the unit of $\mQ$, and commutes with the composition and the action of $S_n$: 
\begin{enumerate}
  \item[i)] $\zeta_1(1)=1$
	\item[ii)] $\zeta_{n+m-1} (\nu \circ_i \mu) = \zeta_n (\nu) \circ_i \zeta_m (\mu)$
	\item[iii)] $\zeta_n (\sigma \circ \nu) = \sigma \circ \zeta_n (\nu)$
\end{enumerate}
\end{definition}

\begin{example} A natural example of an operad is $\mE_V$, the \textbf{operad of endomorphisms} of a finite-dimensional vector space $V$, with $\mE_V(n) := Hom(V^{\otimes n},V)$. 
The action of $S_n$ is the natural permutation of inputs: $(\sigma(\nu))(v_1,...,v_n):=\nu(v_{\sigma(1)},...,v_{\sigma(n)})$, with $\sigma\in S_n$, $\nu\in\mE_V(n)$, $v_i\in V$. 
The composition is the natural composition of maps:
$$(\nu(\mu_1,...,\mu_n))(v_1,...,v_{m_1+...+m_n}) =
\nu(\mu_1(v_1,...,v_{m_1}),...,\mu_n(v_{m_1+...+m_{n-1}+1},...,v_{m_1+...+m_n}))$$
with $\nu\in\mE_V(n)$, $\mu_i\in\mE_V(m_i)$, $v_i\in V$. 
The unit is, of course, the unit map $1\in\mE_V(1)=Hom(V,V)$ 
\end{example}

\begin{definition} A \textbf{representation} of an operad $\mP$ is a homomorphism $\zeta: \mP \rightarrow \mE_V$
\end{definition}

\section{Quadratic operads}

\subsection{Definitions}

From now on, we will assume that in any operad, $\mP(1)$ is the basic field $k$, with $char(k)=0$, and that every vector space mentioned is finite dimensional over this field. This means that composition with an element from $\mP (1)$ is simply scalar multiplication, and we can remove all vertices with only one input from our trees: 
$$ 
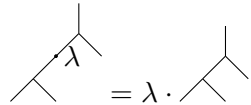
\begin{tikzpicture}
\draw (0,0.4) -- (0,0);
\draw (0,0) -- (-0.3,-0.3) node [right=-1pt] {$\lambda$};
\draw (0,0) -- (0.3,-0.3);
\draw (-0.6,-0.6) -- (-0.3,-0.3) node [shape=circle,draw,fill=black,inner sep=0pt,minimum size=1pt] {};
\draw (-0.6,-0.6) -- (-0.9,-0.9);
\draw (-0.6,-0.6) -- (-0.3,-0.9);
\end{tikzpicture} = \lambda \cdot \begin{tikzpicture}
\draw (0,0.4) -- (0,0);
\draw (0,0) -- (-0.3,-0.3);
\draw (0,0) -- (0.3,-0.3);
\draw (-0.3,-0.3) -- (-0.6,-0.6);
\draw (-0.3,-0.3) -- (0,-0.6);
\end{tikzpicture}$$
Where $\lambda \in k = \mP(1)$. We call these trees, where each vertex has at least two inputs, \textbf{reduced trees}.

Using reduced trees, we can create a free operad from basis elements: 
We start with a collection of $S_n$-modules (vector spaces with an action of $S_n$) $E=\{E(n) | n \geq 2\}$, with which we can ``decorate" a planar reduced tree: 
$$ 
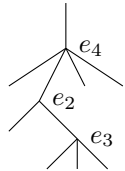
\begin{tikzpicture}
\draw (0,0.6) -- (0,0) node [right=1pt] {\small{$e_4$}};
\draw (0,0) -- (-0.75,-0.5);
\draw (0,0) -- (-0.35,-0.7) node [right=1pt] {\small{$e_2$}};
\draw (0,0) -- (0.25,-0.5);
\draw (0,0) -- (0.75,-0.5);
\draw (-0.35,-0.7) -- (-0.75,-1.1);
\draw (-0.35,-0.7) -- (0.15,-1.2) node [right=1pt] {\small{$e_3$}};
\draw (0.15,-1.2) -- (-0.25,-1.6);
\draw (0.15,-1.2) -- (0.15,-1.6);
\draw (0.15,-1.2) -- (0.55,-1.6);
\end{tikzpicture} $$ 
Here, of course, $e_2 \in E(2)$, $e_3 \in E(3)$ and $e_4 \in E(4)$. These decorated trees form the basis of our free operad: 

\begin{definition} The \textbf{free operad} over $E$, $Free(E)$, is the vector space spanned by all reduced (labelled) trees, where each (internal) vertex is decorated with an element of $E(n)$, $n$ being the number of inputs of the vertex. We compose two elements $\nu,\mu\in Free(E)$ by pasting the output of $\mu$ along the $i$th input of $\nu$. $S_n$ permutes the labels of the inputs of the tree (as always, if the inputs aren't labelled in a picture, we assume the canonical labeling). 
We demand that our elements are linear in each decoration, so
$$ \begin{tikzpicture}
\draw (0,0.6) -- (0,0) node [right=1pt] {\small{$e_4$}};
\draw (0,0) -- (-0.75,-0.5);
\draw (0,0) -- (-0.35,-0.7) node [right=1pt] {\small{$\lambda e_2 + \lambda' e'_2$}};
\draw (0,0) -- (0.25,-0.5);
\draw (0,0) -- (0.75,-0.5);
\draw (-0.35,-0.7) -- (-0.75,-1.1);
\draw (-0.35,-0.7) -- (0.15,-1.2) node [right=1pt] {\small{$e_3$}};
\draw (0.15,-1.2) -- (-0.25,-1.6);
\draw (0.15,-1.2) -- (0.15,-1.6);
\draw (0.15,-1.2) -- (0.55,-1.6);
\end{tikzpicture} = \lambda \begin{tikzpicture}
\draw (0,0.6) -- (0,0) node [right=1pt] {\small{$e_4$}};
\draw (0,0) -- (-0.75,-0.5);
\draw (0,0) -- (-0.35,-0.7) node [right=1pt] {\small{$e_2$}};
\draw (0,0) -- (0.25,-0.5);
\draw (0,0) -- (0.75,-0.5);
\draw (-0.35,-0.7) -- (-0.75,-1.1);
\draw (-0.35,-0.7) -- (0.15,-1.2) node [right=1pt] {\small{$e_3$}};
\draw (0.15,-1.2) -- (-0.25,-1.6);
\draw (0.15,-1.2) -- (0.15,-1.6);
\draw (0.15,-1.2) -- (0.55,-1.6);
\end{tikzpicture} + \lambda' \begin{tikzpicture}
\draw (0,0.6) -- (0,0) node [right=1pt] {\small{$e_4$}};
\draw (0,0) -- (-0.75,-0.5);
\draw (0,0) -- (-0.35,-0.7) node [right=1pt] {\small{$e'_2$}};
\draw (0,0) -- (0.25,-0.5);
\draw (0,0) -- (0.75,-0.5);
\draw (-0.35,-0.7) -- (-0.75,-1.1);
\draw (-0.35,-0.7) -- (0.15,-1.2) node [right=1pt] {\small{$e_3$}};
\draw (0.15,-1.2) -- (-0.25,-1.6);
\draw (0.15,-1.2) -- (0.15,-1.6);
\draw (0.15,-1.2) -- (0.55,-1.6);
\end{tikzpicture}$$
This means that we can view an element of $Free(E)$ as a tree with a tensor product of the decorating elements. 
\end{definition}

\begin{remark} If a certain $E(n)=0$, the free operad over $\{E(n)\}$ will not contain vertices with $n$ inputs. 
\end{remark}

The construction of a free operad allows us to introduce a class of operads, the quadratic operad, that is generated only by binary elements, with certain relations on them: we start with an $S_2$-module $E:=E(2)$, a vector space with an action of $S_2$. 
In terms of trees, basis elements of E look like a tree with one output and two inputs: 
$$ 
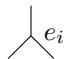
\begin{tikzpicture}
\draw (0,0.4) -- (0,0) node [right=1pt] {\small{$e_i$}};
\draw (0,0) -- (-0.3,-0.3);
\draw (0,0) -- (0.3,-0.3);
\end{tikzpicture} $$
The free operad over $E$, $Free(E)$, has binary trees (trees with exactly two inputs at each vertex) as basis elements (of course, each vertex should be labeled with an element from $E$): 
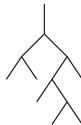
$$ \begin{tikzpicture}
\draw (0,0.4) -- (0,0) node [right=1pt] {\small{}};
\draw (0,0) -- (-0.3,-0.3);
\draw (0,0) -- (0.3,-0.3);
\draw (-0.3,-0.3) -- (-0.50,-0.6);
\draw (-0.3,-0.3) -- (-0.10,-0.6);
\draw (0.3,-0.3) -- (0.10,-0.6);
\draw (0.3,-0.3) -- (0.50,-0.6);
\draw (0.10,-0.6) -- (-0.10,-0.9);
\draw (0.10,-0.6) -- (0.30,-0.9);
\draw (0.30,-0.9) -- (0.50,-1.2);
\draw (0.30,-0.9) -- (0.10,-1.2);
\end{tikzpicture} $$
We then impose a set of relations on the constructed free operad: we quotient it out by an operadic ideal.

\begin{definition} An \textbf{ideal} of an operad $\mP$ is a collection $\mI$ of subspaces $\mI(n)\subset\mP(n)$ such that: 
\begin{enumerate}
	\item[i)] If $\nu\in\mI(n)$, then $\sigma (\nu) \in \mI(n)$, with $\sigma\in S_n$
	\item[ii)] If $\nu\in\mI(n)$ and $\mu\in\mP(m)$, then $\nu\circ_i\mu \in \mI(n+m-1)$ and $\mu\circ_j\nu \in \mI(n+m-1)$ for all possible $i$ and $j$.
\end{enumerate}
\end{definition}

As is the case with rings, we can generate ideals with elements from the operad $\mP$: if $\nu\in\mP(n)$, then $<\nu>$ consists of all elements that can be written as a composition that includes $\nu$. More generally, if $R$ is a subset of $\mP$, then $<R>$ is the ideal consisting of all elements that can be written as a composition that includes an element $\nu\in R$. 

Quotienting an operad out by an ideal is straightforward: given an operad $\mP$ and an ideal $\mI \subset \mP$, we define the quotient operad $\mP / \mI$ by $(\mP / \mI)(n):=\mP(n)/\mI(n)$. This is again an operad with a well-defined composition because of the above definition of an ideal. We can now define a quadratic operad: 

\begin{definition} Let $E$ be an $S_2$-module over $k$, and $R\subset Free(E)(3)$. Then we define $\mQ(E,R):=Free(E)/<R>$, the \textbf{quadratic operad} generated by $E$, with relations generated by $R$. If $\nu_1,...,\nu_n$ are the basis elements of $E$, and $\mu_1,...,\mu_m$ are the basis elements of $R$, we will also use the notation $\mQ(\nu_1,...,\nu_n; \mu_1,...,\mu_m) := \mQ(E,R)$. 
\end{definition}

\subsection{Examples}

We can now construct a number of basic examples of (quadratic) operads, that are related to different kinds of algebras through representations. We use the following definitions: 

\begin{definition} An \textbf{associative algebra} is a vector space $V$ with an extra binary operation, called the multiplication: $\cdot: V \times V \rightarrow V, (x,y)\mapsto x \cdot y$, that is linear in both terms, such that for all $x,y,z \in V$, $(x \cdot y) \cdot z = x \cdot (y \cdot z)$. 
We say that the multiplication is \textbf{commutative} if $x \cdot y = y \cdot x$ for all $x,y \in V$. 
\end{definition}

\begin{definition} A \textbf{Lie algebra} is a vector space $V$ with an extra binary operation, called the Lie bracket: $[.,.]: V\times V \rightarrow V, (x,y)\mapsto[x,y]$, that is linear in both terms, and for all $x,y,z \in V$ satisfies: 
\begin{enumerate}
	\item[i)] $[x,y]=-[y,x]$ (anticommutativity)
	\item[ii)] $[x,[y,z]]+[y,[z,x]]+[z,[x,y]]=0$ (the Jacobi identity)
\end{enumerate}
\end{definition}

\begin{definition} A \textbf{Poisson algebra} is a vector space $V$ with a commutative multiplication and a Lie bracket that satisfy the following distributive law: 
$[x,y \cdot z] = y \cdot [x,z] + [x,y] \cdot z$. 
\end{definition}

Note that $\cdot,[.,.] \in Hom(V \otimes V, V) = \mE_V(2)$, because the multiplication and the Lie bracket are both bilinear. 

\begin{example} Our first example is the \textbf{associative operad}. It is generated by \begin{tikzpicture}
\draw (0,0.25) -- (0,0);
\draw (0,0) -- (-0.2,-0.2);
\draw (0,0) -- (0.2,-0.2);
\end{tikzpicture} and \begin{tikzpicture}
\draw (0,0.25) -- (0,0);
\draw (0,0) -- (-0.06,-0.06) -- (-0.02,-0.1);
\draw (0,0) -- (0.06,-0.06) -- (-0.08,-0.2);
\draw (0.02,-0.14) -- (0.08,-0.2);
\end{tikzpicture}, with $(12)$ sending the one to the other (so \begin{tikzpicture}
\draw (0,0.25) -- (0,0);
\draw (0,0) -- (-0.2,-0.2) node [below=-1pt] {\tiny{2}};
\draw (0,0) -- (0.2,-0.2) node [below=-1pt] {\tiny{1}};
\end{tikzpicture} $=$ \begin{tikzpicture}
\draw (0,0.25) -- (0,0);
\draw (0,0) -- (-0.06,-0.06) -- (-0.02,-0.1);
\draw (0,0) -- (0.06,-0.06) -- (-0.08,-0.2) node [below=-1pt] {\tiny{1}};
\draw (0.02,-0.14) -- (0.08,-0.2) node [below=-1pt] {\tiny{2}};
\end{tikzpicture}). Because of the equivariance, you can write each tree with only \begin{tikzpicture} 
\draw (0,0.25) -- (0,0);
\draw (0,0) -- (-0.2,-0.2);
\draw (0,0) -- (0.2,-0.2);
\end{tikzpicture}'s: for example, 
$$ \begin{tikzpicture}
\draw (0,0.25) -- (0,0);
\draw (0,0) -- (-0.06,-0.06) -- (-0.02,-0.1);
\draw (0,0) -- (0.06,-0.06) -- (-0.08,-0.2) node [left=-2pt] {\tiny{1}};
\draw (0.02,-0.14) -- (0.08,-0.2);
\draw (0.08,-0.2) -- (-0.12,-0.4) node [below=-1pt] {\tiny{2}};
\draw (0.08,-0.2) -- (0.28,-0.4) node [below=-1pt] {\tiny{3}};
\end{tikzpicture} = \begin{tikzpicture}
\draw (0,0.25) -- (0,0);
\draw (0,0) -- (-0.2,-0.2);
\draw (0,0) -- (0.2,-0.2) node [below=-1pt] {\tiny{1}};
\draw (-0.2,-0.2) -- (-0.4,-0.4) node [below=-1pt] {\tiny{2}};
\draw (-0.2,-0.2) -- (0,-0.4) node [below=-1pt] {\tiny{3}};
\end{tikzpicture} $$
The associative relation is 
$$ \begin{tikzpicture}
\draw (0,0.25) -- (0,0);
\draw (0,0) -- (-0.2,-0.2);
\draw (0,0) -- (0.2,-0.2) node [below right=-3pt and -8pt] {\tiny{$\sigma(3)$}};
\draw (-0.2,-0.2) -- (-0.4,-0.4) node [below left=-2pt and -8pt] {\tiny{$\sigma(1)$}};
\draw (-0.2,-0.2) -- (0,-0.4) node [below right=-2pt and -8pt] {\tiny{$\sigma(2)$}};
\end{tikzpicture} = \begin{tikzpicture}
\draw (0,0.25) -- (0,0);
\draw (0,0) -- (-0.2,-0.2) node [below left=-3pt and -8pt] {\tiny{$\sigma(1)$}};
\draw (0,0) -- (0.2,-0.2);
\draw (0.2,-0.2) -- (0,-0.4) node [below left=-2pt and -8pt] {\tiny{$\sigma(2)$}};
\draw (0.2,-0.2) -- (0.4,-0.4) node [below right=-2pt and -8pt] {\tiny{$\sigma(3)$}};
\end{tikzpicture} $$
with all (six) possible permutations of the inputs, so here $R = \{ \begin{tikzpicture}
\draw (0,0.25) -- (0,0);
\draw (0,0) -- (-0.2,-0.2);
\draw (0,0) -- (0.2,-0.2) node [below right=-3pt and -8pt] {\tiny{$\sigma(3)$}};
\draw (-0.2,-0.2) -- (-0.4,-0.4) node [below left=-2pt and -8pt] {\tiny{$\sigma(1)$}};
\draw (-0.2,-0.2) -- (0,-0.4) node [below right=-2pt and -8pt] {\tiny{$\sigma(2)$}};
\end{tikzpicture} - \begin{tikzpicture}
\draw (0,0.25) -- (0,0);
\draw (0,0) -- (-0.2,-0.2) node [below left=-3pt and -8pt] {\tiny{$\sigma(1)$}};
\draw (0,0) -- (0.2,-0.2);
\draw (0.2,-0.2) -- (0,-0.4) node [below left=-2pt and -8pt] {\tiny{$\sigma(2)$}};
\draw (0.2,-0.2) -- (0.4,-0.4) node [below right=-2pt and -8pt] {\tiny{$\sigma(3)$}};
\end{tikzpicture} | \sigma \in S_3\}$. This relation implies that any tree that contains a \begin{tikzpicture}
\draw (0,0.25) -- (0,0);
\draw (0,0) -- (-0.2,-0.2);
\draw (0,0) -- (0.2,-0.2);
\draw (-0.2,-0.2) -- (-0.4,-0.4);
\draw (-0.2,-0.2) -- (0,-0.4);
\end{tikzpicture} can be replaced by a tree that has a \begin{tikzpicture}
\draw (0,0.25) -- (0,0);
\draw (0,0) -- (-0.2,-0.2);
\draw (0,0) -- (0.2,-0.2);
\draw (0.2,-0.2) -- (0,-0.4);
\draw (0.2,-0.2) -- (0.4,-0.4);
\end{tikzpicture}. So for example 
$$ \begin{tikzpicture}
\draw (0,0.25) -- (0,0);
\draw (0,0.15) -- (0,0) [very thick];
\draw (0,0) -- (-0.25,-0.2) -- (-0.37,-0.32) [very thick];
\draw (0,0) -- (0.15,-0.12) [very thick];
\draw (0,0) -- (0.25,-0.2);
\draw (-0.25,-0.2) -- (-0.16,-0.32) [very thick];
\draw (-0.25,-0.2) -- (-0.1,-0.4);
\draw (-0.25,-0.2) -- (-0.45,-0.4);
\draw (0.25,-0.2) -- (0.1,-0.4);
\draw (0.25,-0.2) -- (0.45,-0.4);
\end{tikzpicture} = \begin{tikzpicture}
\draw (0,0.25) -- (0,0);
\draw (0,0.15) -- (0,0) [very thick];
\draw (0,0) -- (-0.12,-0.12) [very thick];
\draw (0,0) -- (-0.2,-0.2);
\draw (0,0) -- (0.2,-0.2) [very thick];
\draw (0.2,-0.2) -- (0.08,-0.32) [very thick];
\draw (0.2,-0.2) -- (0.32,-0.32) [very thick];
\draw (0.2,-0.2) -- (0,-0.4);
\draw (0.2,-0.2) -- (0.4,-0.4);
\draw (0.4,-0.4) -- (0.2,-0.6);
\draw (0.4,-0.4) -- (0.6,-0.6);
\end{tikzpicture} $$
Here, we changed the bolded parts of the tree according to our rule, and we left the rest of the tree intact. Please note that the order of the inputs isn't changed, so for example $$ \begin{tikzpicture}
\draw (0,0.25) -- (0,0);
\draw (0,0) -- (-0.25,-0.2);
\draw (0,0) -- (0.25,-0.2);
\draw (-0.25,-0.2) -- (-0.1,-0.4) node [below=-1pt] {\tiny{4}};
\draw (-0.25,-0.2) -- (-0.45,-0.4) node [below=-1pt] {\tiny{2}};
\draw (0.25,-0.2) -- (0.1,-0.4) node [below=-1pt] {\tiny{1}};
\draw (0.25,-0.2) -- (0.45,-0.4) node [below=-1pt] {\tiny{3}};
\end{tikzpicture} = \begin{tikzpicture}
\draw (0,0.25) -- (0,0);
\draw (0,0) -- (-0.2,-0.2) node [below=-1pt] {\tiny{2}};
\draw (0,0) -- (0.2,-0.2);
\draw (0.2,-0.2) -- (0,-0.4) node [below=-1pt] {\tiny{4}};
\draw (0.2,-0.2) -- (0.4,-0.4);
\draw (0.4,-0.4) -- (0.2,-0.6) node [below=-1pt] {\tiny{1}};
\draw (0.4,-0.4) -- (0.6,-0.6) node [below=-1pt] {\tiny{3}};
\end{tikzpicture} $$ 
Our relation means that we can write every element of our operad as a combination of trees that look like 
$$ \begin{tikzpicture}
\draw (0,0.25) -- (0,0);
\draw (0,0) -- (-0.2,-0.2) node [below left=-3pt and -5pt] {\tiny{$\sigma(1)$}};
\draw (0,0) -- (0.2,-0.2);
\draw (0.2,-0.2) -- (0.0,-0.4) node [below left=-3pt and -5pt] {\tiny{$\sigma(2)$}};
\draw (0.2,-0.2) -- (0.4,-0.4);
\draw (0.4,-0.4) -- (0.6,-0.6) [dotted];
\draw (0.7,-0.7) -- (0.5,-0.9) node [below left=-3pt and -5pt] {\tiny{$\sigma(n$-$2)$}};
\draw (0.6,-0.6) -- (0.9,-0.9);
\draw (0.9,-0.9) -- (0.7,-1.1) node [below left=-3pt and -5pt] {\tiny{$\sigma(n$-$1)$}};
\draw (0.9,-0.9) -- (1.1,-1.1) node [below right=-3pt and -8pt] {\tiny{$\sigma(n)$}};
\end{tikzpicture} $$
So $\mAss:= \mQ(\begin{tikzpicture}
\draw (0,0.25) -- (0,0);
\draw (0,0) -- (-0.2,-0.2);
\draw (0,0) -- (0.2,-0.2);
\end{tikzpicture}, \begin{tikzpicture}
\draw (0,0.25) -- (0,0);
\draw (0,0) -- (-0.06,-0.06) -- (-0.02,-0.1);
\draw (0,0) -- (0.06,-0.06) -- (-0.08,-0.2);
\draw (0.02,-0.14) -- (0.08,-0.2);
\end{tikzpicture};
\begin{tikzpicture}
\draw (0,0.25) -- (0,0);
\draw (0,0) -- (-0.2,-0.2);
\draw (0,0) -- (0.2,-0.2);
\draw (-0.2,-0.2) -- (-0.4,-0.4);
\draw (-0.2,-0.2) -- (0,-0.4);
\end{tikzpicture} - \begin{tikzpicture}
\draw (0,0.25) -- (0,0);
\draw (0,0) -- (-0.2,-0.2);
\draw (0,0) -- (0.2,-0.2);
\draw (0.2,-0.2) -- (0,-0.4);
\draw (0.2,-0.2) -- (0.4,-0.4);
\end{tikzpicture}$) (we routinely forget to label the inputs if they are in the same order in each tree and the relation is satisfied for all possible permutations of inputs) is spanned by these trees. In total we have $n!$ of these trees in $\mAss(n)$, all of them linear independent because there is no way to switch inputs, so $dim(\mAss(n))=n!$. Because the order of inputs is all that matters, we say that $\mAss(n)$ is spanned by \begin{tikzpicture}
\draw (0,0.25) -- (0,-0.00);
\draw (0,0) -- (-0.3,-0.2);
\draw (0,0) -- (-0.15,-0.2);
\draw (-0.10,-0.2) -- (0.25,-0.2) [dotted];
\draw (0,0) -- (0.3,-0.2);
\end{tikzpicture}, which we understand to be of the form shown above.

Any representation of $\mAss$ is defined by the image of \tree. Of course, $\zeta (\tree) \in Hom(V \otimes V, V)$, so our representation defines a (bilinear) multiplication $\cdot:= \zeta(\tree)$ on $V$. Our relation tells us that $(x \cdot y) \cdot z = x \cdot (y \cdot z)$, which is exactly our associative relation. This means that a representation of $\mAss$ is exactly an associative algebra structure on $V$.
\end{example}

\begin{example} The \textbf{commutative operad} is generated by \comtree, with $(12) \circ \comtree = \comtree$. We again have the associative relation, 
$$ \begin{tikzpicture}
\draw (0,0.25) -- (0,0);
\draw (0,0) -- (-0.2,-0.2);
\draw (0,0) -- (0.2,-0.2);
\draw (-0.2,-0.2) -- (-0.4,-0.4);
\draw (-0.2,-0.2) -- (0,-0.4);
\draw (0,0) circle (1pt) [fill=white];
\draw (-0.2,-0.2) circle (1pt) [fill=white];
\end{tikzpicture} = \begin{tikzpicture}
\draw (0,0.25) -- (0,0);
\draw (0,0) -- (-0.2,-0.2);
\draw (0,0) -- (0.2,-0.2);
\draw (0.2,-0.2) -- (0,-0.4);
\draw (0.2,-0.2) -- (0.4,-0.4);
\draw (0,0) circle (1pt) [fill=white];
\draw (0.2,-0.2) circle (1pt) [fill=white];
\end{tikzpicture} $$
So $\mCom := \mQ(\comtree;\begin{tikzpicture}
\draw (0,0.25) -- (0,0);
\draw (0,0) -- (-0.2,-0.2);
\draw (0,0) -- (0.2,-0.2);
\draw (-0.2,-0.2) -- (-0.4,-0.4);
\draw (-0.2,-0.2) -- (0,-0.4);
\draw (0,0) circle (1pt) [fill=white];
\draw (-0.2,-0.2) circle (1pt) [fill=white];
\end{tikzpicture} - \begin{tikzpicture}
\draw (0,0.25) -- (0,0);
\draw (0,0) -- (-0.2,-0.2);
\draw (0,0) -- (0.2,-0.2);
\draw (0.2,-0.2) -- (0,-0.4);
\draw (0.2,-0.2) -- (0.4,-0.4);
\draw (0,0) circle (1pt) [fill=white];
\draw (0.2,-0.2) circle (1pt) [fill=white];
\end{tikzpicture})$. We know that we can use the relation to get our trees in a standard form, but we can now go one step further: because the action of $S_2$ is trivial, we can freely switch inputs of any vertex, so for example
$$ \begin{tikzpicture}
\draw (0,0.25) -- (0,0);
\draw (0,0) -- (-0.2,-0.2) node [below=-1pt] {\tiny{1}};
\draw (0,0) -- (0,-0.2) node [below=-1pt] {\tiny{2}};
\draw (0,0) -- (0.2,-0.2) node [below=-1pt] {\tiny{3}};
\draw (0,0) circle (1pt) [fill=white];
\end{tikzpicture} = \begin{tikzpicture}
\draw (0,0.25) -- (0,0);
\draw (0,0) -- (-0.2,-0.2) node [below=-1pt] {\tiny{1}};
\draw (0,0) -- (0.2,-0.2);
\draw (0.2,-0.2) -- (0,-0.4) node [below=-1pt] {\tiny{2}};
\draw (0.2,-0.2) -- (0.4,-0.4) node [below=-1pt] {\tiny{3}};
\draw (0,0) circle (1pt) [fill=white];
\draw (0.2,-0.2) circle (1pt) [fill=white];
\end{tikzpicture} = \begin{tikzpicture}
\draw (0,0.25) -- (0,0);
\draw (0,0) -- (-0.2,-0.2);
\draw (0,0) -- (0.2,-0.2) node [below=-1pt] {\tiny{3}};
\draw (-0.2,-0.2) -- (-0.4,-0.4) node [below=-1pt] {\tiny{1}};
\draw (-0.2,-0.2) -- (0,-0.4) node [below=-1pt] {\tiny{2}};
\draw (0,0) circle (1pt) [fill=white];
\draw (-0.2,-0.2) circle (1pt) [fill=white];
\end{tikzpicture} = \begin{tikzpicture}
\draw (0,0.25) -- (0,0);
\draw (0,0) -- (-0.2,-0.2);
\draw (0,0) -- (0.2,-0.2) node [below=-1pt] {\tiny{3}};
\draw (-0.2,-0.2) -- (-0.4,-0.4) node [below=-1pt] {\tiny{2}};
\draw (-0.2,-0.2) -- (0,-0.4) node [below=-1pt] {\tiny{1}};
\draw (0,0) circle (1pt) [fill=white];
\draw (-0.2,-0.2) circle (1pt) [fill=white];
\end{tikzpicture} = \begin{tikzpicture}
\draw (0,0.25) -- (0,0);
\draw (0,0) -- (-0.2,-0.2) node [below=-1pt] {\tiny{2}};
\draw (0,0) -- (0.2,-0.2);
\draw (0.2,-0.2) -- (0,-0.4) node [below=-1pt] {\tiny{1}};
\draw (0.2,-0.2) -- (0.4,-0.4) node [below=-1pt] {\tiny{3}};
\draw (0,0) circle (1pt) [fill=white];
\draw (0.2,-0.2) circle (1pt) [fill=white];
\end{tikzpicture} = \begin{tikzpicture}
\draw (0,0.25) -- (0,0);
\draw (0,0) -- (-0.2,-0.2) node [below=-1pt] {\tiny{2}};
\draw (0,0) -- (0,-0.2) node [below=-1pt] {\tiny{1}};
\draw (0,0) -- (0.2,-0.2) node [below=-1pt] {\tiny{3}};
\draw (0,0) circle (1pt) [fill=white];
\end{tikzpicture}$$
This means that any tree in our operad can be rewritten to be of the form \begin{tikzpicture}
\draw (0,0.25) -- (0,-0.00);
\draw (0,0) -- (-0.3,-0.2);
\draw (0,0) -- (-0.15,-0.2);
\draw (-0.10,-0.2) -- (0.25,-0.2) [dotted];
\draw (0,0) -- (0.3,-0.2);
\draw (0,0) circle (1pt) [fill=white];
\end{tikzpicture} $=$ \begin{tikzpicture}
\draw (0,0.25) -- (0,-0.00);
\draw (0,0) -- (-0.3,-0.2) node [below=-1pt] {\tiny{1}};
\draw (0,0) -- (-0.15,-0.2) node [below=-1pt] {\tiny{2}};
\draw (-0.10,-0.2) -- (0.25,-0.2) [dotted];
\draw (0,0) -- (0.3,-0.2) node [below=0pt] {\tiny{$n$}};
\draw (0,0) circle (1pt) [fill=white];
\end{tikzpicture}. So $\mCom(n)$ one-dimensional for each $n$, and it's spanned by \begin{tikzpicture}
\draw (0,0.25) -- (0,-0.00);
\draw (0,0) -- (-0.3,-0.2);
\draw (0,0) -- (-0.15,-0.2);
\draw (-0.10,-0.2) -- (0.25,-0.2) [dotted];
\draw (0,0) -- (0.3,-0.2);
\draw (0,0) circle (1pt) [fill=white];
\end{tikzpicture}. We also see that the action of $S_n$ on $\mCom(n)$ is trivial for each $n$. 

Beacause of the associative relation, representations of $\mCom$ are again associative algebra's, and the trivial action of $S_2$ exactly means that $x \cdot y = y \cdot x$ (with $\cdot:= \zeta (\comtree)$). So representations of the commutative operad are exactly commutative (associative) algebras. 
\end{example}

\begin{example} Our next example is the \textbf{Lie operad}, $\mLie:=\mQ(\tree;\begin{tikzpicture}
\draw (0,0.25) -- (0,0);
\draw (0,0) -- (-0.2,-0.2) node [below=-1pt] {\tiny{1}};
\draw (0,0) -- (0.2,-0.2);
\draw (0.2,-0.2) -- (0,-0.4) node [below=-1pt] {\tiny{2}};
\draw (0.2,-0.2) -- (0.4,-0.4) node [below=-1pt] {\tiny{3}};
\end{tikzpicture} + \begin{tikzpicture}
\draw (0,0.25) -- (0,0);
\draw (0,0) -- (-0.2,-0.2) node [below=-1pt] {\tiny{2}};
\draw (0,0) -- (0.2,-0.2);
\draw (0.2,-0.2) -- (0,-0.4) node [below=-1pt] {\tiny{3}};
\draw (0.2,-0.2) -- (0.4,-0.4) node [below=-1pt] {\tiny{1}};
\end{tikzpicture} + \begin{tikzpicture}
\draw (0,0.25) -- (0,0);
\draw (0,0) -- (-0.2,-0.2) node [below=-1pt] {\tiny{3}};
\draw (0,0) -- (0.2,-0.2);
\draw (0.2,-0.2) -- (0,-0.4) node [below=-1pt] {\tiny{1}};
\draw (0.2,-0.2) -- (0.4,-0.4) node [below=-1pt] {\tiny{2}};
\end{tikzpicture})$, with $(12)\circ \tree = -\tree$. Using this, our relation can be rewritten to $$\begin{tikzpicture}
\draw (0,0.25) -- (0,0);
\draw (0,0) -- (0.2,-0.2) node [below=-1pt] {\tiny{1}};
\draw (0,0) -- (-0.2,-0.2);
\draw (-0.2,-0.2) -- (-0.4,-0.4) node [below=-1pt] {\tiny{2}};
\draw (-0.2,-0.2) -- (0,-0.4) node [below=-1pt] {\tiny{3}};
\end{tikzpicture} = \begin{tikzpicture}
\draw (0,0.25) -- (0,0);
\draw (0,0) -- (-0.2,-0.2) node [below=-1pt] {\tiny{2}};
\draw (0,0) -- (0.2,-0.2);
\draw (0.2,-0.2) -- (0,-0.4) node [below=-1pt] {\tiny{3}};
\draw (0.2,-0.2) -- (0.4,-0.4) node [below=-1pt] {\tiny{1}};
\end{tikzpicture} + \begin{tikzpicture}
\draw (0,0.25) -- (0,0);
\draw (0,0) -- (-0.2,-0.2) node [below=-1pt] {\tiny{3}};
\draw (0,0) -- (0.2,-0.2);
\draw (0.2,-0.2) -- (0,-0.4) node [below=-1pt] {\tiny{1}};
\draw (0.2,-0.2) -- (0.4,-0.4) node [below=-1pt] {\tiny{2}};
\end{tikzpicture}$$
so again, all our trees are a sum of trees looking like
$$ 
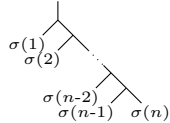
\begin{tikzpicture}
\draw (0,0.25) -- (0,0);
\draw (0,0) -- (-0.2,-0.2) node [below left=-3pt and -5pt] {\tiny{$\sigma(1)$}};
\draw (0,0) -- (0.2,-0.2);
\draw (0.2,-0.2) -- (0.0,-0.4) node [below left=-3pt and -5pt] {\tiny{$\sigma(2)$}};
\draw (0.2,-0.2) -- (0.4,-0.4);
\draw (0.4,-0.4) -- (0.6,-0.6) [dotted];
\draw (0.7,-0.7) -- (0.5,-0.9) node [below left=-3pt and -5pt] {\tiny{$\sigma(n$-$2)$}};
\draw (0.6,-0.6) -- (0.9,-0.9);
\draw (0.9,-0.9) -- (0.7,-1.1) node [below left=-3pt and -5pt] {\tiny{$\sigma(n$-$1)$}};
\draw (0.9,-0.9) -- (1.1,-1.1) node [below right=-3pt and -8pt] {\tiny{$\sigma(n)$}};
\end{tikzpicture} $$
Another way to write our relation is 
$$ \begin{tikzpicture}
\draw (0,0.25) -- (0,0);
\draw (0,0) -- (-0.2,-0.2) node [below=-1pt] {\tiny{3}};
\draw (0,0) -- (0.2,-0.2);
\draw (0.2,-0.2) -- (0,-0.4) node [below=-1pt] {\tiny{1}};
\draw (0.2,-0.2) -- (0.4,-0.4) node [below=-1pt] {\tiny{2}};
\end{tikzpicture} = \begin{tikzpicture}
\draw (0,0.25) -- (0,0);
\draw (0,0) -- (-0.2,-0.2) node [below=-1pt] {\tiny{1}};
\draw (0,0) -- (0.2,-0.2);
\draw (0.2,-0.2) -- (0,-0.4) node [below=-1pt] {\tiny{3}};
\draw (0.2,-0.2) -- (0.4,-0.4) node [below=-1pt] {\tiny{2}};
\end{tikzpicture} + \begin{tikzpicture}
\draw (0,0.25) -- (0,0);
\draw (0,0) -- (-0.2,-0.2) node [below=-1pt] {\tiny{2}};
\draw (0,0) -- (0.2,-0.2);
\draw (0.2,-0.2) -- (0,-0.4) node [below=-1pt] {\tiny{1}};
\draw (0.2,-0.2) -- (0.4,-0.4) node [below=-1pt] {\tiny{3}};
\end{tikzpicture} $$
This means that we can move one label (in our relation above it's 3, typically it's $n$) to the lowest spot in our tree. All such trees are linear independent, so $\mLie (n)$ is spanned by the $(n-1)!$ trees like 
$$ 
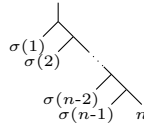
\begin{tikzpicture}
\draw (0,0.25) -- (0,0);
\draw (0,0) -- (-0.2,-0.2) node [below left=-3pt and -5pt] {\tiny{$\sigma(1)$}};
\draw (0,0) -- (0.2,-0.2);
\draw (0.2,-0.2) -- (0.0,-0.4) node [below left=-3pt and -5pt] {\tiny{$\sigma(2)$}};
\draw (0.2,-0.2) -- (0.4,-0.4);
\draw (0.4,-0.4) -- (0.6,-0.6) [dotted];
\draw (0.7,-0.7) -- (0.5,-0.9) node [below left=-3pt and -5pt] {\tiny{$\sigma(n$-$2)$}};
\draw (0.6,-0.6) -- (0.9,-0.9);
\draw (0.9,-0.9) -- (0.7,-1.1) node [below left=-3pt and -5pt] {\tiny{$\sigma(n$-$1)$}};
\draw (0.9,-0.9) -- (1.1,-1.1) node [below=-1pt] {\tiny{$n$}};
\end{tikzpicture} $$

A representation of $\mLie$ will be a Lie algebra: if we define $[.,.] := \zeta(\tree)$ then $(12) \circ \tree = -\tree$ means that $[x,y] = - [y,x]$ and our relation exactly translates into the Jacobi identity. 
\end{example}

\begin{example} The last example we will review is the \textbf{Poisson operad}. It is generated by a commutative element ($\comtree = (12) \circ \comtree$, \begin{tikzpicture}
\draw (0,0.25) -- (0,0);
\draw (0,0) -- (-0.2,-0.2);
\draw (0,0) -- (0.2,-0.2);
\draw (-0.2,-0.2) -- (-0.4,-0.4);
\draw (-0.2,-0.2) -- (0,-0.4);
\draw (0,0) circle (1pt) [fill=white];
\draw (-0.2,-0.2) circle (1pt) [fill=white];
\end{tikzpicture} = \begin{tikzpicture}
\draw (0,0.25) -- (0,0);
\draw (0,0) -- (-0.2,-0.2);
\draw (0,0) -- (0.2,-0.2);
\draw (0.2,-0.2) -- (0,-0.4);
\draw (0.2,-0.2) -- (0.4,-0.4);
\draw (0,0) circle (1pt) [fill=white];
\draw (0.2,-0.2) circle (1pt) [fill=white];
\end{tikzpicture}) and a Lie element ($\tree = - (12) \circ \tree$, \begin{tikzpicture}
\draw (0,0.25) -- (0,0);
\draw (0,0) -- (-0.2,-0.2) node [below=-1pt] {\tiny{1}};
\draw (0,0) -- (0.2,-0.2);
\draw (0.2,-0.2) -- (0,-0.4) node [below=-1pt] {\tiny{2}};
\draw (0.2,-0.2) -- (0.4,-0.4) node [below=-1pt] {\tiny{3}};
\end{tikzpicture} + \begin{tikzpicture}
\draw (0,0.25) -- (0,0);
\draw (0,0) -- (-0.2,-0.2) node [below=-1pt] {\tiny{2}};
\draw (0,0) -- (0.2,-0.2);
\draw (0.2,-0.2) -- (0,-0.4) node [below=-1pt] {\tiny{3}};
\draw (0.2,-0.2) -- (0.4,-0.4) node [below=-1pt] {\tiny{1}};
\end{tikzpicture} + \begin{tikzpicture}
\draw (0,0.25) -- (0,0);
\draw (0,0) -- (-0.2,-0.2) node [below=-1pt] {\tiny{3}};
\draw (0,0) -- (0.2,-0.2);
\draw (0.2,-0.2) -- (0,-0.4) node [below=-1pt] {\tiny{1}};
\draw (0.2,-0.2) -- (0.4,-0.4) node [below=-1pt] {\tiny{2}};
\end{tikzpicture}), with the following distributivity relation:
$$ \begin{tikzpicture}
\draw (0,0.25) -- (0,0);
\draw (0,0) -- (-0.2,-0.2) node [below=-1pt] {\tiny{1}};
\draw (0,0) -- (0.2,-0.2);
\draw (0.2,-0.2) -- (0,-0.4) node [below=-1pt] {\tiny{2}};
\draw (0.2,-0.2) -- (0.4,-0.4) node [below=-1pt] {\tiny{3}};
\draw (0.2,-0.2) circle (1pt) [fill=white];
\end{tikzpicture} = \begin{tikzpicture}
\draw (0,0.25) -- (0,0);
\draw (0,0) -- (-0.2,-0.2) node [below=-1pt] {\tiny{2}};
\draw (0,0) -- (0.2,-0.2);
\draw (0.2,-0.2) -- (0,-0.4) node [below=-1pt] {\tiny{1}};
\draw (0.2,-0.2) -- (0.4,-0.4) node [below=-1pt] {\tiny{3}};
\draw (0,0) circle (1pt) [fill=white];
\end{tikzpicture} + \begin{tikzpicture}
\draw (0,0.25) -- (0,0);
\draw (0,0) -- (-0.2,-0.2);
\draw (0,0) -- (0.2,-0.2) node [below=-1pt] {\tiny{3}};
\draw (-0.2,-0.2) -- (-0.4,-0.4) node [below=-1pt] {\tiny{1}};
\draw (-0.2,-0.2) -- (0,-0.4) node [below=-1pt] {\tiny{2}};
\draw (0,0) circle (1pt) [fill=white];
\end{tikzpicture} $$ 
So, $\mPoiss := \mQ ( \comtree,\tree;
\begin{tikzpicture}
\draw (0,0.25) -- (0,0);
\draw (0,0) -- (-0.2,-0.2);
\draw (0,0) -- (0.2,-0.2);
\draw (-0.2,-0.2) -- (-0.4,-0.4);
\draw (-0.2,-0.2) -- (0,-0.4);
\draw (0,0) circle (1pt) [fill=white];
\draw (-0.2,-0.2) circle (1pt) [fill=white];
\end{tikzpicture} - \begin{tikzpicture}
\draw (0,0.25) -- (0,0);
\draw (0,0) -- (-0.2,-0.2);
\draw (0,0) -- (0.2,-0.2);
\draw (0.2,-0.2) -- (0,-0.4);
\draw (0.2,-0.2) -- (0.4,-0.4);
\draw (0,0) circle (1pt) [fill=white];
\draw (0.2,-0.2) circle (1pt) [fill=white];
\end{tikzpicture}, \begin{tikzpicture}
\draw (0,0.25) -- (0,0);
\draw (0,0) -- (-0.2,-0.2) node [below=-1pt] {\tiny{1}};
\draw (0,0) -- (0.2,-0.2);
\draw (0.2,-0.2) -- (0,-0.4) node [below=-1pt] {\tiny{2}};
\draw (0.2,-0.2) -- (0.4,-0.4) node [below=-1pt] {\tiny{3}};
\end{tikzpicture} + \begin{tikzpicture}
\draw (0,0.25) -- (0,0);
\draw (0,0) -- (-0.2,-0.2) node [below=-1pt] {\tiny{2}};
\draw (0,0) -- (0.2,-0.2);
\draw (0.2,-0.2) -- (0,-0.4) node [below=-1pt] {\tiny{3}};
\draw (0.2,-0.2) -- (0.4,-0.4) node [below=-1pt] {\tiny{1}};
\end{tikzpicture} + \begin{tikzpicture}
\draw (0,0.25) -- (0,0);
\draw (0,0) -- (-0.2,-0.2) node [below=-1pt] {\tiny{3}};
\draw (0,0) -- (0.2,-0.2);
\draw (0.2,-0.2) -- (0,-0.4) node [below=-1pt] {\tiny{1}};
\draw (0.2,-0.2) -- (0.4,-0.4) node [below=-1pt] {\tiny{2}};
\end{tikzpicture}, \begin{tikzpicture}
\draw (0,0.25) -- (0,0);
\draw (0,0) -- (-0.2,-0.2) node [below=-1pt] {\tiny{1}};
\draw (0,0) -- (0.2,-0.2);
\draw (0.2,-0.2) -- (0,-0.4) node [below=-1pt] {\tiny{2}};
\draw (0.2,-0.2) -- (0.4,-0.4) node [below=-1pt] {\tiny{3}};
\draw (0.2,-0.2) circle (1pt) [fill=white];
\end{tikzpicture} - \begin{tikzpicture}
\draw (0,0.25) -- (0,0);
\draw (0,0) -- (-0.2,-0.2) node [below=-1pt] {\tiny{2}};
\draw (0,0) -- (0.2,-0.2);
\draw (0.2,-0.2) -- (0,-0.4) node [below=-1pt] {\tiny{1}};
\draw (0.2,-0.2) -- (0.4,-0.4) node [below=-1pt] {\tiny{3}};
\draw (0,0) circle (1pt) [fill=white];
\end{tikzpicture} - \begin{tikzpicture}
\draw (0,0.25) -- (0,0);
\draw (0,0) -- (-0.2,-0.2);
\draw (0,0) -- (0.2,-0.2) node [below=-1pt] {\tiny{3}};
\draw (-0.2,-0.2) -- (-0.4,-0.4) node [below=-1pt] {\tiny{1}};
\draw (-0.2,-0.2) -- (0,-0.4) node [below=-1pt] {\tiny{2}};
\draw (0,0) circle (1pt) [fill=white];
\end{tikzpicture})$. The distributive relation allows us to move all commutative elements upwards, so $\mPoiss (n)$ is spanned by elements that look like
$$ 
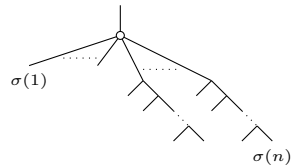
\begin{tikzpicture}
\draw (0,0.4) -- (0,0);
\draw (0,0) -- (-1.20,-0.40) node [below=-1pt] {\tiny{$\sigma(1)$}};
\draw[scale=0.75] (-1.00,-0.40) -- (-0.40,-0.40) [dotted];
\draw (0,0) -- (-0.30,-0.40)
;
\draw (0,0) -- (0.30,-0.60);
\draw (0.3,-0.6) -- (0.1,-0.8);
\draw (0.3,-0.6) -- (0.5,-0.8);
\draw (0.5,-0.8) -- (0.3,-1.0);
\draw (0.5,-0.8) -- (0.7,-1.0);
\draw (0.7,-1.0) -- (0.9,-1.2) [dotted];
\draw (0.9,-1.2) -- (0.7,-1.4);
\draw (0.9,-1.2) -- (1.1,-1.4);
\draw (0,0) -- (1.20,-0.60);
\draw (1.2,-0.6) -- (1.0,-0.8);
\draw (1.2,-0.6) -- (1.4,-0.8);
\draw (1.4,-0.8) -- (1.2,-1.0);
\draw (1.4,-0.8) -- (1.6,-1.0);
\draw (1.6,-1.0) -- (1.8,-1.2) [dotted];
\draw (1.8,-1.2) -- (1.6,-1.4);
\draw (1.8,-1.2) -- (2.0,-1.4) node [below=-1pt] {\tiny{$\sigma(n)$}};
\draw[scale=0.75] (0.40,-0.60) -- (1.00,-0.60) [dotted];
\draw (0,0) circle (1.5pt) [fill=white];
\end{tikzpicture} $$
and by the elements that span $\mCom$ and $\mLie$.
Because of the Jacobi identity, and because $\comtree$ and $\tree$ are (anti-)commutative, these elements are not all linear independent. To compute the dimension of $\mPoiss(n)$, we follow the proof of theorem 4.2 of \cite{BDK}. We first note that every tree like the one above (and all elements that span $\mCom$ and $\mLie$) gives a partition $A_1 \sqcup ... \sqcup A_k = \{1, ... , n\}$, where $A_i$ consists of all labels of the $i$th branch of the tree, and $k$ is the number of branches (with a ``branch", we mean a tree of Lie vertices extending from the commutative vertex). Note that each $A_i$ is ordered. The Jacobi identity allows us to put the largest number of each branch at the bottom, and our commutative element then allows us to rearrange the branches such that $max(A_1) < ... < max(A_k)$. We say that $A_i = \{ a_{i,1}, ... , a_{i,m_i} \}$, and we now see that there is a bijection between $S_n$ and the elements that span $\mPoiss(n)$, with 
$A_1 \sqcup ... \sqcup A_k \mapsto (a_{1,1}, ... , a_{1,m_1}) ... (a_{k,1}, ... , a_{k,m_k})$. So $dim(\mPoiss(n)) = n!$. 

Of course, the relations we gave mean exactly that a representation of $\mPoiss$ is a Poisson algebra with $\cdot:= \zeta(\comtree)$ and $[.,.]:= \zeta(\tree)$. 
\end{example}

\subsection{The quadratic dual}

We now want to define the dual of a quadratic operad $\mQ(E,R)$, which will lead to our first results. For this, we first want to identify $Free(E)^*$ and $Free(E^*)$: we choose the dual of a tree $T$ with vertices $v_i$ decorated with basis elements $e_i \in E(n_i)$ and inputs labelled $\sigma(1),...,\sigma(N)$ to be the same tree, with the same labelling of inputs, and with the vertices $v_i$ decorated by $e_i^*$. Then, we say that our map is linear in each vertex, so $(T,\alpha_1 \otimes ... \otimes \alpha_n): (T,\nu_1 \otimes ... \otimes \nu_n) \rightarrow \alpha_1 (\nu_1) ... \alpha_n (\nu_n)$. 

\begin{definition} In this way, $Free(E)^*$ becomes an operad, with the composition and action of $S_n$ induced by the ones on $Free(E^*)$. 
\end{definition}

In a way, the relations on our dual operad will be the ``orthogonal" of those on our original operad:

\begin{definition} If $V$ is a vector space, and $W\subset V$ is a subspace, then we can define the \textbf{annihilator} of $W$, a subspace of the dual space $V^*$: $Ann(W) := \{\alpha\in V^*|\alpha(w)=0, \forall w\in W\}$.
\end{definition}

We would like to now define the dual of $\mQ(E,R)$ as $Free(E^*)/<Ann(R)>$. However, the construction is not this straightforward; we first need to take care of some sign issues: 

\begin{definition} For a $S_n$-module $V$ over the ground field $k$, we define its \textbf{Czech dual} $V^\vee := V^*\otimes sgn$. Here, $sgn$ is the sign representation of $S_n$, and $S_n$ acts on $V^*$ in the usual way: for $\sigma \in S_n$, $\alpha \in V^\vee$, $\sigma(\alpha)=(sgn(\sigma))(\alpha \circ \sigma^{-1})$. 
\end{definition}

We can now define $Free(E^\vee)$ just like we defined $Free(E^*)$. But while $Free(E^\vee)$ and $Free(E)^\vee$ are clearly isomorphic as vector spaces, they are not operadically isomorphic via the canonical construction $v^* \circ_i w^* = (v\circ_i w)^*$: we need an extra sign convention, as is illustrated by the following computation. 
If 
$$ \begin{tikzpicture}
\draw (0,0.4) -- (0,0) node [right=1pt] {\small{$\alpha$}};
\draw (0,0) -- (-0.3,-0.3) node [right=1pt] {\small{$\beta$}};
\draw (0,0) -- (0.3,-0.3) node [below=-1pt] {\tiny{3}};
\draw (-0.3,-0.3) -- (-0.6,-0.6) node [below=-1pt] {\tiny{1}};
\draw (-0.3,-0.3) -- (0,-0.6) node [below=-1pt] {\tiny{2}};
\end{tikzpicture} (\begin{tikzpicture}
\draw (0,0.4) -- (0,0) node [right=1pt] {\small{$\mu$}};
\draw (0,0) -- (-0.3,-0.3) node [right=1pt] {\small{$\nu$}};
\draw (0,0) -- (0.3,-0.3) node [below=-1pt] {\tiny{3}};
\draw (-0.3,-0.3) -- (-0.6,-0.6) node [below=-1pt] {\tiny{1}};
\draw (-0.3,-0.3) -- (0,-0.6) node [below=-1pt] {\tiny{2}};
\end{tikzpicture}) = \pm\alpha(\mu) \beta(\nu)$$
then
\begin{align*} \begin{tikzpicture}
\draw (0,0.4) -- (0,0) node [right=1pt] {\small{$\alpha$}};
\draw (0,0) -- (-0.3,-0.3) node [below=-1pt] {\tiny{1}};
\draw (0,0) -- (0.3,-0.3) node [right=1pt] {\small{$\beta$}};
\draw (0.3,-0.3) -- (0,-0.6) node [below=-1pt] {\tiny{2}};
\draw (0.3,-0.3) -- (0.6,-0.6) node [below=-1pt] {\tiny{3}};
\end{tikzpicture} (\begin{tikzpicture}
\draw (0,0.4) -- (0,0) node [right=1pt] {\small{$\mu$}};
\draw (0,0) -- (-0.3,-0.3) node [below=-1pt] {\tiny{1}};
\draw (0,0) -- (0.3,-0.3) node [right=1pt] {\small{$\nu$}};
\draw (0.3,-0.3) -- (0,-0.6) node [below=-1pt] {\tiny{2}};
\draw (0.3,-0.3) -- (0.6,-0.6) node [below=-1pt] {\tiny{3}};
\end{tikzpicture}) & = \begin{tikzpicture}
\draw (0,0.4) -- (0,0) node [right=1pt] {\small{$(12)\alpha$}};
\draw (0,0) -- (-0.3,-0.3) node [right=1pt] {\small{$\beta$}};
\draw (0,0) -- (0.3,-0.3) node [below=-1pt] {\tiny{1}};
\draw (-0.3,-0.3) -- (-0.6,-0.6) node [below=-1pt] {\tiny{2}};
\draw (-0.3,-0.3) -- (0,-0.6) node [below=-1pt] {\tiny{3}};
\end{tikzpicture} (\begin{tikzpicture}
\draw (0,0.4) -- (0,0) node [right=1pt] {\small{$\mu$}};
\draw (0,0) -- (-0.3,-0.3) node [below=-1pt] {\tiny{1}};
\draw (0,0) -- (0.3,-0.3) node [right=1pt] {\small{$\nu$}};
\draw (0.3,-0.3) -- (0,-0.6) node [below=-1pt] {\tiny{2}};
\draw (0.3,-0.3) -- (0.6,-0.6) node [below=-1pt] {\tiny{3}};
\end{tikzpicture}) = - \begin{tikzpicture}
\draw (0,0.4) -- (0,0) node [right=1pt] {\small{$\alpha$}};
\draw (0,0) -- (-0.3,-0.3) node [right=1pt] {\small{$\beta$}};
\draw (0,0) -- (0.3,-0.3) node [below=-1pt] {\tiny{1}};
\draw (-0.3,-0.3) -- (-0.6,-0.6) node [below=-1pt] {\tiny{2}};
\draw (-0.3,-0.3) -- (0,-0.6) node [below=-1pt] {\tiny{3}};
\end{tikzpicture} (\begin{tikzpicture}
\draw (0,0.4) -- (0,0) node [right=1pt] {\small{$(12)\mu$}};
\draw (0,0) -- (-0.3,-0.3) node [below=-1pt] {\tiny{1}};
\draw (0,0) -- (0.3,-0.3) node [right=1pt] {\small{$\nu$}};
\draw (0.3,-0.3) -- (0,-0.6) node [below=-1pt] {\tiny{2}};
\draw (0.3,-0.3) -- (0.6,-0.6) node [below=-1pt] {\tiny{3}};
\end{tikzpicture}) \\
& = - \begin{tikzpicture}
\draw (0,0.4) -- (0,0) node [right=1pt] {\small{$\alpha$}};
\draw (0,0) -- (-0.3,-0.3) node [right=1pt] {\small{$\beta$}};
\draw (0,0) -- (0.3,-0.3) node [below=-1pt] {\tiny{1}};
\draw (-0.3,-0.3) -- (-0.6,-0.6) node [below=-1pt] {\tiny{2}};
\draw (-0.3,-0.3) -- (0,-0.6) node [below=-1pt] {\tiny{3}};
\end{tikzpicture} (\begin{tikzpicture}
\draw (0,0.4) -- (0,0) node [right=1pt] {\small{$\mu$}};
\draw (0,0) -- (-0.3,-0.3) node [right=1pt] {\small{$\nu$}};
\draw (0,0) -- (0.3,-0.3) node [below=-1pt] {\tiny{1}};
\draw (-0.3,-0.3) -- (-0.6,-0.6) node [below=-1pt] {\tiny{2}};
\draw (-0.3,-0.3) -- (0,-0.6) node [below=-1pt] {\tiny{3}};
\end{tikzpicture}) = \mp \alpha(\mu) \beta(\nu) 
\end{align*}
With $\alpha, \beta \in E^\vee$, $\nu, \mu \in E$. 
We use the following convention to identify $Free(E^\vee)(3)$ and $Free(E)^\vee(3)$ (since we only look at the complement of $R$ in $Free(E^\vee)(3)$, this will be enough): 
$$\begin{tikzpicture}
\draw (0,0.4) -- (0,0) node [right=1pt] {\small{$\alpha$}};
\draw (0,0) -- (-0.3,-0.3) node [right=1pt] {\small{$\beta$}};
\draw (0,0) -- (0.3,-0.3) node [below=-1pt] {\tiny{$\sigma(3)$}};
\draw (-0.3,-0.3) -- (-0.6,-0.6) node [below=-1pt] {\tiny{$\sigma(1)$}};
\draw (-0.3,-0.3) -- (0,-0.6) node [below=-1pt] {\tiny{$\sigma(2)$}};
\end{tikzpicture} (\begin{tikzpicture}
\draw (0,0.4) -- (0,0) node [right=1pt] {\small{$\mu$}};
\draw (0,0) -- (-0.3,-0.3) node [right=1pt] {\small{$\nu$}};
\draw (0,0) -- (0.3,-0.3) node [below=-1pt] {\tiny{$\sigma(3)$}};
\draw (-0.3,-0.3) -- (-0.6,-0.6) node [below=-1pt] {\tiny{$\sigma(1)$}};
\draw (-0.3,-0.3) -- (0,-0.6) node [below=-1pt] {\tiny{$\sigma(2)$}};
\end{tikzpicture}) = sgn(\sigma) \alpha(\mu) \beta(\nu) $$
$$ \begin{tikzpicture}
\draw (0,0.4) -- (0,0) node [right=1pt] {\small{$\alpha$}};
\draw (0,0) -- (-0.3,-0.3) node [below=-1pt] {\tiny{$\sigma(1)$}};
\draw (0,0) -- (0.3,-0.3) node [right=1pt] {\small{$\beta$}};
\draw (0.3,-0.3) -- (0,-0.6) node [below=-1pt] {\tiny{$\sigma(2)$}};
\draw (0.3,-0.3) -- (0.6,-0.6) node [below=-1pt] {\tiny{$\sigma(3)$}};
\end{tikzpicture} (\begin{tikzpicture}
\draw (0,0.4) -- (0,0) node [right=1pt] {\small{$\mu$}};
\draw (0,0) -- (-0.3,-0.3) node [below=-1pt] {\tiny{$\sigma(1)$}};
\draw (0,0) -- (0.3,-0.3) node [right=1pt] {\small{$\nu$}};
\draw (0.3,-0.3) -- (0,-0.6) node [below=-1pt] {\tiny{$\sigma(2)$}};
\draw (0.3,-0.3) -- (0.6,-0.6) node [below=-1pt] {\tiny{$\sigma(3)$}};
\end{tikzpicture}) = - sgn(\sigma) \alpha(\mu) \beta(\nu) $$

With this identification of $Free(E^\vee)(3)$ and $Free(E)^\vee(3)$, we can consider $Ann(R) \subset Free(E^\vee)(3)$, and the following definition now makes sense: 

\begin{definition} Given a quadratic operad $\mP=\mQ(E,R)$, we define its \textbf{quadratic dual}: $\mP^!:=\mQ(E^\vee,Ann(R))$
\end{definition}

\begin{remark} Because $(E^\vee)^\vee \cong E$ and $Ann(Ann(R)) \cong R$ (for finite dimensional vector spaces), $(\mP^!)^! \cong \mP$
\end{remark}

\begin{theorem} As operads, $\mAss^! \cong \mAss$, $\mCom^! \cong \mLie$, $\mLie^! \cong \mCom$ and $\mPoiss^! \cong \mPoiss$
\end{theorem}
\begin{proof} We start with $\mAss^!$. We see that $(12) \circ \tree^*=-\twisttree^*$, so as generating elements we choose $\tree^*$ and $-\twisttree^*$. Let's look at $Free(\tree,\twisttree)(3)$. It's 12-dimensional, and it's spanned by all six $\sigmaltree$ and all six $\sigmartree$ (recall that all $\twisttree$'s could be eliminated). Our set of relations $R$ is six-dimensional, and spanned by all six $\sigmaltree - \sigmartree$, so we see that $Ann(R) \subset (Free(\tree,\twisttree)(3))^*$ is spanned by all $\sigmaltree^*$ + $\sigmartree^*$. $Ann(R) \subset (Free(\tree,\twisttree)(3))^\vee$ is also spanned by these trees (the sign representation only adds a potential global minus sign here). But as we've seen above, $\sigmartree^* = -sgn(\sigma) \begin{tikzpicture}
\draw (0,0.25) -- (0,0) node [right=-1pt] {\tiny{*}};
\draw (0,0) -- (-0.2,-0.2) node [below left=-3pt and -8pt] {\tiny{$\sigma(1)$}};
\draw (0,0) -- (0.2,-0.2) node [right=-1pt] {\tiny{*}};
\draw (0.2,-0.2) -- (0,-0.4) node [below left=-2pt and -8pt] {\tiny{$\sigma(2)$}};
\draw (0.2,-0.2) -- (0.4,-0.4) node [below right=-2pt and -8pt] {\tiny{$\sigma(3)$}};
\end{tikzpicture}$ while $\sigmaltree^* = sgn(\sigma) \begin{tikzpicture}
\draw (0,0.25) -- (0,0) node [right=-1pt] {\tiny{*}};
\draw (0,0) -- (-0.2,-0.2) node [right=-1pt] {\tiny{*}};
\draw (0,0) -- (0.2,-0.2) node [below right=-3pt and -8pt] {\tiny{$\sigma(3)$}};
\draw (-0.2,-0.2) -- (-0.4,-0.4) node [below left=-2pt and -8pt] {\tiny{$\sigma(1)$}};
\draw (-0.2,-0.2) -- (0,-0.4) node [below right=-2pt and -8pt] {\tiny{$\sigma(2)$}};
\end{tikzpicture}$ so in $Free({\tree,\twisttree}^\vee)(3)$, $Ann(R)$ is spanned by $\begin{tikzpicture}
\draw (0,0.25) -- (0,0) node [right=-1pt] {\tiny{*}};
\draw (0,0) -- (-0.2,-0.2) node [right=-1pt] {\tiny{*}};
\draw (0,0) -- (0.2,-0.2) node [below right=-3pt and -8pt] {\tiny{$\sigma(3)$}};
\draw (-0.2,-0.2) -- (-0.4,-0.4) node [below left=-2pt and -8pt] {\tiny{$\sigma(1)$}};
\draw (-0.2,-0.2) -- (0,-0.4) node [below right=-2pt and -8pt] {\tiny{$\sigma(2)$}};
\end{tikzpicture} - \begin{tikzpicture}
\draw (0,0.25) -- (0,0) node [right=-1pt] {\tiny{*}};
\draw (0,0) -- (-0.2,-0.2) node [below left=-3pt and -8pt] {\tiny{$\sigma(1)$}};
\draw (0,0) -- (0.2,-0.2) node [right=-1pt] {\tiny{*}};
\draw (0.2,-0.2) -- (0,-0.4) node [below left=-2pt and -8pt] {\tiny{$\sigma(2)$}};
\draw (0.2,-0.2) -- (0.4,-0.4) node [below right=-2pt and -8pt] {\tiny{$\sigma(3)$}};
\end{tikzpicture}$. This precisely shows that $\mAss^! = \mQ( \tree^*,-\twisttree^* ; \begin{tikzpicture}
\draw (0,0.25) -- (0,0) node [right=-1pt] {\tiny{*}};
\draw (0,0) -- (-0.2,-0.2) node [right=-1pt] {\tiny{*}};
\draw (0,0) -- (0.2,-0.2);
\draw (-0.2,-0.2) -- (-0.4,-0.4);
\draw (-0.2,-0.2) -- (0,-0.4);
\end{tikzpicture} - \begin{tikzpicture}
\draw (0,0.25) -- (0,0) node [right=-1pt] {\tiny{*}};
\draw (0,0) -- (-0.2,-0.2);
\draw (0,0) -- (0.2,-0.2) node [right=-1pt] {\tiny{*}};
\draw (0.2,-0.2) -- (0,-0.4);
\draw (0.2,-0.2) -- (0.4,-0.4);
\end{tikzpicture} ) \cong \mAss$. 
 
Next we need to show that $\mComm$ and $\mLie$ are each others' dual. Note that the sign representation indeed makes $\comtree^*$ into an anticommutative element ($(12) \circ \commtree^* = - \commtree$) and makes our Lie element $\tree$ into a commutative element. So we need to show that the associative relation and the Jacobi identity are translated into each other. We first note that because of the (anti-)commutativity, both $Free(\comtree)(3)$ and $Free(\tree)(3)$ are three-dimensional, spanned respectively by \begin{tikzpicture}
\draw (0,0.25) -- (0,0);
\draw (0,0) -- (-0.2,-0.2);
\draw (0,0) -- (0.2,-0.2) node [below=-1pt] {\tiny{3}};
\draw (-0.2,-0.2) -- (-0.4,-0.4) node [below=-1pt] {\tiny{1}};
\draw (-0.2,-0.2) -- (0,-0.4) node [below=-1pt] {\tiny{2}};
\draw (0,0) circle (1pt) [fill=white];
\draw (-0.2,-0.2) circle (1pt) [fill=white];
\end{tikzpicture}, \begin{tikzpicture}
\draw (0,0.25) -- (0,0);
\draw (0,0) -- (-0.2,-0.2);
\draw (0,0) -- (0.2,-0.2) node [below=-1pt] {\tiny{1}};
\draw (-0.2,-0.2) -- (-0.4,-0.4) node [below=-1pt] {\tiny{2}};
\draw (-0.2,-0.2) -- (0,-0.4) node [below=-1pt] {\tiny{3}};
\draw (0,0) circle (1pt) [fill=white];
\draw (-0.2,-0.2) circle (1pt) [fill=white];
\end{tikzpicture}, \begin{tikzpicture}
\draw (0,0.25) -- (0,0);
\draw (0,0) -- (-0.2,-0.2);
\draw (0,0) -- (0.2,-0.2) node [below=-1pt] {\tiny{2}};
\draw (-0.2,-0.2) -- (-0.4,-0.4) node [below=-1pt] {\tiny{3}};
\draw (-0.2,-0.2) -- (0,-0.4) node [below=-1pt] {\tiny{1}};
\draw (0,0) circle (1pt) [fill=white];
\draw (-0.2,-0.2) circle (1pt) [fill=white];
\end{tikzpicture} and by \begin{tikzpicture}
\draw (0,0.25) -- (0,0);
\draw (0,0) -- (-0.2,-0.2);
\draw (0,0) -- (0.2,-0.2) node [below=-1pt] {\tiny{3}};
\draw (-0.2,-0.2) -- (-0.4,-0.4) node [below=-1pt] {\tiny{1}};
\draw (-0.2,-0.2) -- (0,-0.4) node [below=-1pt] {\tiny{2}};
\end{tikzpicture}, \begin{tikzpicture}
\draw (0,0.25) -- (0,0);
\draw (0,0) -- (-0.2,-0.2);
\draw (0,0) -- (0.2,-0.2) node [below=-1pt] {\tiny{1}};
\draw (-0.2,-0.2) -- (-0.4,-0.4) node [below=-1pt] {\tiny{2}};
\draw (-0.2,-0.2) -- (0,-0.4) node [below=-1pt] {\tiny{3}};
\end{tikzpicture}, \begin{tikzpicture}
\draw (0,0.25) -- (0,0);
\draw (0,0) -- (-0.2,-0.2);
\draw (0,0) -- (0.2,-0.2) node [below=-1pt] {\tiny{2}};
\draw (-0.2,-0.2) -- (-0.4,-0.4) node [below=-1pt] {\tiny{3}};
\draw (-0.2,-0.2) -- (0,-0.4) node [below=-1pt] {\tiny{1}};
\end{tikzpicture}. In $Free(\comtree)(3)$, the set of relations $R$ is spanned by $\begin{tikzpicture}
\draw (0,0.25) -- (0,0);
\draw (0,0) -- (-0.2,-0.2);
\draw (0,0) -- (0.2,-0.2) node [below=-1pt] {\tiny{3}};
\draw (-0.2,-0.2) -- (-0.4,-0.4) node [below=-1pt] {\tiny{1}};
\draw (-0.2,-0.2) -- (0,-0.4) node [below=-1pt] {\tiny{2}};
\draw (0,0) circle (1pt) [fill=white];
\draw (-0.2,-0.2) circle (1pt) [fill=white];
\end{tikzpicture} - \begin{tikzpicture}
\draw (0,0.25) -- (0,0);
\draw (0,0) -- (-0.2,-0.2) node [below=-1pt] {\tiny{1}};
\draw (0,0) -- (0.2,-0.2);
\draw (0.2,-0.2) -- (0,-0.4) node [below=-1pt] {\tiny{2}};
\draw (0.2,-0.2) -- (0.4,-0.4) node [below=-1pt] {\tiny{3}};
\draw (0,0) circle (1pt) [fill=white];
\draw (0.2,-0.2) circle (1pt) [fill=white];
\end{tikzpicture}$ and $\begin{tikzpicture}
\draw (0,0.25) -- (0,0);
\draw (0,0) -- (-0.2,-0.2);
\draw (0,0) -- (0.2,-0.2) node [below=-1pt] {\tiny{1}};
\draw (-0.2,-0.2) -- (-0.4,-0.4) node [below=-1pt] {\tiny{2}};
\draw (-0.2,-0.2) -- (0,-0.4) node [below=-1pt] {\tiny{3}};
\draw (0,0) circle (1pt) [fill=white];
\draw (-0.2,-0.2) circle (1pt) [fill=white];
\end{tikzpicture} - \begin{tikzpicture}
\draw (0,0.25) -- (0,0);
\draw (0,0) -- (-0.2,-0.2) node [below=-1pt] {\tiny{2}};
\draw (0,0) -- (0.2,-0.2);
\draw (0.2,-0.2) -- (0,-0.4) node [below=-1pt] {\tiny{3}};
\draw (0.2,-0.2) -- (0.4,-0.4) node [below=-1pt] {\tiny{1}};
\draw (0,0) circle (1pt) [fill=white];
\draw (0.2,-0.2) circle (1pt) [fill=white];
\end{tikzpicture}$ (because $\begin{tikzpicture}
\draw (0,0.25) -- (0,0);
\draw (0,0) -- (-0.2,-0.2);
\draw (0,0) -- (0.2,-0.2) node [below=-1pt] {\tiny{2}};
\draw (-0.2,-0.2) -- (-0.4,-0.4) node [below=-1pt] {\tiny{3}};
\draw (-0.2,-0.2) -- (0,-0.4) node [below=-1pt] {\tiny{1}};
\draw (0,0) circle (1pt) [fill=white];
\draw (-0.2,-0.2) circle (1pt) [fill=white];
\end{tikzpicture} - \begin{tikzpicture}
\draw (0,0.25) -- (0,0);
\draw (0,0) -- (-0.2,-0.2) node [below=-1pt] {\tiny{3}};
\draw (0,0) -- (0.2,-0.2);
\draw (0.2,-0.2) -- (0,-0.4) node [below=-1pt] {\tiny{1}};
\draw (0.2,-0.2) -- (0.4,-0.4) node [below=-1pt] {\tiny{2}};
\draw (0,0) circle (1pt) [fill=white];
\draw (0.2,-0.2) circle (1pt) [fill=white];
\end{tikzpicture} = - (\begin{tikzpicture}
\draw (0,0.25) -- (0,0);
\draw (0,0) -- (-0.2,-0.2);
\draw (0,0) -- (0.2,-0.2) node [below=-1pt] {\tiny{3}};
\draw (-0.2,-0.2) -- (-0.4,-0.4) node [below=-1pt] {\tiny{1}};
\draw (-0.2,-0.2) -- (0,-0.4) node [below=-1pt] {\tiny{2}};
\draw (0,0) circle (1pt) [fill=white];
\draw (-0.2,-0.2) circle (1pt) [fill=white];
\end{tikzpicture} - \begin{tikzpicture}
\draw (0,0.25) -- (0,0);
\draw (0,0) -- (-0.2,-0.2) node [below=-1pt] {\tiny{1}};
\draw (0,0) -- (0.2,-0.2);
\draw (0.2,-0.2) -- (0,-0.4) node [below=-1pt] {\tiny{2}};
\draw (0.2,-0.2) -- (0.4,-0.4) node [below=-1pt] {\tiny{3}};
\draw (0,0) circle (1pt) [fill=white];
\draw (0.2,-0.2) circle (1pt) [fill=white];
\end{tikzpicture}) - (\begin{tikzpicture}
\draw (0,0.25) -- (0,0);
\draw (0,0) -- (-0.2,-0.2);
\draw (0,0) -- (0.2,-0.2) node [below=-1pt] {\tiny{1}};
\draw (-0.2,-0.2) -- (-0.4,-0.4) node [below=-1pt] {\tiny{2}};
\draw (-0.2,-0.2) -- (0,-0.4) node [below=-1pt] {\tiny{3}};
\draw (0,0) circle (1pt) [fill=white];
\draw (-0.2,-0.2) circle (1pt) [fill=white];
\end{tikzpicture} - \begin{tikzpicture}
\draw (0,0.25) -- (0,0);
\draw (0,0) -- (-0.2,-0.2) node [below=-1pt] {\tiny{2}};
\draw (0,0) -- (0.2,-0.2);
\draw (0.2,-0.2) -- (0,-0.4) node [below=-1pt] {\tiny{3}};
\draw (0.2,-0.2) -- (0.4,-0.4) node [below=-1pt] {\tiny{1}};
\draw (0,0) circle (1pt) [fill=white];
\draw (0.2,-0.2) circle (1pt) [fill=white];
\end{tikzpicture})$) or, in terms of the basis we mentioned above, by $\begin{tikzpicture}
\draw (0,0.25) -- (0,0);
\draw (0,0) -- (-0.2,-0.2);
\draw (0,0) -- (0.2,-0.2) node [below=-1pt] {\tiny{3}};
\draw (-0.2,-0.2) -- (-0.4,-0.4) node [below=-1pt] {\tiny{1}};
\draw (-0.2,-0.2) -- (0,-0.4) node [below=-1pt] {\tiny{2}};
\draw (0,0) circle (1pt) [fill=white];
\draw (-0.2,-0.2) circle (1pt) [fill=white];
\end{tikzpicture} - \begin{tikzpicture}
\draw (0,0.25) -- (0,0);
\draw (0,0) -- (-0.2,-0.2);
\draw (0,0) -- (0.2,-0.2) node [below=-1pt] {\tiny{1}};
\draw (-0.2,-0.2) -- (-0.4,-0.4) node [below=-1pt] {\tiny{2}};
\draw (-0.2,-0.2) -- (0,-0.4) node [below=-1pt] {\tiny{3}};
\draw (0,0) circle (1pt) [fill=white];
\draw (-0.2,-0.2) circle (1pt) [fill=white];
\end{tikzpicture}$ and $\begin{tikzpicture}
\draw (0,0.25) -- (0,0);
\draw (0,0) -- (-0.2,-0.2);
\draw (0,0) -- (0.2,-0.2) node [below=-1pt] {\tiny{1}};
\draw (-0.2,-0.2) -- (-0.4,-0.4) node [below=-1pt] {\tiny{2}};
\draw (-0.2,-0.2) -- (0,-0.4) node [below=-1pt] {\tiny{3}};
\draw (0,0) circle (1pt) [fill=white];
\draw (-0.2,-0.2) circle (1pt) [fill=white];
\end{tikzpicture} - \begin{tikzpicture}
\draw (0,0.25) -- (0,0);
\draw (0,0) -- (-0.2,-0.2);
\draw (0,0) -- (0.2,-0.2) node [below=-1pt] {\tiny{2}};
\draw (-0.2,-0.2) -- (-0.4,-0.4) node [below=-1pt] {\tiny{3}};
\draw (-0.2,-0.2) -- (0,-0.4) node [below=-1pt] {\tiny{1}};
\draw (0,0) circle (1pt) [fill=white];
\draw (-0.2,-0.2) circle (1pt) [fill=white];
\end{tikzpicture}$. Its annihilator in $Free(\comtree)(3)^*$ is $\begin{tikzpicture}
\draw (0,0.25) -- (0,0);
\draw (0,0) -- (-0.2,-0.2);
\draw (0,0) -- (0.2,-0.2) node [below=-1pt] {\tiny{3}};
\draw (-0.2,-0.2) -- (-0.4,-0.4) node [below=-1pt] {\tiny{1}};
\draw (-0.2,-0.2) -- (0,-0.4) node [below=-1pt] {\tiny{2}};
\draw (0,0) circle (1pt) [fill=white];
\draw (-0.2,-0.2) circle (1pt) [fill=white];
\end{tikzpicture}^* + \begin{tikzpicture}
\draw (0,0.25) -- (0,0);
\draw (0,0) -- (-0.2,-0.2);
\draw (0,0) -- (0.2,-0.2) node [below=-1pt] {\tiny{1}};
\draw (-0.2,-0.2) -- (-0.4,-0.4) node [below=-1pt] {\tiny{2}};
\draw (-0.2,-0.2) -- (0,-0.4) node [below=-1pt] {\tiny{3}};
\draw (0,0) circle (1pt) [fill=white];
\draw (-0.2,-0.2) circle (1pt) [fill=white];
\end{tikzpicture}^* + \begin{tikzpicture}
\draw (0,0.25) -- (0,0);
\draw (0,0) -- (-0.2,-0.2);
\draw (0,0) -- (0.2,-0.2) node [below=-1pt] {\tiny{2}};
\draw (-0.2,-0.2) -- (-0.4,-0.4) node [below=-1pt] {\tiny{3}};
\draw (-0.2,-0.2) -- (0,-0.4) node [below=-1pt] {\tiny{1}};
\draw (0,0) circle (1pt) [fill=white];
\draw (-0.2,-0.2) circle (1pt) [fill=white];
\end{tikzpicture}^*$, which is exactly the Jacobi relation. We know that this translates into $\begin{tikzpicture}
\draw (0,0.25) -- (0,0) node [right=-1pt] {\tiny{*}};
\draw (0,0) -- (-0.2,-0.2) node [right=-1pt] {\tiny{*}};
\draw (0,0) -- (0.2,-0.2) node [below=-1pt] {\tiny{3}};
\draw (-0.2,-0.2) -- (-0.4,-0.4) node [below=-1pt] {\tiny{1}};
\draw (-0.2,-0.2) -- (0,-0.4) node [below=-1pt] {\tiny{2}};
\draw (0,0) circle (1pt) [fill=white];
\draw (-0.2,-0.2) circle (1pt) [fill=white];
\end{tikzpicture} + \begin{tikzpicture}
\draw (0,0.25) -- (0,0) node [right=-1pt] {\tiny{*}};
\draw (0,0) -- (-0.2,-0.2) node [right=-1pt] {\tiny{*}};
\draw (0,0) -- (0.2,-0.2) node [below=-1pt] {\tiny{1}};
\draw (-0.2,-0.2) -- (-0.4,-0.4) node [below=-1pt] {\tiny{2}};
\draw (-0.2,-0.2) -- (0,-0.4) node [below=-1pt] {\tiny{3}};
\draw (0,0) circle (1pt) [fill=white];
\draw (-0.2,-0.2) circle (1pt) [fill=white];
\end{tikzpicture} + \begin{tikzpicture}
\draw (0,0.25) -- (0,0) node [right=-1pt] {\tiny{*}};
\draw (0,0) -- (-0.2,-0.2) node [right=-1pt] {\tiny{*}};
\draw (0,0) -- (0.2,-0.2) node [below=-1pt] {\tiny{2}};
\draw (-0.2,-0.2) -- (-0.4,-0.4) node [below=-1pt] {\tiny{3}};
\draw (-0.2,-0.2) -- (0,-0.4) node [below=-1pt] {\tiny{1}};
\draw (0,0) circle (1pt) [fill=white];
\draw (-0.2,-0.2) circle (1pt) [fill=white];
\end{tikzpicture}$ in $Free(\commtree^\vee)(3)$ so indeed, $\mComm^! \cong \mLie$.\footnote{In this case we don't have any troubles with $\sigmaltree^* = sgn(\sigma) \begin{tikzpicture}
\draw (0,0.25) -- (0,0) node [right=-1pt] {\tiny{*}};
\draw (0,0) -- (-0.2,-0.2) node [right=-1pt] {\tiny{*}};
\draw (0,0) -- (0.2,-0.2) node [below right=-3pt and -8pt] {\tiny{$\sigma(3)$}};
\draw (-0.2,-0.2) -- (-0.4,-0.4) node [below left=-2pt and -8pt] {\tiny{$\sigma(1)$}};
\draw (-0.2,-0.2) -- (0,-0.4) node [below right=-2pt and -8pt] {\tiny{$\sigma(2)$}};
\end{tikzpicture}$: the relation just implies that we add a minus sign before half of the trees, which does not add extra information to our operad.} And as we've seen above, $(\mP^!)^! \cong \mP$, so $\mLie^! \cong (\mComm^!)^! \cong \mComm$. 

Finally, we need to show that $\mPoiss$ is self-dual. We have already seen above that the $\comtree^*$ becomes a Lie-element and vice versa, so we just need to show that the distributive relation translates into the same distributive relation for $\comtree^*$ and $\tree^*$. Let us again look at $Free( \{\comtree, \tree\})(3)$, but this time only at elements consisting of one $\comtree$ and one $\tree$. It's six-dimensional, spanned by 
$\begin{tikzpicture}
\draw (0,0.25) -- (0,0);
\draw (0,0) -- (-0.2,-0.2);
\draw (0,0) -- (0.2,-0.2) node [below=-1pt] {\tiny{3}};
\draw (-0.2,-0.2) -- (-0.4,-0.4) node [below=-1pt] {\tiny{1}};
\draw (-0.2,-0.2) -- (0,-0.4) node [below=-1pt] {\tiny{2}};
\draw (0,0) circle (1pt) [fill=white];
\end{tikzpicture}, \begin{tikzpicture}
\draw (0,0.25) -- (0,0);
\draw (0,0) -- (-0.2,-0.2);
\draw (0,0) -- (0.2,-0.2) node [below=-1pt] {\tiny{1}};
\draw (-0.2,-0.2) -- (-0.4,-0.4) node [below=-1pt] {\tiny{2}};
\draw (-0.2,-0.2) -- (0,-0.4) node [below=-1pt] {\tiny{3}};
\draw (0,0) circle (1pt) [fill=white];
\end{tikzpicture}, \begin{tikzpicture}
\draw (0,0.25) -- (0,0);
\draw (0,0) -- (-0.2,-0.2);
\draw (0,0) -- (0.2,-0.2) node [below=-1pt] {\tiny{2}};
\draw (-0.2,-0.2) -- (-0.4,-0.4) node [below=-1pt] {\tiny{3}};
\draw (-0.2,-0.2) -- (0,-0.4) node [below=-1pt] {\tiny{1}};
\draw (0,0) circle (1pt) [fill=white];
\end{tikzpicture}$ 
and 
$\begin{tikzpicture}
\draw (0,0.25) -- (0,0);
\draw (0,0) -- (-0.2,-0.2);
\draw (0,0) -- (0.2,-0.2) node [below=-1pt] {\tiny{3}};
\draw (-0.2,-0.2) -- (-0.4,-0.4) node [below=-1pt] {\tiny{1}};
\draw (-0.2,-0.2) -- (0,-0.4) node [below=-1pt] {\tiny{2}};
\draw (-0.2,-0.2) circle (1pt) [fill=white];
\end{tikzpicture}, \begin{tikzpicture}
\draw (0,0.25) -- (0,0);
\draw (0,0) -- (-0.2,-0.2);
\draw (0,0) -- (0.2,-0.2) node [below=-1pt] {\tiny{1}};
\draw (-0.2,-0.2) -- (-0.4,-0.4) node [below=-1pt] {\tiny{2}};
\draw (-0.2,-0.2) -- (0,-0.4) node [below=-1pt] {\tiny{3}};
\draw (-0.2,-0.2) circle (1pt) [fill=white];
\end{tikzpicture}, \begin{tikzpicture}
\draw (0,0.25) -- (0,0);
\draw (0,0) -- (-0.2,-0.2);
\draw (0,0) -- (0.2,-0.2) node [below=-1pt] {\tiny{2}};
\draw (-0.2,-0.2) -- (-0.4,-0.4) node [below=-1pt] {\tiny{3}};
\draw (-0.2,-0.2) -- (0,-0.4) node [below=-1pt] {\tiny{1}};
\draw (-0.2,-0.2) circle (1pt) [fill=white];
\end{tikzpicture} $. The set of relations is three-dimensional, and in terms of this basis it is spanned by 
$\begin{tikzpicture}
\draw (0,0.25) -- (0,0);
\draw (0,0) -- (-0.2,-0.2);
\draw (0,0) -- (0.2,-0.2) node [below=-1pt] {\tiny{1}};
\draw (-0.2,-0.2) -- (-0.4,-0.4) node [below=-1pt] {\tiny{2}};
\draw (-0.2,-0.2) -- (0,-0.4) node [below=-1pt] {\tiny{3}};
\draw (-0.2,-0.2) circle (1pt) [fill=white];
\end{tikzpicture} - \begin{tikzpicture}
\draw (0,0.25) -- (0,0);
\draw (0,0) -- (-0.2,-0.2);
\draw (0,0) -- (0.2,-0.2) node [below=-1pt] {\tiny{2}};
\draw (-0.2,-0.2) -- (-0.4,-0.4) node [below=-1pt] {\tiny{3}};
\draw (-0.2,-0.2) -- (0,-0.4) node [below=-1pt] {\tiny{1}};
\draw (0,0) circle (1pt) [fill=white];
\end{tikzpicture} + \begin{tikzpicture}
\draw (0,0.25) -- (0,0);
\draw (0,0) -- (-0.2,-0.2);
\draw (0,0) -- (0.2,-0.2) node [below=-1pt] {\tiny{3}};
\draw (-0.2,-0.2) -- (-0.4,-0.4) node [below=-1pt] {\tiny{1}};
\draw (-0.2,-0.2) -- (0,-0.4) node [below=-1pt] {\tiny{2}};
\draw (0,0) circle (1pt) [fill=white];
\end{tikzpicture}$ and its cyclic permutations. Some linear algebra shows us that $Ann(R)$ is exactly spanned by 
$\begin{tikzpicture}
\draw (0,0.25) -- (0,0) node [right=-1pt] {\tiny{*}};
\draw (0,0) -- (-0.2,-0.2) node [right=-1pt] {\tiny{*}};
\draw (0,0) -- (0.2,-0.2) node [below=-1pt] {\tiny{1}};
\draw (-0.2,-0.2) -- (-0.4,-0.4) node [below=-1pt] {\tiny{2}};
\draw (-0.2,-0.2) -- (0,-0.4) node [below=-1pt] {\tiny{3}};
\draw (0,0) circle (1pt) [fill=white];
\end{tikzpicture} - \begin{tikzpicture}
\draw (0,0.25) -- (0,0) node [right=-1pt] {\tiny{*}};
\draw (0,0) -- (-0.2,-0.2) node [right=-1pt] {\tiny{*}};
\draw (0,0) -- (0.2,-0.2) node [below=-1pt] {\tiny{2}};
\draw (-0.2,-0.2) -- (-0.4,-0.4) node [below=-1pt] {\tiny{3}};
\draw (-0.2,-0.2) -- (0,-0.4) node [below=-1pt] {\tiny{1}};
\draw (-0.2,-0.2) circle (1pt) [fill=white];
\end{tikzpicture} + \begin{tikzpicture}
\draw (0,0.25) -- (0,0) node [right=-1pt] {\tiny{*}};
\draw (0,0) -- (-0.2,-0.2) node [right=-1pt] {\tiny{*}};
\draw (0,0) -- (0.2,-0.2) node [below=-1pt] {\tiny{3}};
\draw (-0.2,-0.2) -- (-0.4,-0.4) node [below=-1pt] {\tiny{1}};
\draw (-0.2,-0.2) -- (0,-0.4) node [below=-1pt] {\tiny{2}};
\draw (-0.2,-0.2) circle (1pt) [fill=white];
\end{tikzpicture}$ and its cyclic permutations, so we indeed see that $Ann(R)$ again describes the same distributive relation, and $\mPoiss^! \cong \mPoiss$. 
\end{proof}

\section{Resolutions}
One way of studying an operad is by looking at a complex of free operads with a differential, such that the (co)homology of the complex gives us our original operad. In other words, we encode the relations on our operad in the differential. To define this, we first need the notion of an operadic complex: 

\begin{definition} A \textbf{dg-operad} (differential graded operad) of degree 1 is an operad, with the following extra structure:
\begin{enumerate}
\item[i)] Each $\mP(n)$ is a graded vector space, and we denote $\mP(n)^i$ for the $i$th degree of this space. This means that the composition becomes a map of graded vector spaces:
$$\circ_i:\mP(n)^j\otimes\mP(m)^k\rightarrow\mP(n+m-1)^{j+k}$$
\item[ii)] A homomorphism of operads $d:\mP(n)^i\rightarrow \mP(n)^{i+1}$ of degree one, such that $d^2=0$. We demand that $d$ satisfies the Leibniz rule:
$$d(\nu\circ_i\mu)=d(\nu)\circ_i\mu+(-1)^j\nu\circ_i d(\mu)$$
where $\nu\in\mP(n)^j$
\end{enumerate}
\end{definition}

\begin{definition} A \textbf{resolution} of an operad $\mP$ is a free dg-operad $\mQ$ such that $H^*(\mQ)\cong\mP$. 
\end{definition}

Note that we can make any operad $\mP$ into a dg-operad by putting each $\mP(n)$ in degree 0, defining $\mP(n)^i:=0$ if $i\neq 0$ and consequently defining $d=0$ as the zero map. From now on, if $V$ is a vector space, we will assume it to be in degree 0. 
But there are more interesting, and less trivial ways to make a dg-operad using an existing operad. First, some notation: if $C$ is a complex of vector spaces, then $C[i]$ is the same complex, with degree shifted by $i$: $C[i]^j:=C^{j+i}$, and $d_{C[i]}=(-1)^i d_C$. We can now construct the cobar complex:

\subsection{The cobar complex}

We start by defining $Cob(\mP):=Free(\mP(n)^*[-1]|n\geq 2) \otimes det(T) \otimes sgn$, 
where $\mP(n)^*
$ is the dual space of $\mP(n)$, $det(T)$ is a one-dimensional representation that takes care of the signs, as described below, and $sgn$ is the sign representation as before. 
Note that by the convention mentioned above, $\mP(n)[-1]^1=\mP(n)$ and $\mP(n)[-1]^i=0$ if $i\neq1$. 
So $Cob(\mP)(n)^m$ consists of trees with $m$ vertices and $n$ outputs, with each vertex ``decorated" by an element of $\mP(n)^*$, where $n$ is the number outputs of the vertex. 
For example, an element of $Cob(\mP)(5)^2$ looks like: 
$$ 
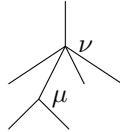
\begin{tikzpicture}
\draw (0,0.6) -- (0,0) node [right=1pt] {\small{$\nu$}};
\draw (0,0) -- (-0.75,-0.5);
\draw (0,0) -- (-0.35,-0.7) node [right=1pt] {\small{$\mu$}};
\draw (0,0) -- (0.25,-0.5);
\draw (0,0) -- (0.75,-0.5);
\draw (-0.35,-0.7) -- (-0.75,-1.1);
\draw (-0.35,-0.7) -- (0.05,-1.1);
\end{tikzpicture} $$
with of course $\nu \in \mP(4)^*$ and $\mu \in \mP(2)^*$. 
We remark that $Cob(\mP)(n)$ can only be non-zero in degrees $1$ to $n-1$: 
its basis elements of lowest degree look like \begin{tikzpicture}
\draw (0,0.25) -- (0,-0.00) node [above right= -4pt and -2pt] {\small{$\nu$}};
\draw (0,0) -- (-0.3,-0.2);
\draw (0,0) -- (-0.15,-0.2);
\draw (-0.10,-0.2) -- (0.25,-0.2) [dotted];
\draw (0,0) -- (0.3,-0.2);
\end{tikzpicture}, which lies in degree 1, and since an element of maximal degree has a minimal amount of outputs at each vertex (two), it is a binary tree, that lies in degree $n-1$. 
Also, please note that we now have two operadic compositions: one in our original operad $\mP$ and one in our new complex $Cob(\mP)$. We will for now denote the two compositions as $\circ^{\mP}$ and $\circ^{C}$ if it's not clear about which one we're talking. 

We now want to define a map $d:Cob(\mP)(n)^i \rightarrow Cob(\mP)(n)^{i+1}$ of degree 1 that makes the cobar resolution into a complex. The degree 1 means that it will map a tree with $i$ vertices to a tree with $i+1$ vertices, but the amount of leaves ($n$) remains the same. 
Since $d$ is demanded to be linear, we need only define it on the basis elements, the decorated trees. On such a tree, we define $d$ to be the dual of the tree-composition $\circ_T$: for $\alpha \in Cob(\mP)(n)$, $d(\alpha) (\nu)= \alpha (\circ_T \nu)$. This $\circ_T$ maps a tree decorated with elements from $\mP$ to the sum of all possible trees you get by contracting one edge: 
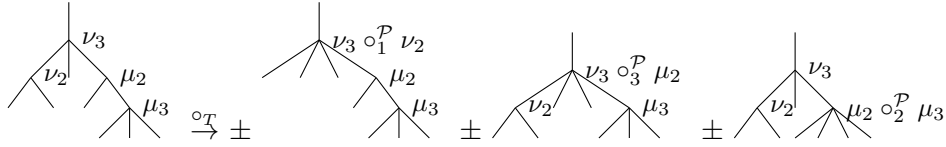
\begin{center}
\begin{tikzpicture}
\draw (0,0.5) -- (0,0) node [right=1pt] {\small{$\nu_3$}};
\draw (0,0) -- (-0.5,-0.5) node [right=1pt] {\small{$\nu_2$}};
\draw (0,0) -- (0,-0.5);
\draw (0,0) -- (0.5,-0.5) node [right=1pt] {\small{$\mu_2$}};
\draw (-0.5,-0.5) -- (-0.8,-0.9);
\draw (-0.5,-0.5) -- (-0.2,-0.9);
\draw (0.5,-0.5) -- (0.2,-0.9);
\draw (0.5,-0.5) -- (0.8,-0.9) node [right=1pt] {\small{$\mu_3$}};
\draw (0.8,-0.9) -- (0.4,-1.3);
\draw (0.8,-0.9) -- (0.8,-1.3);
\draw (0.8,-0.9) -- (1.2,-1.3);
\end{tikzpicture}
$\stackrel{\circ_T}{\rightarrow}$
{$\pm$} \begin{tikzpicture}
\draw (0,0.5) -- (0,0) node [right=1pt] {\small{$\nu_3 \circ^\mP_1 \nu_2$}};
\draw (0,0) -- (-0.75,-0.5);
\draw (0,0) -- (-0.25,-0.5);
\draw (0,0) -- (0.25,-0.5);
\draw (0,0) -- (0.75,-0.5) node [right=1pt] {\small{$\mu_2$}};
\draw (0.75,-0.5) -- (0.45,-0.9);
\draw (0.75,-0.5) -- (1.05,-0.9) node [right=1pt] {\small{$\mu_3$}};
\draw (1.05,-0.9) -- (0.65,-1.3);
\draw (1.05,-0.9) -- (1.05,-1.3);
\draw (1.05,-0.9) -- (1.45,-1.3);
\end{tikzpicture}
$\pm$ \begin{tikzpicture}
\draw (0,0.5) -- (0,0) node [right=1pt] {\small{$\nu_3 \circ^\mP_3 \mu_2$}};
\draw (0,0) -- (-0.75,-0.5) node [right=1pt] {\small{$\nu_2$}};
\draw (0,0) -- (-0.25,-0.5);
\draw (0,0) -- (0.25,-0.5);
\draw (0,0) -- (0.75,-0.5) node [right=1pt] {\small{$\mu_3$}};
\draw (-0.75,-0.5) -- (-1.05,-0.9);
\draw (-0.75,-0.5) -- (-0.45,-0.9);
\draw (0.75,-0.5) -- (0.35,-0.9);
\draw (0.75,-0.5) -- (0.75,-0.9);
\draw (0.75,-0.5) -- (1.15,-0.9);
\end{tikzpicture}
$\pm$ \begin{tikzpicture}
\draw (0,0.5) -- (0,0) node [right=1pt] {\small{$\nu_3$}};
\draw (0,0) -- (-0.5,-0.5) node [right=1pt] {\small{$\nu_2$}};
\draw (0,0) -- (0,-0.5);
\draw (0,0) -- (0.5,-0.5) node [right=1pt] {\small{$\mu_2 \circ^\mP_2 \mu_3$}};
\draw (-0.5,-0.5) -- (-0.8,-0.9);
\draw (-0.5,-0.5) -- (-0.2,-0.9);
\draw (0.5,-0.5) -- (0,-0.9);
\draw (0.5,-0.5) -- (0.3,-0.9);
\draw (0.5,-0.5) -- (0.7,-0.9);
\draw (0.5,-0.5) -- (1,-0.9);
\end{tikzpicture}
\end{center}
This means that $d$ will ``pull an edge out": on the most simple element, it will look like 
$$
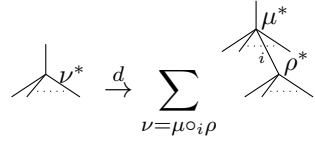
\begin{tikzpicture}
\draw (0,0.4) -- (0,0) node [right=1pt] {\small{$\nu^*$}};
\draw (0,0) -- (-0.45,-0.3);
\draw (0,0) -- (-0.25,-0.3);
\draw [scale=0.75] (-0.15,-0.3) -- (0.35,-0.3) [dotted];
\draw (0,0) -- (0.45,-0.3);
\end{tikzpicture} \stackrel{d}{\rightarrow}\sum_{\nu = \mu \circ_i \rho}
\begin{tikzpicture} [scale=1.5]
\draw (0,0.25) -- (0,0) node [above right= -4pt and -2pt] {\small{$\mu^*$}};
\draw (0,0) -- (-0.3,-0.2);
\draw (0,0) -- (-0.15,-0.2);
\draw (0,0) -- (0.2,-0.4) node [above right= -4pt and -2pt] {\small{$\rho^*$}};
\path (0.125,-0.25) node [left=-2pt] {\tiny{$i$}};
\draw (-0.075,-0.15) -- (0.1875,-0.15) [dotted];
\draw (0,0) -- (0.3,-0.2);
\draw (0.2,-0.4) -- (-0.1,-0.6);
\draw (0.2,-0.4) -- (0.05,-0.6);
\draw (0.125,-0.55) -- (0.3875,-0.55) [dotted];
\draw (0.2,-0.4) -- (0.5,-0.6);
\end{tikzpicture} $$
Note that this is indeed a map of degree $1$. 

Now, to make sure that $d^2=0$, we will need the $det(T)$ mentioned above (the $\otimes sgn$ doesn't contribute to the signs within $d$ at all, because $d$ commutes with the action of $S_n$). We illustrate this by calculating $d^2(\begin{tikzpicture}
\draw (0,0.25) -- (0,0);
\draw (0,0) -- (-0.3,-0.2);
\draw (0,0) -- (-0.1,-0.2);
\draw (0,0) -- (0.1,-0.2);
\draw (0,0) -- (0.3,-0.2);
\end{tikzpicture})$ in the associative operad 
using the graded Leibniz rule: 
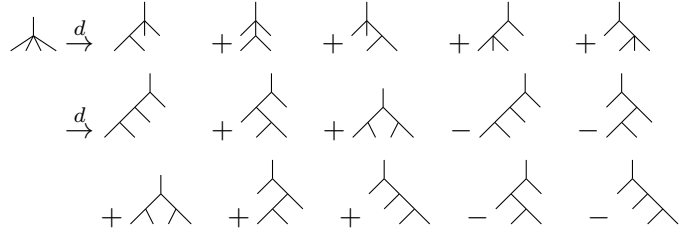
\begin{align*}
\begin{tikzpicture}
\draw (0,0.25) -- (0,0);
\draw (0,0) -- (-0.3,-0.2);
\draw (0,0) -- (-0.1,-0.2);
\draw (0,0) -- (0.1,-0.2);
\draw (0,0) -- (0.3,-0.2);
\end{tikzpicture}
\stackrel{d}{\rightarrow} & \ \
\begin{tikzpicture}
\draw (0,0.25) -- (0,0);
\draw (0,0) -- (-0.2,-0.2);
\draw (0,0) -- (0,-0.2);
\draw (0,0) -- (0.2,-0.2);
\draw (-0.2,-0.2) -- (-0.4,-0.4);
\draw (-0.2,-0.2) -- (0,-0.4);
\end{tikzpicture} \ \ \ \ \, 
+ \begin{tikzpicture}
\draw (0,0.25) -- (0,0);
\draw (0,0) -- (-0.2,-0.2);
\draw (0,0) -- (0,-0.2);
\draw (0,0) -- (0.2,-0.2);
\draw (0,-0.2) -- (-0.2,-0.4);
\draw (0,-0.2) -- (0.2,-0.4);
\end{tikzpicture} \ \ \ \ \, 
+ \begin{tikzpicture}
\draw (0,0.25) -- (0,0);
\draw (0,0) -- (-0.2,-0.2);
\draw (0,0) -- (0,-0.2);
\draw (0,0) -- (0.2,-0.2);
\draw (0.2,-0.2) -- (0,-0.4);
\draw (0.2,-0.2) -- (0.4,-0.4);
\end{tikzpicture} \ \ \ \ \, 
+ \begin{tikzpicture}
\draw (0,0.25) -- (0,0);
\draw (0,0) -- (-0.2,-0.2);
\draw (0,0) -- (0.2,-0.2);
\draw (-0.2,-0.2) -- (-0.4,-0.4);
\draw (-0.2,-0.2) -- (-0.2,-0.4);
\draw (-0.2,-0.2) -- (0,-0.4);
\end{tikzpicture} \ \ \ \ \, 
+ \begin{tikzpicture}
\draw (0,0.25) -- (0,0);
\draw (0,0) -- (-0.2,-0.2);
\draw (0,0) -- (0.2,-0.2);
\draw (0.2,-0.2) -- (0,-0.4);
\draw (0.2,-0.2) -- (0.2,-0.4);
\draw (0.2,-0.2) -- (0.4,-0.4);
\end{tikzpicture} 
\\
\stackrel{d}{\rightarrow} & \ 
\begin{tikzpicture}
\draw (0,0.25) -- (0,0);
\draw (0,0) -- (-0.2,-0.2);
\draw (0,0) -- (0.2,-0.2);
\draw (-0.2,-0.2) -- (-0.4,-0.4);
\draw (-0.2,-0.2) -- (0,-0.4);
\draw (-0.4,-0.4) -- (-0.6,-0.6);
\draw (-0.4,-0.4) -- (-0.2,-0.6);
\end{tikzpicture} \ \ \ \ 
+ 
\begin{tikzpicture}
\draw (0,0.25) -- (0,0);
\draw (0,0) -- (-0.2,-0.2);
\draw (0,0) -- (0.2,-0.2);
\draw (-0.2,-0.2) -- (-0.4,-0.4);
\draw (-0.2,-0.2) -- (0,-0.4);
\draw (0,-0.4) -- (-0.2,-0.6);
\draw (0,-0.4) -- (0.2,-0.6);
\end{tikzpicture} \ \ \ 
+ 
\begin{tikzpicture}
\draw (0,0.25) -- (0,0);
\draw (0,0) -- (-0.2,-0.2);
\draw (0,0) -- (0.2,-0.2);
\draw (-0.2,-0.2) -- (-0.4,-0.4);
\draw (-0.2,-0.2) -- (-0.1,-0.4);
\draw (0.2,-0.2) -- (0.1,-0.4);
\draw (0.2,-0.2) -- (0.4,-0.4);
\end{tikzpicture} \ \ \ 
- 
\begin{tikzpicture}
\draw (0,0.25) -- (0,0);
\draw (0,0) -- (-0.2,-0.2);
\draw (0,0) -- (0.2,-0.2);
\draw (-0.2,-0.2) -- (-0.4,-0.4);
\draw (-0.2,-0.2) -- (0,-0.4);
\draw (-0.4,-0.4) -- (-0.6,-0.6);
\draw (-0.4,-0.4) -- (-0.2,-0.6);
\end{tikzpicture} \ \ \ 
- 
\begin{tikzpicture}
\draw (0,0.25) -- (0,0);
\draw (0,0) -- (-0.2,-0.2);
\draw (0,0) -- (0.2,-0.2);
\draw (0.2,-0.2) -- (0,-0.4);
\draw (0.2,-0.2) -- (0.4,-0.4);
\draw (0,-0.4) -- (-0.2,-0.6);
\draw (0,-0.4) -- (0.2,-0.6);
\end{tikzpicture} 
\\
& + 
\begin{tikzpicture}
\draw (0,0.25) -- (0,0);
\draw (0,0) -- (-0.2,-0.2);
\draw (0,0) -- (0.2,-0.2);
\draw (-0.2,-0.2) -- (-0.4,-0.4);
\draw (-0.2,-0.2) -- (-0.1,-0.4);
\draw (0.2,-0.2) -- (0.1,-0.4);
\draw (0.2,-0.2) -- (0.4,-0.4);
\end{tikzpicture} \ \ \ 
+ 
\begin{tikzpicture}
\draw (0,0.25) -- (0,0);
\draw (0,0) -- (-0.2,-0.2);
\draw (0,0) -- (0.2,-0.2);
\draw (0.2,-0.2) -- (0,-0.4);
\draw (0.2,-0.2) -- (0.4,-0.4);
\draw (0,-0.4) -- (-0.2,-0.6);
\draw (0,-0.4) -- (0.2,-0.6);
\end{tikzpicture} \ \ \ 
+ 
\begin{tikzpicture}
\draw (0,0.25) -- (0,0);
\draw (0,0) -- (-0.2,-0.2);
\draw (0,0) -- (0.2,-0.2);
\draw (0.2,-0.2) -- (0,-0.4);
\draw (0.2,-0.2) -- (0.4,-0.4);
\draw (0.4,-0.4) -- (0.2,-0.6);
\draw (0.4,-0.4) -- (0.6,-0.6);
\end{tikzpicture} \ \ \ 
- 
\begin{tikzpicture}
\draw (0,0.25) -- (0,0);
\draw (0,0) -- (-0.2,-0.2);
\draw (0,0) -- (0.2,-0.2);
\draw (-0.2,-0.2) -- (-0.4,-0.4);
\draw (-0.2,-0.2) -- (0,-0.4);
\draw (0,-0.4) -- (-0.2,-0.6);
\draw (0,-0.4) -- (0.2,-0.6);
\end{tikzpicture} \ \ \ \ 
- 
\begin{tikzpicture}
\draw (0,0.25) -- (0,0);
\draw (0,0) -- (-0.2,-0.2);
\draw (0,0) -- (0.2,-0.2);
\draw (0.2,-0.2) -- (0,-0.4);
\draw (0.2,-0.2) -- (0.4,-0.4);
\draw (0.4,-0.4) -- (0.2,-0.6);
\draw (0.4,-0.4) -- (0.6,-0.6);
\end{tikzpicture} 
\end{align*}
We are now left with two terms that look the same. 
Algebraically, these are $(\tree \circ^C_2 \tree) \circ^C_1 \tree$ and $(\tree \circ^C_1 \tree) \circ^C_3 \tree$, which are described by the same tree, but the order of composition is opposite. We want our $det(T)$ to reflect this, but keep the signs given by the graded Leibniz rule of the other terms intact.

We define $det(T)$ as follows: let $T$ be a tree, and $ed(T)=\{ e_1, ..., e_{m-1} \}$ its internal edges (with $m$ the number of vertices of a tree), then we define $k^{ed(T)}$ to be the vector space spanned by $ed(T)$. Then $det(T)$ is the one-dimensional space spanned by the top exterior power of $k^{ed(T)}$, $e_1\wedge ... \wedge e_{m-1}$. So a basis element of $Cob(\mP)(n)^i$ looks like $T\otimes e_{\sigma(1)} \wedge ... \wedge e_{\sigma(i-1)}$ where $T$ is a decorated tree and $\sigma \in S_i$. If $i=1$, there are no internal edges, and $det(T)=k$, which doesn't add anything. 

We now need to define how the composition and the tree-composition work in terms of these wedge-products. For this, we will always put the active edge (the new one, or the one that is contracted) at the first place: we say that the edge that $\circ_T$ acts on must be the first in the wedge product, and if $T \otimes e_1 \wedge ... \wedge e_{m-1} $ and $T' \otimes f_1 \wedge ... \wedge f_{l-1}$ are trees in $Cob(\mP)$ we define their composition as 
$$ (T \otimes e_1 \wedge ... \wedge e_{m-1}) \circ^C_i (T' \otimes f_1 \wedge ... \wedge f_{l-1}) = 
T \circ^{Free}_i T' \otimes e \wedge e_1 \wedge ... \wedge e_{m-1} \wedge f_1 \wedge ... \wedge f_{l-1}$$
where $\circ^{Free}_i$ pastes the two trees along the $i$th edge of $T$ and $e$ is the new edge gained this way. This means that our $d$, which adds one edge to the tree, will put this new edge at the first place of our exterior product. So for example, (again in $\mAss$) 
$$
\begin{tikzpicture} [scale=1.5]
\draw (0,0.25) -- (0,0) node [left] {\tiny{$e_1$}};
\draw (0,0) -- (-0.2,-0.2);
\draw (0,0) -- (0,-0.2);
\draw (0,0) -- (0.2,-0.2);
\draw (-0.2,-0.2) -- (-0.4,-0.4);
\draw (-0.2,-0.2) -- (0,-0.4);
\end{tikzpicture} \otimes e_1
\stackrel{d}{\rightarrow}
\begin{tikzpicture} [scale=1.5]
\draw (0,0.25) -- (0,0) node [left] {\tiny{$e_2$}};
\draw (0,0) -- (-0.2,-0.2) node [left] {\tiny{$e_1$}};
\draw (0,0) -- (0.2,-0.2);
\draw (-0.2,-0.2) -- (-0.4,-0.4);
\draw (-0.2,-0.2) -- (0,-0.4);
\draw (-0.4,-0.4) -- (-0.6,-0.6);
\draw (-0.4,-0.4) -- (-0.2,-0.6);
\end{tikzpicture} \otimes e_2 \wedge e_1 + 
\begin{tikzpicture} [scale=1.5]
\draw (0,0.25) -- (0,0) node [left] {\tiny{$e_1$}};
\draw (0,0) node [right] {\tiny{$e_2$}} -- (-0.2,-0.2);
\draw (0,0) -- (0.2,-0.2);
\draw (-0.2,-0.2) -- (-0.4,-0.4);
\draw (-0.2,-0.2) -- (-0.1,-0.4);
\draw (0.2,-0.2) -- (0.1,-0.4);
\draw (0.2,-0.2) -- (0.4,-0.4);
\end{tikzpicture} \otimes e_2 \wedge e_1
$$
If we want to calculate the differential of a tree in our complex without using wedge products, we first have to choose some way of labelling the internal edges of all possible trees. Then we have to label the $m-1$ internal edges of the tree in that fixed way, apply $d$ as defined above, labeling our new edge with $1$ and adding 1 to all the old labels, and then we relabel our tree to the chosen labeling, which adds a minus sign (if the relabeling is an odd permutation) or not (if it's even).\footnote{There is nothing mysterious or canonical about choosing a standard way of labeling, and the choice does not have any influence on our contruction; it's just an (arbitrary) choice of calling some trees negative. This also means that if we choose a standard way of labeling, we can forget to label edges in our computations. We will use the conventinon that we label by level, from left to right: we first label all edges coming from the top vertex, then the edges from the vertex that has edge 1 as output, then those of the vertex on edge 2, and so on, all from left to right. So for example, we'd label like this: 
$$
\begin{tikzpicture} [scale=1.5]
\draw (0,0.3) -- (0,0);
\draw (0,0) -- (-0.4,-0.3);
\path (-0.12,-0.08) node [left] {\tiny{1}};
\draw (0,0) -- (0,-0.3);
\path (-0.06,-0.17) node {\tiny{2}};
\draw (0,0) -- (0.4,-0.3);
\path (0.12,-0.08) node [right] {\tiny{3}};
\draw (-0.4,-0.3) -- (-0.55,-0.5);
\draw (-0.4,-0.3) -- (-0.25,-0.5);
\path (-0.27,-0.36) node {\tiny{4}};
\draw (0,-0.3) -- (-0.15,-0.5); 
\draw (0,-0.3) -- (0.15,-0.5);
\path (0.13,-0.36) node {\tiny{5}};
\draw (0.4,-0.3) -- (0.25,-0.5);
\draw (0.4,-0.3) -- (0.55,-0.5);
\draw (-0.55,-0.5) -- (-0.7,-0.7);
\draw (-0.55,-0.5) -- (-0.4,-0.7);
\path (-0.42,-0.56) node {\tiny{6}};
\draw (0.15,-0.5) -- (0,-0.7);
\draw (0.15,-0.5) -- (0.3,-0.7);
\draw (-0.4,-0.7) -- (-0.55,-0.9);
\draw (-0.4,-0.7) -- (-0.25,-0.9);
\end{tikzpicture}
$$}

It is now easily shown that our $d$ satisfies $d^2=0$: first, we remark that each part of $d^2(T)$ has two new edges compared to $T$. Each tree that is created by $d^2$ is created in two ways: if $e$ and $f$ are our two new edges, we can either pull $e$ or $f$ outside first. In the first case, we get 
$T \otimes$ $f \wedge e \wedge e_1 \wedge ... \wedge e_{m-1}$, while in the second case we get 
$T \otimes$ $e \wedge f \wedge e_1 \wedge ... \wedge e_{m-1}$, and these terms cancel each other precisely. 

\begin{definition} We call the constructed complex $(Cob(\mP),d)$, the \textbf{cobar complex} of $\mP$.
\end{definition}

\subsection{The double cobar complex}

We can of course construct the cobar complex of the cobar complex of an operad, if we need to. However, to do this we first need to know how to construct the cobar complex of an arbitrary dg-operad. 

First, some conventions on the dual of a complex. If $V=(\{ V^i \}, d$) is a complex with $d$ degree 1, we define the dual complex 
$V^*:= (\{ (V^*)^i \}, d^*)$ where $(V^*)^i:=(V^{-i})^*$, which means that $d^*: \alpha \mapsto \alpha \circ d$ is still a differential of degree 1. Because $d^2=0$, $(d^*)^2: \alpha \mapsto \alpha \circ d^2 = 0$. Our convention on grades means that the dual of our cobar complex $Cob(\mP)(n)$ (which lived in degrees 1 to $n-1$) lives in degrees $-n+1$ to $-1$. 

If $\mP=\{\mP (n)^i\}$ is a dg-operad with differential $d_1$, the cobar complex of $\mP$ becomes a bicomplex $Cob(\mP) = \{ Cob(\mP)(n)^{i,j} \}$ with two degrees and two differentials: we have the dual of the first cobar diferential, $d_1^*: Cob(\mP)(n)^{i,j} \rightarrow Cob(\mP)(n)^{i+1,j}$, and the new cobar differential $d_2:Cob(\mP)(n)^{i,j} \rightarrow Cob(\mP)(n)^{i,j+1}$. We want to have one generalised degree such that our generalised differential $d=d_1^* + d_2$ makes sense (i.e. is of well-defined degree) so we define $C_\mP(n)^m := \bigoplus_{i+j=m} Cob(\mP)(n)^{i,j}$. For our new differential $d$ to indeed be a differential, we need $d_1^*d_2 = -d_2d_1^*$. 

We now see that the cobar complex of the cobar complex of an operad $\mP$ consists of trees with $j$ vertices and $n$ inputs, with each vertex $v_i$ decorated with an element of $(Cob(\mP)(m_i)^{k_i})^*$ with $\sum m_i = n+j-1$, $k_i<m_i$. The degree of this element is $j - \sum_{i=1}^j k_i$. Note that the element with the lowest possible degree now is the tree with one vertex decorated with an element with highest possible degree in $Cob(\mP)^*$, a binary tree, and has degree $1- (n-1) = -n+2$. 
An element with highest possible degree is at each vertex decorated with an element of lowest possible degree (1) in $Cob(\mP)$, and has degree $j- \sum_{i=1}^j 1 = 0$. 

There are several ways to visualise the elements of our double cobar complex. One way to do this is as trees, with a ``bubble" at each vertex. Inside the bubble is a tree whose vertices are decorated with elements from $\mP$. In this case, an element of $Cob(Cob(\mAss))(8)^{-1}$ would look like
$$
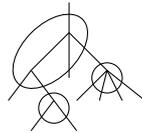
\begin{tikzpicture}
\draw (0,0.4) -- (0,0);
\draw (0,0) -- (-0.5,-0.5);
\draw (0,0) -- (0,-0.6);
\draw (0,0) -- (0.5,-0.5);
\draw (0.5,-0.6) circle (0.2);
\draw (-0.5,-0.5) -- (-0.8,-0.9);
\draw (-0.5,-0.5) -- (-0.2,-0.9);
\draw (-0.2,-1) circle (0.2);
\draw [rotate=45] (-0.35,0) ellipse (0.6 and 0.35);
\draw (0.5,-0.5) -- (0.1,-0.9);
\draw (0.5,-0.5) -- (0.4,-0.9);
\draw (0.5,-0.5) -- (0.7,-0.9);
\draw (0.5,-0.5) -- (1,-0.9);
\draw (-0.2,-0.9) -- (-0.5,-1.3);
\draw (-0.2,-0.9) -- (0.1,-1.3);
\end{tikzpicture}
$$
Another way to show these elements is by writing the edges of the ``outer" complex (the ones that would be outside of the bubbles in the notation above) as interrupted edges. Then, the same element would look like
$$
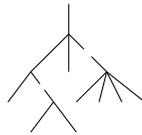
\begin{tikzpicture}
\draw (0,0.4) -- (0,0);
\draw (0,0) -- (-0.5,-0.5);
\draw (0,0) -- (0,-0.5);
\draw (0,0) -- (0.2,-0.2);
\draw (0.3,-0.3) -- (0.5,-0.5);
\draw (-0.5,-0.5) -- (-0.8,-0.9);
\draw (-0.5,-0.5) -- (-0.38,-0.66);
\draw (-0.32,-0.74) -- (-0.2,-0.9);
\draw (0.5,-0.5) -- (0.1,-0.9);
\draw (0.5,-0.5) -- (0.4,-0.9);
\draw (0.5,-0.5) -- (0.7,-0.9);
\draw (0.5,-0.5) -- (1,-0.9);
\draw (-0.2,-0.9) -- (-0.5,-1.3);
\draw (-0.2,-0.9) -- (0.1,-1.3);
\end{tikzpicture}
$$
A third way would be to label all edges with 1 (for the edges inside of the bubbles) and 2 (for the edges outside of the bubbles). Our element would then look like
$$
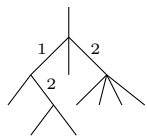
\begin{tikzpicture}
\draw (0,0.4) -- (0,0);
\draw (0,0) -- (-0.5,-0.5) -- (-0.8,-0.9);
\path (-0.15,-0.15) node [left] {\tiny{1}};
\draw (0,0) -- (0,-0.5);
\draw (0,0) -- (0.5,-0.5);
\path (0.15,-0.15) node [right] {\tiny{2}};
\draw (-0.5,-0.5) -- (-0.2,-0.9);
\path (-0.43,-0.62) node [right] {\tiny{2}};
\draw (0.5,-0.5) -- (0.1,-0.9);
\draw (0.5,-0.5) -- (0.4,-0.9);
\draw (0.5,-0.5) -- (0.7,-0.9);
\draw (0.5,-0.5) -- (1,-0.9);
\draw (-0.2,-0.9) -- (-0.5,-1.3);
\draw (-0.2,-0.9) -- (0.1,-1.3);
\end{tikzpicture}
$$
(this labelling has nothing to do with the labelling we used to order the edges with our $det(T)$) \\
We now have to be very careful when mentioning trees and edges, as to which kind we are talking about. Abusing our earlier conventions, we will call the edges in bubbles ``internal" and the edges outside ``external".
Also, we now have three compositions and two tree-compositions, so we'll denote the (tree-)composition on $Cob(\mP)$ as $\circ_i^{C_1}$ ($\circ_T^\mP$), and the (tree-)composition on $Cob(Cob(\mP))$ as $\circ_i^{C_2}$ ($\circ_T^{C_1}$).

Now, let's look at how $d_1^*$ and $d_2$ act on our operad. As before, $d_2$ is the dual of $\circ^{C_1}_T$, so it ``cuts" a vertex (or in our case, a bubble) into two parts. 
But we now know exactly how the composition $\circ^{C_1}_i$ works: it's the pasting of two free elements, without any relations. In other words, it makes one external edge into an internal edge. So our differential $d_2$ does the reverse, changing an internal edge to an external edge. For example:
$$ 
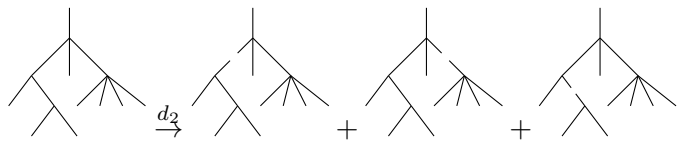
\begin{tikzpicture}
\draw (0,0.4) -- (0,0);
\draw (0,0) -- (-0.5,-0.5); 
\draw (0,0) -- (0,-0.5);
\draw (0,0) -- (0.5,-0.5); 
\draw (-0.5,-0.5) -- (-0.8,-0.9);
\draw (-0.5,-0.5) -- (-0.2,-0.9); 
\draw (0.5,-0.5) -- (0.1,-0.9);
\draw (0.5,-0.5) -- (0.4,-0.9);
\draw (0.5,-0.5) -- (0.7,-0.9);
\draw (0.5,-0.5) -- (1,-0.9);
\draw (-0.2,-0.9) -- (-0.5,-1.3);
\draw (-0.2,-0.9) -- (0.1,-1.3);
\end{tikzpicture} 
\stackrel{d_2}{\rightarrow} 
\begin{tikzpicture}
\draw (0,0.4) -- (0,0);
\draw (0,0) -- (-0.2,-0.2); 
\draw (-0.3,-0.3) -- (-0.5,-0.5);
\draw (0,0) -- (0,-0.5);
\draw (0,0) -- (0.5,-0.5); 
\draw (-0.5,-0.5) -- (-0.8,-0.9);
\draw (-0.5,-0.5) -- (-0.2,-0.9); 
\draw (0.5,-0.5) -- (0.1,-0.9);
\draw (0.5,-0.5) -- (0.4,-0.9);
\draw (0.5,-0.5) -- (0.7,-0.9);
\draw (0.5,-0.5) -- (1,-0.9);
\draw (-0.2,-0.9) -- (-0.5,-1.3);
\draw (-0.2,-0.9) -- (0.1,-1.3);
\end{tikzpicture} 
+ 
\begin{tikzpicture}
\draw (0,0.4) -- (0,0);
\draw (0,0) -- (-0.5,-0.5); 
\draw (0,0) -- (0,-0.5);
\draw (0,0) -- (0.2,-0.2); 
\draw (0.3,-0.3) -- (0.5,-0.5);
\draw (-0.5,-0.5) -- (-0.8,-0.9);
\draw (-0.5,-0.5) -- (-0.2,-0.9); 
\draw (0.5,-0.5) -- (0.1,-0.9);
\draw (0.5,-0.5) -- (0.4,-0.9);
\draw (0.5,-0.5) -- (0.7,-0.9);
\draw (0.5,-0.5) -- (1,-0.9);
\draw (-0.2,-0.9) -- (-0.5,-1.3);
\draw (-0.2,-0.9) -- (0.1,-1.3);
\end{tikzpicture} 
+ 
\begin{tikzpicture}
\draw (0,0.4) -- (0,0);
\draw (0,0) -- (-0.5,-0.5); 
\draw (0,0) -- (0,-0.5);
\draw (0,0) -- (0.5,-0.5); 
\draw (-0.5,-0.5) -- (-0.8,-0.9);
\draw (-0.5,-0.5) -- (-0.38,-0.66); 
\draw (-0.32,-0.74) -- (-0.2,-0.9);
\draw (0.5,-0.5) -- (0.1,-0.9);
\draw (0.5,-0.5) -- (0.4,-0.9);
\draw (0.5,-0.5) -- (0.7,-0.9);
\draw (0.5,-0.5) -- (1,-0.9);
\draw (-0.2,-0.9) -- (-0.5,-1.3);
\draw (-0.2,-0.9) -- (0.1,-1.3);
\end{tikzpicture}$$
Of course, the other differential is $d_1^*=(\circ_T^\mP {}^*)^*=\circ_T^\mP$, which contracts an internal edge. 
So both differentials remove one internal edge: $d_1^*$ contracts it and $d_2$ pulls it outside. 

For our differential to square to zero, we need $d_1^*d_2 = -d_2d_1^*$. This is true because both maps remove two internal edges, and for each pair of internal edges the resulting trees look the same, and the maps differ by a global minus sign. 
This is illustrated by the following example in $Cob(Cob(\mAss))(4)$: 
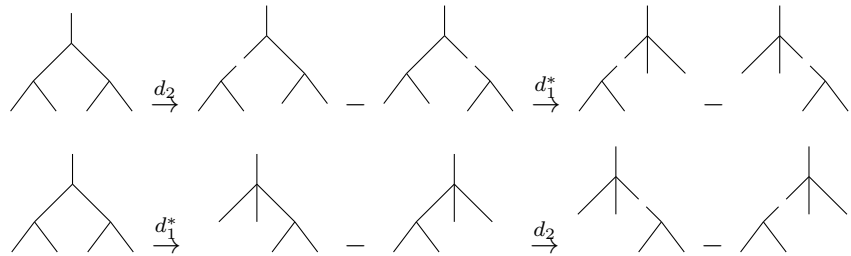
$$ \begin{tikzpicture}
\draw (0,0.4) -- (0,0);
\draw (0,0) -- (-0.5,-0.5);
\draw (0,0) -- (0.5,-0.5);
\draw (-0.5,-0.5) -- (-0.8,-0.9);
\draw (-0.5,-0.5) -- (-0.2,-0.9);
\draw (0.5,-0.5) -- (0.2,-0.9);
\draw (0.5,-0.5) -- (0.8,-0.9);
\end{tikzpicture} 
\ \stackrel{d_2}{\rightarrow} \ 
\begin{tikzpicture}
\draw (0,0.4) -- (0,0);
\draw (0,0) -- (-0.3,-0.3);
\draw (0,0) -- (0.5,-0.5);
\draw (-0.4,-0.4) -- (-0.6,-0.6);
\draw (-0.6,-0.6) -- (-0.9,-1);
\draw (-0.6,-0.6) -- (-0.3,-1);
\draw (0.5,-0.5) -- (0.2,-0.9);
\draw (0.5,-0.5) -- (0.8,-0.9);
\end{tikzpicture} \ - \ \begin{tikzpicture}
\draw (0,0.4) -- (0,0);
\draw (0,0) -- (-0.5,-0.5);
\draw (0,0) -- (0.3,-0.3);
\draw (-0.5,-0.5) -- (-0.8,-0.9);
\draw (-0.5,-0.5) -- (-0.2,-0.9);
\draw (0.4,-0.4) -- (0.6,-0.6);
\draw (0.6,-0.6) -- (0.3,-1);
\draw (0.6,-0.6) -- (0.9,-1);
\end{tikzpicture} 
\ \stackrel{d_1^*}{\rightarrow} \ 
\begin{tikzpicture}
\draw (0,0.4) -- (0,0);
\draw (0,0) -- (-0.3,-0.3);
\draw (0,0) -- (0,-0.5);
\draw (0,0) -- (0.5,-0.5);
\draw (-0.4,-0.4) -- (-0.6,-0.6);
\draw (-0.6,-0.6) -- (-0.9,-1);
\draw (-0.6,-0.6) -- (-0.3,-1);
\end{tikzpicture} \ - \ \begin{tikzpicture}
\draw (0,0.4) -- (0,0);
\draw (0,0) -- (-0.5,-0.5);
\draw (0,0) -- (0,-0.5);
\draw (0,0) -- (0.3,-0.3);
\draw (0.4,-0.4) -- (0.6,-0.6);
\draw (0.6,-0.6) -- (0.3,-1);
\draw (0.6,-0.6) -- (0.9,-1);
\end{tikzpicture} $$
$$ \begin{tikzpicture}
\draw (0,0.4) -- (0,0);
\draw (0,0) -- (-0.5,-0.5);
\draw (0,0) -- (0.5,-0.5);
\draw (-0.5,-0.5) -- (-0.8,-0.9);
\draw (-0.5,-0.5) -- (-0.2,-0.9);
\draw (0.5,-0.5) -- (0.2,-0.9);
\draw (0.5,-0.5) -- (0.8,-0.9);
\end{tikzpicture} 
\ \stackrel{d_1^*}{\rightarrow} \ \ \ 
\begin{tikzpicture}
\draw (0,0.4) -- (0,0);
\draw (0,0) -- (-0.5,-0.5);
\draw (0,0) -- (0,-0.5);
\draw (0,0) -- (0.5,-0.5);
\draw (0.5,-0.5) -- (0.2,-0.9);
\draw (0.5,-0.5) -- (0.8,-0.9);
\end{tikzpicture} \ \ - \ \ \begin{tikzpicture}
\draw (0,0.4) -- (0,0);
\draw (0,0) -- (-0.5,-0.5);
\draw (0,0) -- (0,-0.5);
\draw (0,0) -- (0.5,-0.5);
\draw (-0.5,-0.5) -- (-0.8,-0.9);
\draw (-0.5,-0.5) -- (-0.2,-0.9);
\end{tikzpicture} 
\ \ \ \stackrel{d_2}{\rightarrow} \ 
\begin{tikzpicture}
\draw (0,0.4) -- (0,0);
\draw (0,0) -- (-0.5,-0.5);
\draw (0,0) -- (0,-0.5);
\draw (0,0) -- (0.3,-0.3);
\draw (0.4,-0.4) -- (0.6,-0.6);
\draw (0.6,-0.6) -- (0.3,-1);
\draw (0.6,-0.6) -- (0.9,-1);
\end{tikzpicture} \ - \ \begin{tikzpicture}
\draw (0,0.4) -- (0,0);
\draw (0,0) -- (-0.3,-0.3);
\draw (0,0) -- (0,-0.5);
\draw (0,0) -- (0.5,-0.5);
\draw (-0.4,-0.4) -- (-0.6,-0.6);
\draw (-0.6,-0.6) -- (-0.9,-1);
\draw (-0.6,-0.6) -- (-0.3,-1);
\end{tikzpicture} $$
(note that we use $det(T)$ here for the signs)

We now arrive at our first result of resolutions: 

\begin{theorem} For any operad $\mP$, $H^*(Cob(Cob(\mP))) \cong \mP$. 
\end{theorem}
\begin{proof} While not hard, this proof requires the use of filtrations which will be explained in chapter 3, so we'll postpone the proof until then. 
\end{proof}

\subsection{Koszulness}

Although the theorem above gives us a surefire way of finding a resolution for an operad, the double cobar complex is an enormous construction for all but the most simple operads. And it turns out that for many quadratic operads, we need not go that far: 

\begin{definition} A \textbf{Koszul operad} is a quadratic operad $\mP$ such that $H^*(Cob(\mP)) \cong \mP^!$
\end{definition}

Luckily, the examples we reviewed are all Koszul operads: 

\begin{prop} $\mAss$, $\mCom$, $\mLie$ and $\mPoiss$ are all Koszul operads. 
\end{prop}

\begin{proof} In the interest of brevity, we will not give the proof here. We remark that the Koszulness of $\mAss$, $\mCom$ and $\mLie$ was proved in \cite{GK} using the Koszulness of the corresponding algebras. In \cite{Mar} it was shown that the Koszulness of $\mPoiss$ follows from the Koszulness of $\mCom$ and $\mLie$ and the distributive law in $\mPoiss$. 
\end{proof}

\chapter{Wheeled operads}

\section{Definitions and examples}

We now define an extention to the theory of operads, wheeled operads. Again, this structure occurs most naturally when discussing the operad endomorphisms, $Hom(V^{\otimes n},V)$, as the trace operator: 

\begin{definition} The $i$-trace of an element $\nu \in Hom(V^{\otimes n},V)$ is $$tr_i (\nu) = \sum_{j=1}^{dim(V)} \nu_j (...,e_j,...) \in Hom(V^{\otimes n-1},k))$$ with $e_j$ filled in on the $i$th spot, $\{e_1,...,e_{dim(V)}\}$ a basis for $V$, and for a vector $v$, $v_j = e_j^*(v)$. 
\end{definition}

From linear algebra, we know that this definition does not depend on the choice of basis $\{e_1,...,e_{dim(V)}\}$. We also know that the trace satisfies the cyclic property $ tr(a_{1} a_{2} ...  a_{n}) = tr(a_{2} ...  a_{n} a_{1}) $. 
The trace of a linear map is a map with one less input and an output in the ground field $k$, which we think of as $V^{\otimes 0}$. This inspires the following definition:

\begin{definition} A \textbf{wheeled operad} $\mP$ is an ordinary operad $\mP_o=\{\mP_1(n)|n\geq 1\}$ together with 
a collection of vector spaces $\mP_w = \{\mP_w(n)|n\geq 0\}$ and the following extra structure: 
\begin{enumerate}
	\item[i)] an action of $S_n$ on each $\mP_w(n)$; 
	\item[ii)] an extention of our composition: $\circ_i: \mP_w(n) \otimes \mP_o (m) \rightarrow \mP_w(n+m-1)$ for $n \geq 1$, with $1 \leq i \leq n$; 
	\item[iii)] a collection of maps, called \textbf{contractions}, $\xi_i: \mP_o(n) \rightarrow \mP_w(n-1)$ (with $1\leq i \leq n$) for each $n$ that satisfies the following cyclic property: $\xi_{i+j-1} (\nu \circ_i \mu) = \xi_{i+j-1} (\mu \circ_j \nu)$ (with $\nu \in \mP_o (n)$, $\mu \in \mP_o (m)$, $i\leq n$ and $j\leq m$). Note that this is unrelated to the contraction of an edge in a tree. 
\end{enumerate}
We call $\mP_o$ and $\mP_w$ the operadic and wheeled parts of $\mP$, respectively.
\end{definition}

As before, we can think of elements of our operad in terms of graphs. An element of $\mP_w$ is viewed as an element with no output: 
$$ \begin{tikzpicture} [scale=1.5]
\draw (0,0) node [above right= -4pt and -2pt] {\small{$\nu$}};
\draw (-0.3,-0.2) -- (0,0) -- (0.3,-0.2);
\draw (0,0) -- (-0.15,-0.2);
\draw (-0.075,-0.15) -- (0.18,-0.15) [dotted];
\end{tikzpicture} $$
We call trees without an output outputless trees (so a ``tree" will always have an output). 
The $i$-contraction operator adds a loop from the output to the $i$th input (while relabeling inputs $\{i+1,...,n\}$ to $\{i,...,n-1\}$): 
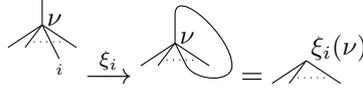
$$\begin{tikzpicture} [scale=1.5]
\draw (0,0.25) -- (0,0) node [above right= -4pt and -2pt] {\small{$\nu$}};
\draw (0,0) -- (-0.3,-0.2);
\draw (0,0) -- (-0.15,-0.2);
\draw (0,0) -- (0.15,-0.3) node [below=-2pt] {\tiny{$i$}};
\draw (-0.075,-0.15) -- (0.18,-0.15) [dotted];
\draw (0,0) -- (0.3,-0.2);
\end{tikzpicture} 
\stackrel{\xi_i}{\longrightarrow}
\begin{tikzpicture} [scale=1.5]
\draw (0,0) -- (-0.3,-0.2);
\draw (0,0) -- (-0.15,-0.2);
\draw (0,0) -- (0.1,-0.2) to [out=300,in=270] (0.5,-0.2) to [out=90,in=90] (0,0.25) -- (0,-0.00) node [above right= -4pt and -2pt] {\small{$\nu$}};
\draw (-0.075,-0.15) -- (0.18,-0.15) [dotted];
\draw (0,0) -- (0.3,-0.2);
\end{tikzpicture} = \begin{tikzpicture} [scale=1.5]
\draw (0,0) node [above right= -4pt and -2pt] {\small{$\xi_i(\nu)$}};
\draw (0,0) -- (-0.3,-0.2);
\draw (0,0) -- (-0.15,-0.2);
\draw (-0.075,-0.15) -- (0.18,-0.15) [dotted];
\draw (0,0) -- (0.3,-0.2);
\end{tikzpicture}$$
This indeed makes for a graph with one less input and no output. The graph is still directed: the loop is directed downwards. We call trees with such a loop \textbf{wheeled trees} (even though they aren't really trees). 
The rule for contractions of compositions makes sense in terms of these graphs: we can move the lowest element in a loop along, to make it the highest, and vice versa: 
$$
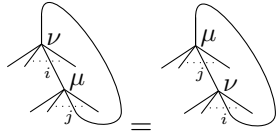
\begin{tikzpicture} [scale=1.5]
\draw (0,0.25) -- (0,0) node [above right= -4pt and -2pt] {\small{$\nu$}};
\draw (0,0) -- (-0.3,-0.2);
\draw (0,0) -- (-0.15,-0.2);
\draw (0,0) -- (0.2,-0.4) node [above right= -4pt and -2pt] {\small{$\mu$}};
\path (0.125,-0.25) node [left=-2pt] {\tiny{$i$}};
\draw (-0.075,-0.15) -- (0.1875,-0.15) [dotted];
\draw (0,0) -- (0.3,-0.2);
\draw (0.2,-0.4) -- (-0.1,-0.6);
\draw (0.2,-0.4) -- (0.05,-0.6);
\draw (0.2,-0.4) -- (0.3,-0.6) to [out=300,in=270] (0.7,-0.6) to [out=90,in=90] (0,0.25) -- (0,-0.00);
\path (0.325,-0.65) node [left=-2pt] {\tiny{$j$}};
\draw (0.125,-0.55) -- (0.3875,-0.55) [dotted];
\draw (0.2,-0.4) -- (0.5,-0.6);
\end{tikzpicture} = \begin{tikzpicture} [scale=1.5]
\draw (0,0.25) -- (0,0) node [above right= -4pt and -2pt] {\small{$\mu$}};
\draw (0,0) -- (-0.3,-0.2);
\draw (0,0) -- (-0.15,-0.2);
\draw (0,0) -- (0.2,-0.4) node [above right= -4pt and -2pt] {\small{$\nu$}};
\path (0.125,-0.25) node [left=-2pt] {\tiny{$j$}};
\draw (-0.075,-0.15) -- (0.1875,-0.15) [dotted];
\draw (0,0) -- (0.3,-0.2);
\draw (0.2,-0.4) -- (-0.1,-0.6);
\draw (0.2,-0.4) -- (0.05,-0.6);
\draw (0.2,-0.4) -- (0.3,-0.6) to [out=300,in=270] (0.7,-0.6) to [out=90,in=90] (0,0.25) -- (0,-0.00);
\path (0.325,-0.65) node [left=-2pt] {\tiny{$i$}};
\draw (0.125,-0.55) -- (0.3875,-0.55) [dotted];
\draw (0.2,-0.4) -- (0.5,-0.6);
\end{tikzpicture} $$ 
The composition is again viewed the contraction of internal edges of trees. Note that $\nu \circ_i \mu$ with $\mu\in\mP_w$ would not make sense in terms of our trees, and indeed isn't defined. 

Starting with an ordinary operad, we can create a wheeled one: 

\begin{definition} The \textbf{wheeled completion}, of an operad $\mP$ is $\mP\wh_c = \mP \oplus \xi(\mP)$, where $\xi$ satisfies no relations but the one mentioned above, $\xi_{i+j-1} (\nu \circ_i \mu) = \xi_{i+j-1} (\mu \circ_j \nu)$, and with a composition and $S_n$-action that are extentions of the ones on $\mP$. 
\end{definition}

We can now ``lift" some of the definitions of the previous chapter to the wheeled case. 

\begin{definition} Given a collection of $S_n$-modules, $E=\{E(n)\}$ and another collection of $S_n$-modules $F=\{F(n)\}$, we define the \textbf{free wheeled operad} over $E$ and $F$, $Free\wh(E,F)$, to be the wheeled operad $Free\wh(E,F) = Free\wh(E,F)_o \oplus Free\wh(E,F)_w$, where $Free\wh(E,F)_o(n) = Free(E)(n)$, and $Free\wh(E,F)_w$ is spanned by the collection of wheeled trees decorated with appropriate elements of $E$, and by the collection of outputless trees, where the top vertex is decorated with an element of $F$, and the rest of the vertices are decorated with elements of $E$. The wheeled trees again satisfy the cyclic property. 

As in the non-wheeled case, the composition pastes the output of a non-wheeled tree to one of the inputs of a (possibly wheeled or outputless) tree, and $S_n$ permutes the inputs. The contraction adds a wheel to a non-wheeled tree. We again demand that our elements are linear in each decoration. 
Note that if $F=0$, $Free\wh(E):=Free\wh(E,0)$ is the wheeled completion of $Free(E)$. 
\end{definition}

\begin{definition} An \textbf{ideal} $\mI$ of a wheeled operad $\mP$ consists of an operadic ideal $\mI_o = \{\mI_o(n)\} \subset \mP_o$ and a collection of subspaces $\mI_w (n) \subset \mP_w (n)$ such that: 
\begin{enumerate}
	\item[i)] If $\nu\in\mI_w(n)$, then $\sigma (\nu) \in \mI_w(n)$, with $\sigma\in S_n$;
	\item[ii)] If $\nu\in\mI_w(n)$ and $\mu\in\mP_o(m)$ then $\nu \circ_i \mu \in \mI_w(n+m-1)$, and if $\nu\in\mP_w(n)$ and $\mu\in\mI_o(m)$ then $\nu \circ_i \mu \in \mI_w(n+m-1)$;
	\item[iii)] If $\nu\in\mI_o(n)$ then $\xi_i(\nu)\in\mI_w(n-1)$. 
\end{enumerate}
\end{definition}

Again, given a subset $R \subset \mP$ we can define the ideal $<R> \subset \mP$, and quotienting out is well-defined because of our definition of an ideal. 

\begin{definition} Given an $S_2$-module $E$, and a subset $R = R_o \oplus R_w  \subset Free\wh_o(E)(3) \oplus Free\wh_w(E)(1)$, we define their \textbf{quadratic wheeled operad} $\mQ\wh(E,R)=Free\wh(E)/<R>$. Note that we don't have a second $S_n$-module $F$, so $Free\wh_w(E)$ consists only of wheeled trees, not outputless trees. 
\end{definition}

\begin{remark} If $R_w=\{0\}$, we can define the ordinary quadratic operad $\mQ(E,R)$. And likewise, if $\mQ(E,R)$ is a quadratic operad, we can define $Q\wh(E,R \oplus \{0\})$, which is exactly its wheeled completion. 
\end{remark}

For the definition of the wheeled quadratic dual, we need an identification of $Free\wh_w(E^\vee)(1)$ and $(Free\wh_w(E))^\vee(1)$. We use the following convention: 
$$ 
) = - \alpha (\nu)$$
Where $\alpha \in E^\vee$ and $\nu \in E$. 

\begin{definition} For a quadratic wheeled operad $\mP = \mQ\wh(E,R)$, we define its quadratic wheeled dual $\mP^! = \mQ\wh(E^\vee,Ann(R))$. 
\end{definition}

This means we can now look at the wheeled completion of the four quadratic operads we studied in the previous chapter, and their wheeled duals: 

\begin{example} We start with the wheeled associative operad, $\mAssw$. Of course, $\mAss\wh_o \cong \mAss$, so we need only explore what $\mAss\wh_w$ looks like. $\mAss(n)$ is spanned by $%
$, so a spanning element of of $\mAss\wh_w(n)$ looks like 
$$%
$$
The wheel-symmetry can be described as follows: we can represent our spanning element of $\mAss(n)$ by
$$%
$$
So on the left side of the wheeled edge we can cyclically permute our inputs. Of course, we have the same kind of symmetry on the right side. Other then this, there is no way to permute the order of inputs, and $R_w=\{0\}$, so 
$$dim(\mAssw_w(n)) = \sum_{i=0}^n (\stackrel{n}{i})\frac{i!}{i} \frac{(n-i)!}{n-i} = \sum_{i=0}^n (\stackrel{n}{i}) (i-1)! (n-i-1)! $$ 
(here, we define $(-1)!=1$). 

We know that $\mAss^!=\mAss$, so for $(\mAssw)^!$, we look at the same space, but now $R_w^! = Free\wh_w(1)$, the complete space. This means that 
$$%
$$
So we now consider wheeled trees where the loop is completely on the left or completely on the right to be zero. This means that 
$$dim(\mAssw_w(n)) = \sum_{i=1}^{n-1} (\stackrel{n}{i}) (i-1)! (n-i-1)! $$ 
\end{example}

\begin{example} If we now look at the wheeled commutative operad $\mComw$, we again know that $\mComw_o = \mCom$. Since $\mCom(n)$ is spanned by one element, $%
$, $\mComw_w(n)$ is also one-dimensional, spanned by 
$$%
$$

We know that $\mLie^!=\mCom$, and $%
 \in <R_w^!> = <Free\wh_w(1)>$, so the wheeled part of so $(\mLiew)^!$ is zero. 
\end{example}

\begin{example} We now turn to the wheeled Lie operad. We know that $\mLie(n)$ is spanned by 
$$%
$$
where we can always choose one input (in our case, $n$) to be at the last place. This means that if we add a loop, we can always choose it to be at the last place, so $\mLiew(n)$ is spanned by 
$$%
$$
By the cyclic property, we can again choose one of inputs with a chosen label (for example, we can choose the lowest input to always be $n$) and $dim(\mLiew_w(n)) = (n-1)!$. 

Now, $\mCom^! = \mLie$, and $%
 = 0$ doesn't add any new relations in $(\mComw)^!_w (n)$ (except in $(\mComw)^!_w (1)$ of course) because %
$$
So $(\mComw)^!_w (n) \cong \mLiew_w (n)$ if $n\geq 2$ and $(\mComw)^!_w (1) = 0$. 
\end{example}

\begin{example} Finally, we review the wheeled Poisson operad. $\mPoiss(n)$ is spanned by 
$$%
$$
and by the spanning elements of $\mComm$ and $\mLie$. Adding a wheel to an input of a commutative vertex in a tree that has a Lie vertex makes it zero:
$$ %
 = 0 $$
Because of this, and because $%
$ is commutative, $\mPoissw_w(n)$ is spanned by
$$%
$$
with at least one Lie element in the wheel, along with the trees that span $\mComw_w$ and $\mLiew_w$.  
However, the Poisson condition, $%
$ in combination with the cyclic property, makes for very complicated relations on $\mPoissw$, for example: 
$$ %
 $$
Because of this, we do not know how to compute the dimension of $\mPoissw_w(n)$.

Of course, $\mPoiss^! \cong \mPoiss$. $(\mPoissw)_w^!$ is $\mPoissw_w$ with extra relations $%
=0$. The first relation only implies that $%
 = 0$, because we already know that for example $%
 = 0$. We also know that $%
=0$, but the second relation is more influential than this might suggest: for example, it implies that
$$ %
 $$
Because of this, it's probably even harder to compute the dimension of $(\mPoissw)_w^!$. 
\end{example}

The definitions of dg-operads, resolutions and the Cobar complex for wheeled operads translate literally from their non-wheeled counterparts: 

\begin{definition} A \textbf{wheeled dg-operad} of degree 1 is a wheeled operad $\mP = \mP_o \oplus \mP_w$ with the following extra structure: 
\begin{enumerate}
\item[i)] $\mP_o$ is a dg-operad; 
\item[ii)] Each $\mP_w(n)$ is a graded vector space, and we denote the $i$th degree of this space by $\mP_w(n)^i$. The composition and contraction extend naturally to maps of graded vector spaces:
$\circ_i: \mP_w(n)^j \otimes \mP_o(m)^k \rightarrow \mP_w(n+m-1)^{j+k}$ and 
$\xi_i: \mP_o(n)^j \rightarrow \mP_w(n-1)^j$. 
\item[iii)] The differental is extended to $\mP_w$, such that it is still of degree 1 and $d^2=0$. It still satisfies (a modified version of) the Leibniz rule:
$$d(\nu\circ_i\mu)=d(\nu)\circ_i\mu+(-1)^j\nu\circ_i d(\mu)$$
where $\nu\in\mP_w(n)^j$. 
\end{enumerate}
\end{definition}

\begin{definition}A resolution of a wheeled operad $\mP$ is a free wheeled dg-operad $\mQ$ such that $H^*(\mQ) \cong \mP$. 
\end{definition}

\begin{definition} We define the \textbf{Cobar complex} of a wheeled operad $\mP$ to be $(Cob\wh(\mP),d)$, with 
$Cob\wh(\mP) = Free\wh(\mP_o(n)^*[-1],\mP_w(m)^* | n \geq 2, m \geq 1) \otimes det(T) \otimes sgn$. We now have two compositions, $\circ_i^\mP$ and $\circ_i^C$ and two contractions, $\xi_i^\mP$ and $\xi_i^C$. Our map $d$ is now an extention of the differential of the normal Cobar complex: it's $(\circ_T)^* + (\xi_T)^*$ where 
$$ 
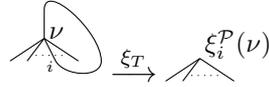
\begin{tikzpicture} [scale=1.5]
\draw (0,0) -- (-0.3,-0.2);
\draw (0,0) -- (-0.15,-0.2);
\draw (0,0) -- (0.1,-0.2) to [out=300,in=270] (0.5,-0.2) to [out=90,in=90] (0,0.25) -- (0,-0.00) node [above right= -4pt and -2pt] {\small{$\nu$}};
\path (0.125,-0.25) node [left=-2pt] {\tiny{$i$}};
\draw[scale=0.75] (-0.1,-0.2) -- (0.25,-0.2) [dotted];
\draw (0,0) -- (0.3,-0.2);
\end{tikzpicture} \stackrel{\xi_T}{\longrightarrow} \begin{tikzpicture} [scale=1.5]
\draw (0,0) node [above right= -4pt and -2pt] {\small{$\xi_i^\mP(\nu)$}};
\draw (0,0) -- (-0.3,-0.2);
\draw (0,0) -- (-0.15,-0.2);
\draw (-0.075,-0.15) -- (0.18,-0.15) [dotted];
\draw (0,0) -- (0.3,-0.2);
\end{tikzpicture} $$ 
So $(\xi_T)^*$ also adds an edge to the tree, which we again put at the first place of the wedge product. 
We note two things. Firstly, $\xi_T$ does not act on an element like $$
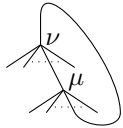
\begin{tikzpicture} [scale=1.5]
\draw (0,0) node [above right= -4pt and -2pt] {\small{$\nu$}};
\draw (0,0) -- (-0.3,-0.2);
\draw (0,0) -- (-0.15,-0.2);
\draw (0,0) -- (0.2,-0.4) node [above right= -4pt and -2pt] {\small{$\mu$}};
\draw (-0.075,-0.15) -- (0.1875,-0.15) [dotted];
\draw (0,0) -- (0.3,-0.2);
\draw (0.2,-0.4) -- (-0.1,-0.6);
\draw (0.2,-0.4) -- (0.05,-0.6);
\draw (0.2,-0.4) -- (0.3,-0.6) to [out=300,in=270] (0.7,-0.6) to [out=90,in=90] (0,0.25) -- (0,-0.00);
\draw (0.125,-0.55) -- (0.3875,-0.55) [dotted];
\draw (0.2,-0.4) -- (0.5,-0.6);
\end{tikzpicture}$$
but $\circ_T$ acts on both edges of this tree (though not at the same time).
Secondly, we note that $d$ is still of degree 1, because we didn't shift the degree of $\mP_w(m)^*$. 
\end{definition} 

\section{Wheeled Koszulness and main problem}

\begin{definition} A wheeled operad $\mP$ is called \textbf{wheeled Koszul} if $H^*(Cob(\mP)) \cong \mP^!$. 
\end{definition}

\begin{example} In \cite{MMS} and \cite{Mer} it was shown that $\mAssw$, $\mComw$ and $\mLiew$ are wheeled Koszul operads. But it wasn't known if this was a general feature of all Koszul operads. 
\end{example}

This leads us to our main question: 

\begin{main} Is the wheeled completion of an ordinary Koszul operad always a wheeled Koszul operad?
\end{main}

As it turns out, the answer to this question is ``no": in the next chapter, it is shown that in the case of the wheeled Poisson operad, a cycle appears that does not live in the highest degree (so the operad is not wheeled quadratic). The cycle is: 

$$
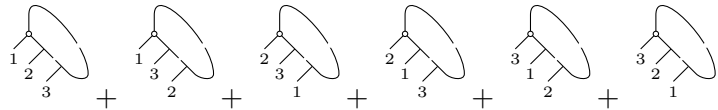
\begin{tikzpicture}
\draw (0,0) -- (-0.2,-0.2) node [below=-2pt] {\tiny{1}};
\draw (0,0) -- (0.2,-0.2); 
\draw (0.2,-0.2) -- (0.0,-0.4) node [below=-2pt] {\tiny{2}};
\draw (0.2,-0.2) -- (0.32,-0.32); 
\draw (0.35,-0.35) -- (0.43,-0.43);
\draw (0.43,-0.43) -- (0.23,-0.63) node [below=-2pt] {\tiny{3}}; 
\draw (0.43,-0.43) -- (0.53,-0.53) to [out=315,in=270] (0.78,-0.53) to [out=90,in=90] (0,0.25) -- (0,0);
\fill[white] (0.66,-0.15) circle (1pt);
\draw (0,-0) circle (1pt) [fill=white];
\end{tikzpicture} + \begin{tikzpicture}
\draw (0,0) -- (-0.2,-0.2) node [below=-2pt] {\tiny{1}};
\draw (0,0) -- (0.2,-0.2); 
\draw (0.2,-0.2) -- (0.0,-0.4) node [below=-2pt] {\tiny{3}};
\draw (0.2,-0.2) -- (0.32,-0.32); 
\draw (0.35,-0.35) -- (0.43,-0.43);
\draw (0.43,-0.43) -- (0.23,-0.63) node [below=-2pt] {\tiny{2}}; 
\draw (0.43,-0.43) -- (0.53,-0.53) to [out=315,in=270] (0.78,-0.53) to [out=90,in=90] (0,0.25) -- (0,0);
\fill[white] (0.66,-0.15) circle (1pt);
\draw (0,-0) circle (1pt) [fill=white];
\end{tikzpicture} + \begin{tikzpicture}
\draw (0,0) -- (-0.2,-0.2) node [below=-2pt] {\tiny{2}};
\draw (0,0) -- (0.2,-0.2); 
\draw (0.2,-0.2) -- (0.0,-0.4) node [below=-2pt] {\tiny{3}};
\draw (0.2,-0.2) -- (0.32,-0.32); 
\draw (0.35,-0.35) -- (0.43,-0.43);
\draw (0.43,-0.43) -- (0.23,-0.63) node [below=-2pt] {\tiny{1}}; 
\draw (0.43,-0.43) -- (0.53,-0.53) to [out=315,in=270] (0.78,-0.53) to [out=90,in=90] (0,0.25) -- (0,0);
\fill[white] (0.66,-0.15) circle (1pt);
\draw (0,-0) circle (1pt) [fill=white];
\end{tikzpicture} + \begin{tikzpicture}
\draw (0,0) -- (-0.2,-0.2) node [below=-2pt] {\tiny{2}};
\draw (0,0) -- (0.2,-0.2); 
\draw (0.2,-0.2) -- (0.0,-0.4) node [below=-2pt] {\tiny{1}};
\draw (0.2,-0.2) -- (0.32,-0.32); 
\draw (0.35,-0.35) -- (0.43,-0.43);
\draw (0.43,-0.43) -- (0.23,-0.63) node [below=-2pt] {\tiny{3}}; 
\draw (0.43,-0.43) -- (0.53,-0.53) to [out=315,in=270] (0.78,-0.53) to [out=90,in=90] (0,0.25) -- (0,0);
\fill[white] (0.66,-0.15) circle (1pt);
\draw (0,-0) circle (1pt) [fill=white];
\end{tikzpicture} + \begin{tikzpicture}
\draw (0,0) -- (-0.2,-0.2) node [below=-2pt] {\tiny{3}};
\draw (0,0) -- (0.2,-0.2); 
\draw (0.2,-0.2) -- (0.0,-0.4) node [below=-2pt] {\tiny{1}};
\draw (0.2,-0.2) -- (0.32,-0.32); 
\draw (0.35,-0.35) -- (0.43,-0.43);
\draw (0.43,-0.43) -- (0.23,-0.63) node [below=-2pt] {\tiny{2}}; 
\draw (0.43,-0.43) -- (0.53,-0.53) to [out=315,in=270] (0.78,-0.53) to [out=90,in=90] (0,0.25) -- (0,0);
\fill[white] (0.66,-0.15) circle (1pt);
\draw (0,-0) circle (1pt) [fill=white];
\end{tikzpicture} + \begin{tikzpicture}
\draw (0,0) -- (-0.2,-0.2) node [below=-2pt] {\tiny{3}};
\draw (0,0) -- (0.2,-0.2); 
\draw (0.2,-0.2) -- (0.0,-0.4) node [below=-2pt] {\tiny{2}};
\draw (0.2,-0.2) -- (0.32,-0.32); 
\draw (0.35,-0.35) -- (0.43,-0.43);
\draw (0.43,-0.43) -- (0.23,-0.63) node [below=-2pt] {\tiny{1}}; 
\draw (0.43,-0.43) -- (0.53,-0.53) to [out=315,in=270] (0.78,-0.53) to [out=90,in=90] (0,0.25) -- (0,0);
\fill[white] (0.66,-0.15) circle (1pt);
\draw (0,-0) circle (1pt) [fill=white];
\end{tikzpicture}$$
The notation will be explained in the next chapter.

\chapter{Computations and techniques}

In this chapter, we will compute some of the cohomology of $\mPoissw_w$ in the first few degrees, to check our main question. Before this, we introduce some techniques to ease our computations:

\section{Filtrations}

We start with a proposition that should be very clear:

\begin{prop} If $V$ is a complex and $V = \bigoplus_i V_i$ with $d(V_i) \subset V_i$ for each $i$, $H^*(V) = \bigoplus_i H^*(V_i)$. 
\end{prop}

This means that we can split a complex into the parts that have nothing to do with each other, and compute each cohomology seperately, to get the cohomology of the total complex. Our next tool is somewhat weaker, and yields weaker results:

\begin{definition} Given a complex of vector spaces $V$ with differential $d$, a \textbf{filtration} of $V$ is a sequence of subspaces $\{0\}=F_0 \subset F_1 \subset ... \subset F_n = V$ such that $d(F_i) \subset F_i$. This means that the induced differential $d_i: F_i / F_{i-1}$ and 
$H_i:= H^*(F_i / F_{i+1}, d_i)$ are well-defined. We call $i$ the \textbf{level} of a space $F_i$ or an element $x \in F_i$. 
\end{definition}

Some notation: if $\{0\}=F_0 \subset F_1 \subset ... \subset F_n = V$ is a filtration of $V$, we choose our basis 
$\{ e_{1_1}, ..., e_{1_{m_1}}, ..., e_{n_1}, ..., e_{n_{m_n}} \}$ such that $\{e_{1_1}, ..., e_{i_{m_i}} \}$ is a basis for $F_i$. This means that $\{\overline{e}_{i_1}, ..., \overline{e}_{i_{m_i}} \}$ is a basis for $F_i / F_{i-1}$, where $\overline{e}_{i_j} = e_{i_j} + F_{i-1}$, and we now have a canonical embedding $F_i / F_{i-1} \hookrightarrow F_i$. 
Given such a basis and an element $x \in V$, $x_i:= \sum_{j=i_1}^{i_{m_i}} e_j^* (x) e_j$ is the part of $x$ that lies in $F_i$ and not in $F_{i-1}$, and $\overline{x}_i = x_i + F_{i-1} \in F_i / F_{i-1}$. 

\begin{theorem} Let $\{0\}=F_0 \subset F_1 \subset ... \subset F_n = V$ be a filtration of $V$. If $x \in V$, $x \neq 0 \in H^k(V)$, there is an $x'$ and an $i$ such that $x \stackrel{/im}{=} x'$ and $\overline{x'}_i \neq 0 \in H_i^k$. 
\end{theorem}
\begin{proof} Let $j$ be the level of $x$: $x = \sum_{i=1}^j x_i$. We have $x \neq 0 \in H^k(V)$ so $d(x) = 0$ and $x \notin im(d)$. This means that $d(x_j) \in F_{j-1}$, because $d(x - x_j) = d (\sum_{i=1}^{j-1} x_i) \in F_{j-1}$. If $\overline{x}_j \neq 0 \in H_j$, we're done. If not, $\overline{x}_j = 0 \in H_j$ and $d(\overline{x}_j) = 0 \in F_j / F_{j-1}$ so $\overline{x}_j$ lies in the image of $d_j$: there exists a $y_j \in F_j^{k-1}$ such that $d_j(\overline{y}_j) = \overline{x}_j \in F_j / F_{j-1}$. 

In $F_j$, $d(y_j) = x_j + \tilde{x}_j$ with $\tilde{x}_j \in F_{j-1}$. We see that in $H^k(V)$, modulo the image of $d$, $x = \sum_{i=1}^j x_i \stackrel{/im}{=} \sum_{i=1}^{j-1} x_i - \tilde{x}_j \in F_{j-1}$. So modulo the image, we can lower the level of $x$ by one. 
With induction to $j$, we now see that there is either an $i$ such that $x'_i \neq 0 \in H_i$ (where $x'$ is $x$, modified by an element of the image), or we arrive at the lowest level where $x = d(y) + x'$ with $y = \sum_i y_i$ and $x' \in F_1$. 

If $\overline{x'} = 0 \in H_1^k$, $x'=d(y')$ for a $y' \in F_1^{k-1}$, because $d(x') = d(x) - d(d(y')) = 0$. But this would mean that $x = d(y) + d(y') \in im(d)$, so $x = 0 \in H^K$, which we excluded. We see that $\overline{x'} \neq 0 \in H_1^k$, with $x \stackrel{/im}{=} x'$. 
\end{proof}

The upshot of this theorem is that we only need to look for nontrivial cycles in the cohomology of the filtrations, which are often much easier to compute. However, the proof also shows that we will in general need to modify (i.e. add terms from other levels to) the cycles to get an element from the cohomology of $V$. 
The important corollary is: 

\begin{corollary} \begin{enumerate}
	\item[i)] Two elements $\overline{x} \neq \overline{y} \in H^k$ with $\overline{x}, \overline{y} \neq 0$ have different representants in at least one of the $H_i^k$. 
	\item[ii)] $\sum_i dim(H_i^k) \geq H^k$. 
	\item[iii)] If $H_i^k = 0$ for all $i$, $H^k(V) = 0$. 
\end{enumerate}
\end{corollary}
\begin{proof} For (i), note that if $\overline{x}_i = \overline{y}_i$ for all $i$, $(\overline{x-y})_i = 0 \in H_i^k$ while $\overline{x-y} \neq 0 \in H^k$ which is clearly in contradiction to the theorem. Parts (ii) and (iii) follow. 
\end{proof}

\begin{remark} All of these techniques fall into the broader framework of spectral sequences, but because our complexes are (relatively) simple, we will not expand on it here.
\end{remark}

With our notion of filtrations, we can prove our theorem on the double cobar complex: 

\begin{theorem} For any operad $\mP$, $H^*(Cob(Cob(\mP))) \cong \mP$. 
\end{theorem}
\begin{proof} Recall that our differential was $d = d_1^* + d_2$, and that both $d_1^*$ and $d_2$ remove an internal edge (a ``free edge" from the first cobar complex, inside the bubble): either by contracting it with $d_1^* = \circ_T^\mP$ or making it into an external edge with $d_2$. 

We start by splitting our complex: $d$ doesn't add or remove any inputs, so we just need to check that $H^*(Cob(Cob(\mP)))(n) \cong \mP(n)$. Recall that we really had a bicomplex, looking like this:
$$ \begin{tikzpicture}
  \matrix (m) [matrix of math nodes, row sep=2em,
    column sep=2em]
    {  & ... & ... &  \\
      ... & Cob(Cob(\mP))(n)^{i,j} & Cob(Cob(\mP))(n)^{i+1,j} & ... \\
      ... & Cob(Cob(\mP))(n)^{i,j+1} & Cob(Cob(\mP))(n)^{i+1,j+1} & ... \\ 
       & ... & ... &  \\ };
   {[start chain] \chainin (m-2-1);
   \chainin (m-2-2);
   \chainin (m-2-3) [join={node[above]{$d_1^*$}}];
   \chainin (m-2-4); }
   {[start chain] \chainin (m-3-1);
   \chainin (m-3-2);
   \chainin (m-3-3) [join={node[above]{$d_1^*$}}];
   \chainin (m-3-4); }
   {[start chain] \chainin (m-1-2);
   \chainin (m-2-2);
   \chainin (m-3-2) [join={node[right]{$d_2$}}];
   \chainin (m-4-2); }
   {[start chain] \chainin (m-1-3);
   \chainin (m-2-3);
   \chainin (m-3-3) [join={node[right]{$d_2$}}];
   \chainin (m-4-3); }
\end{tikzpicture} $$
And our generalized degree was a diagonal in this picture. However, our technique of filtrations allows us to filter the diagram in another way: because $d (Cob(Cob(\mP))(n)^{i,j} \subset Cob(Cob(\mP))(n)^{i+1,j} \oplus Cob(Cob(\mP))(n)^{i,j+1}$, $F_l := \bigoplus_{i \geq l, j} (Cob(Cob(\mP))(n)^{i,j}$ is a filtration with $d_2$ as induced differential (we could say $F_l := \bigoplus_{i, j \geq l} (Cob(Cob(\mP))(n)^{i,j}$, but we understand $d_2$ a lot better). So we ``cut" the complex: in our diagram, an $F_l$ consists of all spaces to the right of a vertical line. 

We can now split our (filtered) complex even further: remember that $d_2$ changed the edges of the trees that spanned $Cob(Cob(\mP))$, by turning an edge labelled with 1 into an edge labelled with 2. So the shape of the tree is unchanged, and we can split $F_l$ into different parts, each spanned by one tree. 

Say that this tree has one internal edge. In this case, our (splitted, filtered and then splitted again) complex looks like this: 
$$ 
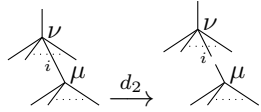
\begin{tikzpicture} [scale=1.5]
\draw (0,0.25) -- (0,0) node [above right= -4pt and -2pt] {\small{$\nu$}};
\draw (0,0) -- (-0.3,-0.2);
\draw (0,0) -- (-0.15,-0.2);
\draw (0,0) -- (0.2,-0.4) node [above right= -4pt and -2pt] {\small{$\mu$}};
\path (0.125,-0.25) node [left=-2pt] {\tiny{$i$}};
\draw (-0.075,-0.15) -- (0.1875,-0.15) [dotted];
\draw (0,0) -- (0.3,-0.2);
\draw (0.2,-0.4) -- (-0.1,-0.6);
\draw (0.2,-0.4) -- (0.05,-0.6);
\draw (0.125,-0.55) -- (0.3875,-0.55) [dotted];
\draw (0.2,-0.4) -- (0.5,-0.6);
\end{tikzpicture} \stackrel{d_2}{\longrightarrow} \begin{tikzpicture} [scale=1.5]
\draw (0,0.25) -- (0,0) node [above right= -4pt and -2pt] {\small{$\nu$}};
\draw (0,0) -- (-0.3,-0.2);
\draw (0,0) -- (-0.15,-0.2);
\draw (0,0) -- (0.12,-0.24);
\path (0.125,-0.25) node [left=-2pt] {\tiny{$i$}};
\draw (-0.075,-0.15) -- (0.1875,-0.15) [dotted];
\draw (0,0) -- (0.3,-0.2);
{[xshift=1.0,yshift=-2.0]
\draw (0.12,-0.24) -- (0.2,-0.4) node [above right= -4pt and -2pt] {\small{$\mu$}};
\draw (0.2,-0.4) -- (-0.1,-0.6);
\draw (0.2,-0.4) -- (0.05,-0.6);
\draw (0.125,-0.55) -- (0.3875,-0.55) [dotted];
\draw (0.2,-0.4) -- (0.5,-0.6);
}
\end{tikzpicture}$$
With $\nu, \mu \in \mP$. It is clear that the cohomology of this complex is zero. For larger complexes, let us choose a different notation because we only care for the labels of edges: we denote a tree with two internal edges, the first one labelled with 1 and the second one labelled with 2 by $\{1,2\}$. 
So the complex of a tree with two internal edges looks like this:
$$ 
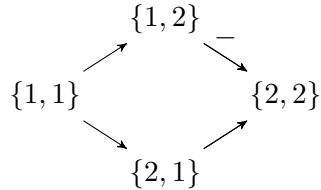
\begin{tikzpicture}
  \matrix (m) [matrix of math nodes, row sep=1em,
    column sep=1em]
    {  & \{1,2\} &  \\
      \{1,1\} &   & \{2,2\} \\
       & \{2,1\} &  \\  };
   {[start chain] 
   \chainin (m-2-1);
   \chainin (m-1-2);
   \chainin (m-2-3) [join={node[above]{$-$}}]; }
   {[start chain] 
   \chainin (m-2-1);
   \chainin (m-3-2);
   \chainin (m-2-3) [join={node[above]{$ $}}]; }
\end{tikzpicture} $$
Again, it is clear that the cohomology here is zero. However, we note something else: we can filter the complex by choosing the first edge to be fixed: $F_1 := \{ \{2,1\}, \{2,2\}\}$ and $F_2$ is everything. This is again a filtration ($d_2$ only changes edges from 1 to 2) and $F_1$ and $F_2 / F_1$ look exactly like the first complex we reviewed. This is a general pattern: if we look at a complex $C_m$ of trees with $m$ internal edges, $F_1 := \{ \{2,...\} \}$ and $F_2 := \{ \{1,...\}, \{2,...\} \}$ is a filtration and $F_1 \cong F_2/F_1 \cong C_{m-1}$. So with induction to the number of edges, and because our first complex had trivial cohomology, we see that all complexes spanned by trees with one or more internal edges have trivial cohomology. 

The only elements that we are left with are trees without internal edges. These are our candidates for the cohomology of the whole complex. Note that these trees are our basic trees, labelled with elements from $\mP$ (here we implicitly use our identification of $Free(\mP^*)$ and $Free(\mP)^*$ and of course that $(\mP^*)^* = \mP$). 
Let us look if they lie in the image of $d$. An element \begin{tikzpicture}
\draw (0,0.25) -- (0,-0.00) node [above right= -4pt and -2pt] {\small{$\nu$}};
\draw (0,0) -- (-0.3,-0.2);
\draw (0,0) -- (-0.15,-0.2);
\draw (-0.10,-0.2) -- (0.25,-0.2) [dotted];
\draw (0,0) -- (0.3,-0.2);
\end{tikzpicture} lies in $Cob(Cob(\mP))(n)^{-1,1}$. Because $Cob(Cob(\mP))(n)^{-1,0} = 0$, we only need to look at $Cob(Cob(\mP))(n)^{-2,1}$, which is mapped to $Cob(Cob(\mP))(n)^{-1,1}$ and $Cob(Cob(\mP))(n)^{-2,2}$. So locally, our complex looks like
$$ 
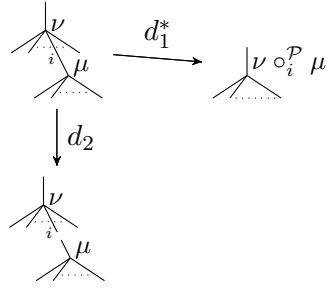
\begin{tikzpicture}
  \matrix (m) [matrix of math nodes, row sep=2em,
    column sep=3em]
    { \begin{tikzpicture} [scale=1.5]
\draw (0,0.25) -- (0,0) node [above right= -4pt and -2pt] {\small{$\nu$}};
\draw (0,0) -- (-0.3,-0.2);
\draw (0,0) -- (-0.15,-0.2);
\draw (0,0) -- (0.2,-0.4) node [above right= -4pt and -2pt] {\small{$\mu$}};
\path (0.125,-0.25) node [left=-2pt] {\tiny{$i$}};
\draw (-0.075,-0.15) -- (0.1875,-0.15) [dotted];
\draw (0,0) -- (0.3,-0.2);
\draw (0.2,-0.4) -- (-0.1,-0.6);
\draw (0.2,-0.4) -- (0.05,-0.6);
\draw (0.125,-0.55) -- (0.3875,-0.55) [dotted];
\draw (0.2,-0.4) -- (0.5,-0.6);
\end{tikzpicture} & \begin{tikzpicture} [scale=1.5]
\draw (0,0.25) -- (0,-0.00) node [above right= -4pt and -2pt] {\small{$\nu \circ_i^\mP \mu$}};
\draw (0,0) -- (-0.3,-0.2);
\draw (0,0) -- (-0.15,-0.2);
\draw (-0.10,-0.2) -- (0.25,-0.2) [dotted];
\draw (0,0) -- (0.3,-0.2);
\end{tikzpicture} \\
      \begin{tikzpicture} [scale=1.5]
\draw (0,0.25) -- (0,0) node [above right= -4pt and -2pt] {\small{$\nu$}};
\draw (0,0) -- (-0.3,-0.2);
\draw (0,0) -- (-0.15,-0.2);
\draw (0,0) -- (0.12,-0.24);
\path (0.125,-0.25) node [left=-2pt] {\tiny{$i$}};
\draw (-0.075,-0.15) -- (0.1875,-0.15) [dotted];
\draw (0,0) -- (0.3,-0.2);
{[xshift=1.0,yshift=-2.0]
\draw (0.12,-0.24) -- (0.2,-0.4) node [above right= -4pt and -2pt] {\small{$\mu$}};
\draw (0.2,-0.4) -- (-0.1,-0.6);
\draw (0.2,-0.4) -- (0.05,-0.6);
\draw (0.125,-0.55) -- (0.3875,-0.55) [dotted];
\draw (0.2,-0.4) -- (0.5,-0.6);
}
\end{tikzpicture} &  \\ };
   {[start chain] 
   \chainin (m-1-1);
   \chainin (m-1-2) [join={node[above]{$d_1^*$}}]; }
   {[start chain] 
   \chainin (m-1-1);
   \chainin (m-2-1) [join={node[right]{$d_2$}}]; }
\end{tikzpicture} $$
We see that \begin{tikzpicture} [scale=1.5]
\draw (0,0.25) -- (0,-0.00) node [above right= -4pt and -2pt] {\small{$\nu \circ_i^\mP \mu$}};
\draw (0,0) -- (-0.3,-0.2);
\draw (0,0) -- (-0.15,-0.2);
\draw (-0.10,-0.2) -- (0.25,-0.2) [dotted];
\draw (0,0) -- (0.3,-0.2);
\end{tikzpicture} does not lie in the image of $d$ (if it would, we would see $\nu = \mu \circ_i^\mP \rho$ with $\mu \circ_i^{C_2} \rho = 0$ so $\mu, \rho = 0$) and modulo the image, 
$$\begin{tikzpicture} [scale=1.5]
\draw (0,0.25) -- (0,0) node [above right= -4pt and -2pt] {\small{$\nu$}};
\draw (0,0) -- (-0.3,-0.2);
\draw (0,0) -- (-0.15,-0.2);
\draw (0,0) -- (0.12,-0.24);
\path (0.125,-0.25) node [left=-2pt] {\tiny{$i$}};
\draw (-0.075,-0.15) -- (0.1875,-0.15) [dotted];
\draw (0,0) -- (0.3,-0.2);
{[xshift=1.0,yshift=-2.0]
\draw (0.12,-0.24) -- (0.2,-0.4) node [above right= -4pt and -2pt] {\small{$\mu$}};
\draw (0.2,-0.4) -- (-0.1,-0.6);
\draw (0.2,-0.4) -- (0.05,-0.6);
\draw (0.125,-0.55) -- (0.3875,-0.55) [dotted];
\draw (0.2,-0.4) -- (0.5,-0.6);}
\end{tikzpicture} = \begin{tikzpicture} [scale=1.5]
\draw (0,0.25) -- (0,-0.00) node [above right= -4pt and -2pt] {\small{$\nu \circ_i^\mP \mu$}};
\draw (0,0) -- (-0.3,-0.2);
\draw (0,0) -- (-0.15,-0.2);
\draw (-0.10,-0.2) -- (0.25,-0.2) [dotted];
\draw (0,0) -- (0.3,-0.2);
\end{tikzpicture} $$
which is exactly the relation of $\circ^\mP$ and $\circ^{C_2}$ that we want (remember that the composition on $Cob(Cob(\mP))$ is $\circ_i^{C_2}$ and not $\circ_i^\mP$). 

In terms of relations on our operad, we now see that if $\nu \circ_i^\mP \mu = \rho \circ_j^\mP \tau$, 
$$\begin{tikzpicture} [scale=1.5]
\draw (0,0.25) -- (0,0) node [above right= -4pt and -2pt] {\small{$\nu$}};
\draw (0,0) -- (-0.3,-0.2);
\draw (0,0) -- (-0.15,-0.2);
\draw (0,0) -- (0.12,-0.24);
\path (0.125,-0.25) node [left=-2pt] {\tiny{$i$}};
\draw (-0.075,-0.15) -- (0.1875,-0.15) [dotted];
\draw (0,0) -- (0.3,-0.2);
{[xshift=1.0,yshift=-2.0]
\draw (0.12,-0.24) -- (0.2,-0.4) node [above right= -4pt and -2pt] {\small{$\mu$}};
\draw (0.2,-0.4) -- (-0.1,-0.6);
\draw (0.2,-0.4) -- (0.05,-0.6);
\draw (0.125,-0.55) -- (0.3875,-0.55) [dotted];
\draw (0.2,-0.4) -- (0.5,-0.6);}
\end{tikzpicture} = \begin{tikzpicture} [scale=1.5]
\draw (0,0.25) -- (0,-0.00) node [above right= -4pt and -2pt] {\small{$\nu \circ_i^\mP \mu$}};
\draw (0,0) -- (-0.3,-0.2);
\draw (0,0) -- (-0.15,-0.2);
\draw (-0.10,-0.2) -- (0.25,-0.2) [dotted];
\draw (0,0) -- (0.3,-0.2);
\end{tikzpicture} = \begin{tikzpicture} [scale=1.5]
\draw (0,0.25) -- (0,-0.00) node [above right= -4pt and -2pt] {\small{$\rho \circ_j^\mP \tau$}};
\draw (0,0) -- (-0.3,-0.2);
\draw (0,0) -- (-0.15,-0.2);
\draw (-0.10,-0.2) -- (0.25,-0.2) [dotted];
\draw (0,0) -- (0.3,-0.2);
\end{tikzpicture} = \begin{tikzpicture} [scale=1.5]
\draw (0,0.25) -- (0,0) node [above right= -4pt and -2pt] {\small{$\rho$}};
\draw (0,0) -- (-0.3,-0.2);
\draw (0,0) -- (-0.15,-0.2);
\draw (0,0) -- (0.12,-0.24);
\path (0.125,-0.25) node [left=-2pt] {\tiny{$j$}};
\draw (-0.075,-0.15) -- (0.1875,-0.15) [dotted];
\draw (0,0) -- (0.3,-0.2);
{[xshift=1.0,yshift=-2.0]
\draw (0.12,-0.24) -- (0.2,-0.4) node [above right= -4pt and -2pt] {\small{$\tau$}};
\draw (0.2,-0.4) -- (-0.1,-0.6);
\draw (0.2,-0.4) -- (0.05,-0.6);
\draw (0.125,-0.55) -- (0.3875,-0.55) [dotted];
\draw (0.2,-0.4) -- (0.5,-0.6);}
\end{tikzpicture} $$
so all relations on $\mP$ translate into relations on $Cob(Cob(\mP))(n)^{-1,1}$. We indeed see that $H^*(Cob(Cob(\mP)))$ is exactly spanned by (trees decorated with) elements from $\mP$, with composition equal to that of $\mP$ and all relations from $\mP$. So $H^*(Cob(Cob(\mP))) \cong \mP$. 
\end{proof}

\section{Checking the Koszulness of $\mPoissw$
}

We now want to check the wheeled Koszulness of $\mPoissw$ at the first few $n$. For this, we need a suitable basis in terms of which we can compute the cohomology. 

\subsection{A basis for the first few $Cob(\mPoissw)(n)$
}

First, we write down the possible trees of $Cob\wh(\mP)_w(n)$ for the first few $n$ (we already know that $\mPoiss$ is a Koszul operad, so we don't need to check the operadic part of $\mPoissw$). We keep in mind that we should decorate each vertex, and that we are able to switch inputs of each element: if we write $\begin{tikzpicture}
\draw (0,0.25) -- (0,0);
\draw (0,0) -- (-0.5,-0.5);
\draw (0,0) -- (0.2,-0.2);
\draw (0,0) -- (0,-0.2);
\draw (-0.3,-0.3) -- (-0.1,-0.5);
\end{tikzpicture}$, we forget to write $\begin{tikzpicture}
\draw (0,0.25) -- (0,0);
\draw (0,0) -- (-0.2,-0.2);
\draw (0,0) -- (0.2,-0.2);
\draw (0,0) -- (0,-0.3);
\draw (0,-0.3) -- (-0.2,-0.5);
\draw (0,-0.3) -- (0.2,-0.5);
\end{tikzpicture}$ and $\begin{tikzpicture}
\draw (0,0.25) -- (0,0);
\draw (0,0) -- (0.5,-0.5);
\draw (0,0) -- (-0.2,-0.2);
\draw (0,0) -- (0,-0.2);
\draw (0.3,-0.3) -- (0.1,-0.5);
\end{tikzpicture}$. On the other hand, we do pay attention to the wheel: $\begin{tikzpicture}
\draw (0,0) -- (0.3,-0.3) to [out=315,in=270] (0.6,-0.3) to [out=90,in=90] (0,0.25) -- (0,-0.00);
\draw (0,0) -- (-0.2,-0.2);
\draw (0.2,-0.2) -- (0,-0.4);
\end{tikzpicture}$ and $\begin{tikzpicture}
\draw (0,0) -- (0.1,-0.1) to [out=315,in=270] (0.35,-0.1) to [out=90,in=90] (0,0.25) -- (0,-0.00);
\draw (0,0) -- (-0.2,-0.2);
\draw (-0.2,-0.2) -- (-0.4,-0.4);
\draw (-0.2,-0.2) -- (-0.0,-0.4);
\end{tikzpicture}$ do not represent the same elements. So we have the following possible trees: 

\begin{tabular}{r l}
$Cob\wh(\mP)_w(1)^0$ & \begin{tikzpicture}
\draw (0,0) node [shape=circle,draw,fill=black,inner sep=0pt,minimum size=1pt] {} -- (-0.2,-0.2);
\end{tikzpicture} \\
$Cob\wh(\mP)_w(1)^1$ & \begin{tikzpicture}
\draw (0,0) -- (0.1,-0.1) to [out=315,in=270] (0.35,-0.1) to [out=90,in=90] (0,0.25) -- (0,-0.00);
\draw (0,0) -- (-0.2,-0.2);
\end{tikzpicture} \\
\end{tabular}

\begin{tabular}{r l}
$Cob\wh(\mP)_w(2)^0$ & \begin{tikzpicture}
\draw (-0.2,-0.2) -- (0,0) node [shape=circle,draw,fill=black,inner sep=0pt,minimum size=1pt] {} -- (0.2,-0.2);
\end{tikzpicture} \\
$Cob\wh(\mP)_w(2)^1$ & \begin{tikzpicture}
\draw (0,0) node [shape=circle,draw,fill=black,inner sep=0pt,minimum size=1pt] {} -- (-0.2,-0.2);
\draw (-0.2,-0.2) -- (-0.4,-0.4);
\draw (-0.2,-0.2) -- (0,-0.4);
\end{tikzpicture}, \begin{tikzpicture}
\draw (0,0) -- (0.1,-0.1) to [out=315,in=270] (0.35,-0.1) to [out=90,in=90] (0,0.25) -- (0,-0.00);
\draw (0,0) -- (-0.2,-0.2);
\draw (0,0) -- (0,-0.2);
\end{tikzpicture} \\
$Cob\wh(\mP)_w(2)^2$ & \begin{tikzpicture}
\draw (0,0) -- (0.3,-0.3) to [out=315,in=270] (0.6,-0.3) to [out=90,in=90] (0,0.25) -- (0,-0.00);
\draw (0,0) -- (-0.2,-0.2);
\draw (0.2,-0.2) -- (0,-0.4);
\end{tikzpicture}, \begin{tikzpicture}
\draw (0,0) -- (0.1,-0.1) to [out=315,in=270] (0.35,-0.1) to [out=90,in=90] (0,0.25) -- (0,-0.00);
\draw (0,0) -- (-0.2,-0.2);
\draw (-0.2,-0.2) -- (-0.4,-0.4);
\draw (-0.2,-0.2) -- (-0.0,-0.4);
\end{tikzpicture} \\
\end{tabular}

\begin{tabular}{r l}
$Cob\wh(\mP)_w(3)^0$ & \begin{tikzpicture}
\draw (0,0) -- (0,-0.2);
\draw (-0.2,-0.2) -- (0,0) node [shape=circle,draw,fill=black,inner sep=0pt,minimum size=1pt] {} -- (0.2,-0.2);
\end{tikzpicture} \\
$Cob\wh(\mP)_w(3)^1$ & \begin{tikzpicture}
\draw (-0.2,-0.2) -- (0,0) node [shape=circle,draw,fill=black,inner sep=0pt,minimum size=1pt] {} -- (0.2,-0.2);
\draw (-0.2,-0.2) -- (-0.4,-0.4);
\draw (-0.2,-0.2) -- (0,-0.4);
\end{tikzpicture}, \begin{tikzpicture}
\draw (0,0) node [shape=circle,draw,fill=black,inner sep=0pt,minimum size=1pt] {} -- (-0.2,-0.2);
\draw (-0.2,-0.2) -- (-0.4,-0.4);
\draw (-0.2,-0.2) -- (-0.2,-0.4);
\draw (-0.2,-0.2) -- (0,-0.4);
\end{tikzpicture}, \begin{tikzpicture}
\draw (0,0) -- (-0.3,-0.2);
\draw (0,0) -- (-0.1,-0.2);
\draw (0,0) -- (0.1,-0.2);
\draw (0,0) -- (0.15,-0.1) to [out=330,in=270] (0.40,-0.1) to [out=90,in=90] (0,0.25) -- (0,-0.00);
\end{tikzpicture} \\
$Cob\wh(\mP)_w(3)^2$ & \begin{tikzpicture}
\draw (0,0) node [shape=circle,draw,fill=black,inner sep=0pt,minimum size=1pt] {} -- (-0.2,-0.2);
\draw (-0.2,-0.2) -- (-0.4,-0.4);
\draw (-0.2,-0.2) -- (0,-0.4);
\draw (0,-0.4) -- (-0.2,-0.6);
\draw (0,-0.4) -- (0.2,-0.6);
\end{tikzpicture}, \begin{tikzpicture}
\draw (0,0) -- (0.3,-0.3) to [out=315,in=270] (0.6,-0.3) to [out=90,in=90] (0,0.25) -- (0,-0.00);
\draw (0,0) -- (-0.2,-0.2);
\draw (0,0) -- (0,-0.2);
\draw (0.2,-0.2) -- (0,-0.4);
\end{tikzpicture}, \begin{tikzpicture}
\draw (0,0) -- (0.1,-0.1) to [out=315,in=270] (0.35,-0.1) to [out=90,in=90] (0,0.25) -- (0,-0.00);
\draw (0,0) -- (-0.2,-0.2);
\draw (0,0) -- (0,-0.2);
\draw (-0.2,-0.2) -- (-0.4,-0.4);
\draw (-0.2,-0.2) -- (0,-0.4);
\end{tikzpicture}, \begin{tikzpicture}
\draw (0,0) -- (0.1,-0.1) to [out=315,in=270] (0.35,-0.1) to [out=90,in=90] (0,0.25) -- (0,-0.00);
\draw (0,0) -- (-0.2,-0.2);
\draw (-0.2,-0.2) -- (-0.4,-0.4);
\draw (-0.2,-0.2) -- (-0.2,-0.4);
\draw (-0.2,-0.2) -- (0,-0.4);
\end{tikzpicture} \\
$Cob\wh(\mP)_w(3)^3$ & \begin{tikzpicture}
\draw (0,0) -- (0.5,-0.5) to [out=315,in=270] (0.75,-0.5) to [out=90,in=90] (0,0.25) -- (0,-0.00);
\draw (0,0) -- (-0.2,-0.2);
\draw (0.2,-0.2) -- (0,-0.4);
\draw (0.4,-0.4) -- (0.2,-0.6);
\end{tikzpicture}, \begin{tikzpicture}
\draw (0,0) -- (0.40,-0.40) to [out=315,in=270] (0.65,-0.40) to [out=90,in=90] (0,0.25) -- (0,-0.00);
\draw (0,0) -- (-0.3,-0.3) -- (-0.45,-0.45);
\draw (-0.3,-0.3) -- (-0.15,-0.45);
\draw (0.3,-0.3) -- (0.15,-0.45);
\end{tikzpicture}, \begin{tikzpicture}
\draw (0,0) -- (0.1,-0.1) to [out=315,in=270] (0.35,-0.1) to [out=90,in=90] (0,0.25) -- (0,-0.00);
\draw (0,0) -- (-0.2,-0.2);
\draw (-0.2,-0.2) -- (-0.4,-0.4);
\draw (-0.2,-0.2) -- (-0.0,-0.4);
\draw (0,-0.4) -- (-0.2,-0.6);
\draw (0,-0.4) -- (0.2,-0.6);
\end{tikzpicture} \\
\end{tabular}

Next, we need a basis of $\mPoissw(n)$ for the first few $n$. For brevity, instead of labelling our inputs we will write something like

\begin{tabular}{r l}
$\mLie_o(3)$ & $2 \times \begin{tikzpicture}
\draw (0,0.25) -- (0,0);
\draw (0,0) -- (-0.2,-0.2);
\draw (0,0) -- (0.2,-0.2);
\draw (0.2,-0.2) -- (0,-0.4);
\draw (0.2,-0.2) -- (0.4,-0.4);
\end{tikzpicture}$
\end{tabular} \\
which means that $\mLie_o(3)$ is spanned by two elements that look like $\begin{tikzpicture}
\draw (0,0.25) -- (0,0);
\draw (0,0) -- (-0.2,-0.2);
\draw (0,0) -- (0.2,-0.2);
\draw (0.2,-0.2) -- (0,-0.4);
\draw (0.2,-0.2) -- (0.4,-0.4);
\end{tikzpicture}$ (for example, $\begin{tikzpicture}
\draw (0,0.25) -- (0,0);
\draw (0,0) -- (-0.2,-0.2) node [below=-1pt] {\tiny{1}};
\draw (0,0) -- (0.2,-0.2);
\draw (0.2,-0.2) -- (0,-0.4) node [below=-1pt] {\tiny{2}};
\draw (0.2,-0.2) -- (0.4,-0.4) node [below=-1pt] {\tiny{3}};
\end{tikzpicture}$ and $\begin{tikzpicture}
\draw (0,0.25) -- (0,0);
\draw (0,0) -- (-0.2,-0.2) node [below=-1pt] {\tiny{2}};
\draw (0,0) -- (0.2,-0.2);
\draw (0.2,-0.2) -- (0,-0.4) node [below=-1pt] {\tiny{3}};
\draw (0.2,-0.2) -- (0.4,-0.4) node [below=-1pt] {\tiny{1}};
\end{tikzpicture}$). We write $\comtree$ instead of $1 \times \comtree$. 

\begin{tabular}{r l}
$\mPoissw_o(2)$ & \comtree, \tree\\
$\mPoissw_o(3)$ & \begin{tikzpicture}
\draw (0,0.25) -- (0,0);
\draw (0,0) -- (-0.2,-0.2);
\draw (0,0) -- (0,-0.2);
\draw (0,0) -- (0.2,-0.2);
\draw (0,0) circle (1pt) [fill=white];
\end{tikzpicture}, $3 \times \begin{tikzpicture}
\draw (0,0.25) -- (0,0);
\draw (0,0) -- (-0.2,-0.2);
\draw (0,0) -- (0.2,-0.2);
\draw (0.2,-0.2) -- (0,-0.4);
\draw (0.2,-0.2) -- (0.4,-0.4);
\draw (0,0) circle (1pt) [fill=white];
\end{tikzpicture}$, $2 \times \begin{tikzpicture}
\draw (0,0.25) -- (0,0);
\draw (0,0) -- (-0.2,-0.2);
\draw (0,0) -- (0.2,-0.2);
\draw (0.2,-0.2) -- (0,-0.4);
\draw (0.2,-0.2) -- (0.4,-0.4);
\end{tikzpicture}$ \\
$\mPoissw_o(4)$ & \begin{tikzpicture}
\draw (0,0) -- (-0.3,-0.2);
\draw (0,0) -- (-0.1,-0.2);
\draw (0,0) -- (0.1,-0.2);
\draw (0,0) -- (0.15,-0.1) to [out=330,in=270] (0.4,-0.1) to [out=90,in=90] (0,0.25) -- (0,-0.00);
\draw (0,0) circle (1pt) [fill=white];
\end{tikzpicture}, $6 \times \begin{tikzpicture}
\draw (0,0.25) -- (0,0);
\draw (0,0) -- (0.5,-0.5);
\draw (0,0) -- (-0.2,-0.2);
\draw (0,0) -- (0,-0.2);
\draw (0.3,-0.3) -- (0.1,-0.5);
\draw (0,0) circle (1pt) [fill=white];
\end{tikzpicture}$, $8 \times \begin{tikzpicture}
\draw (0,0.25) -- (0,0);
\draw (0,0) -- (0.6,-0.6);
\draw (0,0) -- (-0.2,-0.2);
\draw (0.2,-0.2) -- (0,-0.4);
\draw (0.4,-0.4) -- (0.2,-0.6);
\draw (0,0) circle (1pt) [fill=white];
\end{tikzpicture}$, $3 \times \begin{tikzpicture}
\draw (0,0.25) -- (0,0);
\draw (0,0) -- (0.45,-0.45);
\draw (0,0) -- (-0.3,-0.3) -- (-0.45,-0.45);
\draw (-0.3,-0.3) -- (-0.15,-0.45);
\draw (0.3,-0.3) -- (0.15,-0.45);
\draw (0,0) circle (1pt) [fill=white];
\end{tikzpicture}$, $6 \times \begin{tikzpicture}
\draw (0,0.25) -- (0,0);
\draw (0,0) -- (0.6,-0.6);
\draw (0,0) -- (-0.2,-0.2);
\draw (0.2,-0.2) -- (0,-0.4);
\draw (0.4,-0.4) -- (0.2,-0.6);
\end{tikzpicture}$ \\
\end{tabular}

\begin{tabular}{r l}
$\mPoissw_w(1)$ & \begin{tikzpicture}
\draw (0,0) -- (0.1,-0.1) to [out=315,in=270] (0.35,-0.1) to [out=90,in=90] (0,0.25) -- (0,0);
\draw (0,0) -- (-0.2,-0.2);
\draw (0,0) circle (1pt) [fill=white];
\end{tikzpicture}, \begin{tikzpicture}
\draw (0,0) -- (0.1,-0.1) to [out=315,in=270] (0.35,-0.1) to [out=90,in=90] (0,0.25) -- (0,0);
\draw (0,0) -- (-0.2,-0.2);
\end{tikzpicture} \\
$\mPoissw_w(2)$ & \begin{tikzpicture}
\draw (0,0) -- (0.1,-0.1) to [out=315,in=270] (0.35,-0.1) to [out=90,in=90] (0,0.25) -- (0,0);
\draw (0,0) -- (-0.2,-0.2);
\draw (0,0) -- (0,-0.2);
\draw (0,0) circle (1pt) [fill=white];
\end{tikzpicture}, $2 \times \begin{tikzpicture}
\draw (0,0) -- (0.3,-0.3) to [out=315,in=270] (0.6,-0.3) to [out=90,in=90] (0,0.25) -- (0,0);
\draw (0,0) -- (-0.2,-0.2);
\draw (0.2,-0.2) -- (0,-0.4);
\draw (0,0) circle (1pt) [fill=white];
\end{tikzpicture}$, \begin{tikzpicture}
\draw (0,0) -- (0.3,-0.3) to [out=315,in=270] (0.6,-0.3) to [out=90,in=90] (0,0.25) -- (0,0);
\draw (0,0) -- (-0.2,-0.2);
\draw (0.2,-0.2) -- (0,-0.4);
\end{tikzpicture} \\
$\mPoissw_w(3)$ & \begin{tikzpicture}
\draw (0,0) -- (-0.3,-0.2);
\draw (0,0) -- (-0.1,-0.2);
\draw (0,0) -- (0.1,-0.2);
\draw (0,0) -- (0.15,-0.1) to [out=330,in=270] (0.4,-0.1) to [out=90,in=90] (0,0.25) -- (0,0);
\draw (0,0) circle (1pt) [fill=white];
\end{tikzpicture}, $3 \times \begin{tikzpicture}
\draw (0,0) -- (0.4,-0.4) to [out=315,in=270] (0.65,-0.4) to [out=90,in=90] (0,0.25) -- (0,0);
\draw (0,0) -- (-0.2,-0.2);
\draw (0,0) -- (0,-0.2);
\draw (0.3,-0.3) -- (0.1,-0.5);
\draw (0,0) circle (1pt) [fill=white];
\end{tikzpicture}$, $3 \times \begin{tikzpicture}
\draw (0,0) -- (0.5,-0.5) to [out=315,in=270] (0.75,-0.5) to [out=90,in=90] (0,0.25) -- (0,0);
\draw (0,0) -- (-0.2,-0.2);
\draw (0.2,-0.2) -- (0,-0.4);
\draw (0.4,-0.4) -- (0.2,-0.6);
\draw (0,0) circle (1pt) [fill=white];
\end{tikzpicture}$, \begin{tikzpicture}
\draw (0,0) -- (0.40,-0.40) to [out=315,in=270] (0.65,-0.40) to [out=90,in=90] (0,0.25) -- (0,0);
\draw (0,0) -- (-0.3,-0.3) -- (-0.45,-0.45);
\draw (-0.3,-0.3) -- (-0.15,-0.45);
\draw (0.3,-0.3) -- (0.15,-0.45);
\draw (0,0) circle (1pt) [fill=white];
\end{tikzpicture}, $2 \times \begin{tikzpicture}
\draw (0,0) -- (0.5,-0.5) to [out=315,in=270] (0.75,-0.5) to [out=90,in=90] (0,0.25) -- (0,0);
\draw (0,0) -- (-0.2,-0.2);
\draw (0.2,-0.2) -- (0,-0.4);
\draw (0.4,-0.4) -- (0.2,-0.6);
\end{tikzpicture}$ \\
\end{tabular}

The basis for $\mPoissw_o$ is pretty straightforward. As said, $\mPoissw_w$ is harder, but we already saw in chapter 2 that
$$ \begin{tikzpicture}
\draw (0,0) -- (0.40,-0.40) to [out=315,in=270] (0.65,-0.40) to [out=90,in=90] (0,0.25) -- (0,-0.00);
\draw (0,0) -- (-0.3,-0.3) -- (-0.45,-0.45) node [below=-2pt] {\tiny{1}};
\draw (-0.3,-0.3) -- (-0.15,-0.45) node [below=-2pt] {\tiny{2}};
\draw (0.3,-0.3) -- (0.15,-0.45) node [below=-2pt] {\tiny{3}};
\draw (0,0) circle (1pt) [fill=white];
\end{tikzpicture} = \begin{tikzpicture}
\draw (0,0) -- (0.40,-0.40) to [out=315,in=270] (0.65,-0.40) to [out=90,in=90] (0,0.25) -- (0,-0.00);
\draw (0,0) -- (-0.3,-0.3) -- (-0.45,-0.45) node [below=-2pt] {\tiny{3}};
\draw (-0.3,-0.3) -- (-0.15,-0.45) node [below=-2pt] {\tiny{1}};
\draw (0.3,-0.3) -- (0.15,-0.45) node [below=-2pt] {\tiny{2}};
\draw (0,0) circle (1pt) [fill=white];
\end{tikzpicture} = \begin{tikzpicture}
\draw (0,0) -- (0.40,-0.40) to [out=315,in=270] (0.65,-0.40) to [out=90,in=90] (0,0.25) -- (0,-0.00);
\draw (0,0) -- (-0.3,-0.3) -- (-0.45,-0.45) node [below=-2pt] {\tiny{2}};
\draw (-0.3,-0.3) -- (-0.15,-0.45) node [below=-2pt] {\tiny{3}};
\draw (0.3,-0.3) -- (0.15,-0.45) node [below=-2pt] {\tiny{1}};
\draw (0,0) circle (1pt) [fill=white];
\end{tikzpicture}, \ \
\begin{tikzpicture}
\draw (0,0) -- (0.5,-0.5) to [out=315,in=270] (0.75,-0.5) to [out=90,in=90] (0,0.25) -- (0,-0.00);
\draw (0,0) -- (-0.2,-0.2) node [below=-2pt] {\tiny{1}};
\draw (0.2,-0.2) -- (0,-0.4) node [below=-2pt] {\tiny{3}};
\draw (0.4,-0.4) -- (0.2,-0.6) node [below=-2pt] {\tiny{2}};
\draw (0,0) circle (1pt) [fill=white];
\end{tikzpicture} = \begin{tikzpicture}
\draw (0,0) -- (0.5,-0.5) to [out=315,in=270] (0.75,-0.5) to [out=90,in=90] (0,0.25) -- (0,-0.00);
\draw (0,0) -- (-0.2,-0.2) node [below=-2pt] {\tiny{1}};
\draw (0.2,-0.2) -- (0,-0.4) node [below=-2pt] {\tiny{2}};
\draw (0.4,-0.4) -- (0.2,-0.6) node [below=-2pt] {\tiny{3}};
\draw (0,0) circle (1pt) [fill=white];
\end{tikzpicture} - \begin{tikzpicture}
\draw (0,0) -- (0.40,-0.40) to [out=315,in=270] (0.65,-0.40) to [out=90,in=90] (0,0.25) -- (0,-0.00);
\draw (0,0) -- (-0.3,-0.3) -- (-0.45,-0.45) node [below=-2pt] {\tiny{1}};
\draw (-0.3,-0.3) -- (-0.15,-0.45) node [below=-2pt] {\tiny{2}};
\draw (0.3,-0.3) -- (0.15,-0.45) node [below=-2pt] {\tiny{3}};
\draw (0,0) circle (1pt) [fill=white];
\end{tikzpicture}, \ \
\begin{tikzpicture}
\draw (0,0) -- (0.1,-0.1) to [out=315,in=270] (0.35,-0.1) to [out=90,in=90] (0,0.25) -- (0,-0.00);
\draw (0,0) -- (-0.2,-0.2) node [below left=-5pt and -3pt] {\tiny{1}};
\draw (0,0) -- (0,-0.2);
\draw (-0.2,-0.4) node [below=-2pt] {\tiny{2}} -- (0,-0.2) -- (0.2,-0.4) node [below=-2pt] {\tiny{3}};
\draw (0,0) circle (1pt) [fill=white];
\end{tikzpicture} = 0 $$
The only new information we used above is that in the proof for the last equation, the input labelled with 1 is inconsequential, so we also see that $\begin{tikzpicture}
\draw (0,0) -- (0.1,-0.1) to [out=315,in=270] (0.35,-0.1) to [out=90,in=90] (0,0.25) -- (0,0);
\draw (0,0) -- (-0.2,-0.2);
\draw (-0.2,-0.2) -- (-0.4,-0.4);
\draw (-0.2,-0.2) -- (-0.0,-0.4);
\draw (0,0) circle (1pt) [fill=white];
\end{tikzpicture} = 0$. 

Thirdly (because we want to check whether $H^*(Cob\wh(\mPoissw)) \cong (\mPoissw)^!$) we need a basis for $(\mPoissw)_w^!$. 
We saw that $(\mPoissw)^!$ is $\mPoissw$ with two extra relations: $\begin{tikzpicture}
\draw (0,0) -- (0.1,-0.1) to [out=315,in=270] (0.35,-0.1) to [out=90,in=90] (0,0.25) -- (0,-0.00);
\draw (0,0) -- (-0.2,-0.2);
\draw (0,0) circle (1pt) [fill=white];
\end{tikzpicture} = 0$ and $\begin{tikzpicture}
\draw (0,0) -- (0.1,-0.1) to [out=315,in=270] (0.35,-0.1) to [out=90,in=90] (0,0.25) -- (0,-0.00);
\draw (0,0) -- (-0.2,-0.2);
\end{tikzpicture}=0$. This means that for example 
$$ \begin{tikzpicture}
\draw (0,0) -- (0.1,-0.1) to [out=315,in=270] (0.35,-0.1) to [out=90,in=90] (0,0.25) -- (0,-0.00);
\draw (0,0) -- (-0.2,-0.2);
\draw (-0.2,-0.2) -- (-0.4,-0.4) node [below=-2pt] {\tiny{1}};
\draw (-0.2,-0.2) -- (-0.2,-0.4) node [below=-2pt] {\tiny{2}};
\draw (-0.2,-0.2) -- (-0.0,-0.4) node [below=-2pt] {\tiny{3}};
\draw (-0.2,-0.2) circle (1pt) [fill=white];
\end{tikzpicture} = \begin{tikzpicture}
\draw (0,0) -- (0.4,-0.4) to [out=315,in=270] (0.65,-0.4) to [out=90,in=90] (0,0.25) -- (0,0);
\draw (0,0) -- (-0.2,-0.2) node [below=-2pt] {\tiny{1}};
\draw (0,0) -- (0,-0.2) node [below=-2pt] {\tiny{2}};
\draw (0.3,-0.3) -- (0.1,-0.5) node [below=-2pt] {\tiny{3}};
\draw (0,0) circle (1pt) [fill=white];
\end{tikzpicture} + \begin{tikzpicture}
\draw (0,0) -- (0.4,-0.4) to [out=315,in=270] (0.65,-0.4) to [out=90,in=90] (0,0.25) -- (0,0);
\draw (0,0) -- (-0.2,-0.2) node [below=-2pt] {\tiny{2}};
\draw (0,0) -- (0,-0.2) node [below=-2pt] {\tiny{3}};
\draw (0.3,-0.3) -- (0.1,-0.5) node [below=-2pt] {\tiny{1}};
\draw (0,0) circle (1pt) [fill=white];
\end{tikzpicture} + \begin{tikzpicture}
\draw (0,0) -- (0.4,-0.4) to [out=315,in=270] (0.65,-0.4) to [out=90,in=90] (0,0.25) -- (0,0);
\draw (0,0) -- (-0.2,-0.2) node [below=-2pt] {\tiny{3}};
\draw (0,0) -- (0,-0.2) node [below=-2pt] {\tiny{1}};
\draw (0.3,-0.3) -- (0.1,-0.5) node [below=-2pt] {\tiny{2}};
\draw (0,0) circle (1pt) [fill=white];
\end{tikzpicture} = 0 $$
The basis for the first three $n$ of $(\mPoissw)_w^!(n)$  is: 

\begin{tabular}{r l}
$(\mPoissw)_w^!(1)$ & 0 \\
$(\mPoissw)_w^!(2)$ & \begin{tikzpicture}
\draw (0,0) -- (0.3,-0.3) to [out=315,in=270] (0.6,-0.3) to [out=90,in=90] (0,0.25) -- (0,0);
\draw (0,0) -- (-0.2,-0.2);
\draw (0.2,-0.2) -- (0,-0.4);
\draw (0,0) circle (1pt) [fill=white];
\end{tikzpicture}, \begin{tikzpicture}
\draw (0,0) -- (0.3,-0.3) to [out=315,in=270] (0.6,-0.3) to [out=90,in=90] (0,0.25) -- (0,0);
\draw (0,0) -- (-0.2,-0.2);
\draw (0.2,-0.2) -- (0,-0.4);
\end{tikzpicture} \\
$(\mPoissw)_w^!(3)$ & $2 \times \begin{tikzpicture}
\draw (0,0) -- (0.4,-0.4) to [out=315,in=270] (0.65,-0.4) to [out=90,in=90] (0,0.25) -- (0,0);
\draw (0,0) -- (-0.2,-0.2);
\draw (0,0) -- (0,-0.2);
\draw (0.3,-0.3) -- (0.1,-0.5);
\draw (0,0) circle (1pt) [fill=white];
\end{tikzpicture}$, $3 \times \begin{tikzpicture}
\draw (0,0) -- (0.5,-0.5) to [out=315,in=270] (0.75,-0.5) to [out=90,in=90] (0,0.25) -- (0,0);
\draw (0,0) -- (-0.2,-0.2);
\draw (0.2,-0.2) -- (0,-0.4);
\draw (0.4,-0.4) -- (0.2,-0.6);
\draw (0,0) circle (1pt) [fill=white];
\end{tikzpicture}$, $2 \times \begin{tikzpicture}
\draw (0,0) -- (0.5,-0.5) to [out=315,in=270] (0.75,-0.5) to [out=90,in=90] (0,0.25) -- (0,0);
\draw (0,0) -- (-0.2,-0.2);
\draw (0.2,-0.2) -- (0,-0.4);
\draw (0.4,-0.4) -- (0.2,-0.6);
\end{tikzpicture}$ \\
\end{tabular}

Before we can write down elements of $Cob\wh(\mPoissw)(n)$, we need a new notation: both elements from $\mPoissw$ and elements from $Cob\wh(\mP)$ look like trees, so we need a way to distinguish between their edges. We do this in the same way we handled the double cobar complex: we write elements from $\mPoissw$ as we used to, and we write the free edges from the cobar complex as broken edges. This means that $d$ will ``break" edges: for example, it will map $\begin{tikzpicture}
\draw (0,0.25) -- (0,0);
\draw (0,0) -- (-0.2,-0.2) node [below=-2pt] {\tiny{1}};
\draw (0,0) -- (0,-0.2) node [below=-2pt] {\tiny{2}};
\draw (0,0) -- (0.2,-0.2) node [below=-2pt] {\tiny{3}};
\draw (0,0) circle (1pt) [fill=white];
\end{tikzpicture}$ to 
$\begin{tikzpicture}
\draw (0,0.25) -- (0,0);
\draw (0,0) -- (-0.2,-0.2) node [below=-1pt] {\tiny{1}};
\draw (0,0) -- (0.12,-0.12);
\draw (0,0) circle (1pt) [fill=white];
{[xshift=1.0,yshift=-1.0]
\draw (0.12,-0.12) -- (0.2,-0.2);
\draw (0.2,-0.2) -- (0,-0.4) node [below=-1pt] {\tiny{2}};
\draw (0.2,-0.2) -- (0.4,-0.4) node [below=-1pt] {\tiny{3}};
\draw (0.2,-0.2) circle (1pt) [fill=white]; }
\end{tikzpicture} + \begin{tikzpicture}
\draw (0,0.25) -- (0,0);
\draw (0,0) -- (-0.2,-0.2) node [below=-1pt] {\tiny{2}};
\draw (0,0) -- (0.12,-0.12);
\draw (0,0) circle (1pt) [fill=white];
{[xshift=1.0,yshift=-1.0]
\draw (0.12,-0.12) -- (0.2,-0.2);
\draw (0.2,-0.2) -- (0,-0.4) node [below=-1pt] {\tiny{3}};
\draw (0.2,-0.2) -- (0.4,-0.4) node [below=-1pt] {\tiny{1}};
\draw (0.2,-0.2) circle (1pt) [fill=white]; }
\end{tikzpicture} + \begin{tikzpicture}
\draw (0,0.25) -- (0,0);
\draw (0,0) -- (-0.2,-0.2) node [below=-1pt] {\tiny{3}};
\draw (0,0) -- (0.12,-0.12);
\draw (0,0) circle (1pt) [fill=white];
{[xshift=1.0,yshift=-1.0]
\draw (0.12,-0.12) -- (0.2,-0.2);
\draw (0.2,-0.2) -- (0,-0.4) node [below=-1pt] {\tiny{1}};
\draw (0.2,-0.2) -- (0.4,-0.4) node [below=-1pt] {\tiny{2}};
\draw (0.2,-0.2) circle (1pt) [fill=white]; }
\end{tikzpicture}$. 

With our new notation and the generic elements of $Cob\wh(\mP)_w$, we could now write down a basis for the first few $Cob\wh(\mPoissw)_w(n)$, just by filling in all vertices of $Cob\wh(\mP)_w$ with the appropriate elements of $\mPoiss\wh$. But these spaces get very big very quickly: $Cob\wh(\mPoiss)_w(2)$ is 22-dimensional, and the dimension of $Cob\wh(\mPoiss)_w(3)$ is over 200. So we choose to not write down the complete basis, and focus on the place where Koszulness breaks down. 

\subsection{A non-trivial cycle}

To compute the cohomology of the complex, we first note that the number of commutative and Lie vertices in a tree doesn't change under the relations on $\mPoiss$. So we can split each $Cob\wh(\mPoissw)_w(n)$ into $n+1$ different complexes, the first one with $n$ commutative vertices, the second with $n-1$ commutative vertices and one Lie vertex, and so on. 

Secondly, we will in fact ignore the $\otimes sgn$ in our computations, because it has no impact on the existence of cycles, but will just give us $(\mPoissw)_w$ twisted by the sign representation (if Koszulness holds). 
However, note that this means that $\comtree$ becomes anti-commutative, and should become the ``Lie vertex", just like in the case of $\mCom^!$. So when we for example look at $H_*(Cob\wh(\mPoissw)_w(3))$ with one commutative vertex and two Lie vertices, Koszulness would dictate that this space would be isomorphic to $(\mPoissw)_w^!(3)$ with one Lie vertex and two commutative vertices.

Thirdly, we remark that in the case of vector spaces, the homology and cohomology of a space are equal (where we view the cohomology as the homology of the dual space, with the dual of the differential). So we will in fact look at the homology of the cobar complex, with $d^* = \circ_T + \xi_T$ because this differential is easier to compute. This differential pastes two bubbles, and afterwards we might need to use some of the relations  on $\mPoissw$ to make sure that our new element lies in the basis we have. 

And lastly, we now see that our differential admits a filtration: because we only use the Poisson relation $ \begin{tikzpicture}
\draw (0,0.25) -- (0,0);
\draw (0,0) -- (-0.2,-0.2) node [below=-1pt] {\tiny{1}};
\draw (0,0) -- (0.2,-0.2);
\draw (0.2,-0.2) -- (0,-0.4) node [below=-1pt] {\tiny{2}};
\draw (0.2,-0.2) -- (0.4,-0.4) node [below=-1pt] {\tiny{3}};
\draw (0.2,-0.2) circle (1pt) [fill=white];
\end{tikzpicture} = \begin{tikzpicture}
\draw (0,0.25) -- (0,0);
\draw (0,0) -- (-0.2,-0.2) node [below=-1pt] {\tiny{2}};
\draw (0,0) -- (0.2,-0.2);
\draw (0.2,-0.2) -- (0,-0.4) node [below=-1pt] {\tiny{1}};
\draw (0.2,-0.2) -- (0.4,-0.4) node [below=-1pt] {\tiny{3}};
\draw (0,0) circle (1pt) [fill=white];
\end{tikzpicture} + \begin{tikzpicture}
\draw (0,0.25) -- (0,0);
\draw (0,0) -- (-0.2,-0.2);
\draw (0,0) -- (0.2,-0.2) node [below=-1pt] {\tiny{3}};
\draw (-0.2,-0.2) -- (-0.4,-0.4) node [below=-1pt] {\tiny{1}};
\draw (-0.2,-0.2) -- (0,-0.4) node [below=-1pt] {\tiny{2}};
\draw (0,0) circle (1pt) [fill=white];
\end{tikzpicture} $ in one way (from the left to the right), if the $i$th input of a tree $T$ comes from a Lie vertex, the $i$th input of $d^*(T)$ will never be an input from a commutative vertex (we will call inputs from commutative and Lie vertices commutative inputs and Lie inputs, respectively). 
This means that we can use the following filtration: $F_1$ is spanned by trees with only Lie inputs, $F_{2,i}$ is spanned by trees where all inputs are Lie inputs, except for the $i$th one, which can be both. $F_{3,i,j}$ is spanned by trees where all inputs are Lie inputs but the $i$th and $j$th ones can be both, and so on. Then we indeed see that $d^* (F_{j, i_1, ..., i_{j-1}}) \subset F_{j, i_1, ..., i_{j-1}}$.
In $Cob\wh(\mPoissw)_w(n)$, we define $F_j = \bigoplus_{ \{i_1, ... , i_{j-1} \} \subset \{ 1, ... , n-1\} } F_{j, i_1, ..., i_{j-1}}$. We see that all $F_{j, i_1, ..., i_{j-1}} / F_{j-1}$ are isomorphic, and $H_j := H_*(F_j / F_{j-1}) = \bigoplus_{ \{i_1, ... , i_{j-1} \} \subset \{ 1, ... , n-1\} } H_* (F_{j, i_1, ..., i_{j-1}} / F_{j-1})$. 

We can now start computing the homology of $Cob\wh(\mPoissw)_w$. It turns out that with two inputs, $\mPoissw$ satisfies the Koszul condition. But with $H_*(Cob\wh(\mPoissw)_w(3))$ with one commutative vertex and two Lie vertices, Koszulness breaks down. We will skip the explicit calculation of $d^*$ of all elements, and give the homology of each level of filtration:
\begin{tabular}{r l l l l l l l}
$H_3$: & 0 & $\leftarrow$ & 0 & $\leftarrow$ & 0 & $\leftarrow$ & 0 \\
\end{tabular} \\
\begin{tabular}{r l l l l l l l}
$H_2$: & 0 & $\leftarrow$ & 0 & $\leftarrow$ & $\begin{tikzpicture}
\draw (0,0) -- (-0.2,-0.2) node [below=-2pt] {\tiny{1}};
\draw (0,0) -- (0.2,-0.2); 
\draw (0.2,-0.2) -- (0.0,-0.4) node [below=-2pt] {\tiny{2}};
\draw (0.2,-0.2) -- (0.32,-0.32); 
\draw (0.35,-0.35) -- (0.43,-0.43);
\draw (0.43,-0.43) -- (0.23,-0.63) node [below=-2pt] {\tiny{3}}; 
\draw (0.43,-0.43) -- (0.53,-0.53) to [out=315,in=270] (0.78,-0.53) to [out=90,in=90] (0,0.25) -- (0,0);
\fill[white] (0.66,-0.15) circle (1pt);
\draw (0.43,-0.43) circle (1pt) [fill=white];
\end{tikzpicture}$, $\begin{tikzpicture}
\draw (0,0) -- (-0.2,-0.2) node [below=-2pt] {\tiny{2}};
\draw (0,0) -- (0.2,-0.2); 
\draw (0.2,-0.2) -- (0.0,-0.4) node [below=-2pt] {\tiny{3}};
\draw (0.2,-0.2) -- (0.32,-0.32); 
\draw (0.35,-0.35) -- (0.43,-0.43);
\draw (0.43,-0.43) -- (0.23,-0.63) node [below=-2pt] {\tiny{1}}; 
\draw (0.43,-0.43) -- (0.53,-0.53) to [out=315,in=270] (0.78,-0.53) to [out=90,in=90] (0,0.25) -- (0,0);
\fill[white] (0.66,-0.15) circle (1pt);
\draw (0.43,-0.43) circle (1pt) [fill=white];
\end{tikzpicture}$, $\begin{tikzpicture}
\draw (0,0) -- (-0.2,-0.2) node [below=-2pt] {\tiny{3}};
\draw (0,0) -- (0.2,-0.2); 
\draw (0.2,-0.2) -- (0.0,-0.4) node [below=-2pt] {\tiny{1}};
\draw (0.2,-0.2) -- (0.32,-0.32); 
\draw (0.35,-0.35) -- (0.43,-0.43);
\draw (0.43,-0.43) -- (0.23,-0.63) node [below=-2pt] {\tiny{2}}; 
\draw (0.43,-0.43) -- (0.53,-0.53) to [out=315,in=270] (0.78,-0.53) to [out=90,in=90] (0,0.25) -- (0,0);
\fill[white] (0.66,-0.15) circle (1pt);
\draw (0.43,-0.43) circle (1pt) [fill=white];
\end{tikzpicture}$ 
& $\leftarrow$ & 
$\begin{tikzpicture}
\draw (0,0) -- (-0.2,-0.2) node [below=-2pt] {\tiny{1}};
\draw (0,0) -- (0.12,-0.12); 
\draw (0.15,-0.15) -- (0.23,-0.23);
\draw (0.23,-0.23) -- (0.03,-0.43) node [below=-2pt] {\tiny{2}};
\draw (0.23,-0.23) -- (0.35,-0.35); 
\draw (0.38,-0.38) -- (0.46,-0.46);
\draw (0.46,-0.46) -- (0.26,-0.66) node [below=-2pt] {\tiny{3}}; 
\draw (0.46,-0.46) -- (0.56,-0.56) to [out=315,in=270] (0.81,-0.56) to [out=90,in=90] (0,0.25) -- (0,0);
\fill[white] (0.66,-0.15) circle (1pt);
\draw (0,-0) circle (1pt) [fill=white];
\end{tikzpicture} - \begin{tikzpicture}
\draw (0,0) -- (-0.2,-0.2) node [below=-2pt] {\tiny{1}};
\draw (0,0) -- (0.12,-0.12); 
\draw (0.15,-0.15) -- (0.23,-0.23);
\draw (0.23,-0.23) -- (0.03,-0.43) node [below=-2pt] {\tiny{3}};
\draw (0.23,-0.23) -- (0.35,-0.35); 
\draw (0.38,-0.38) -- (0.46,-0.46);
\draw (0.46,-0.46) -- (0.26,-0.66) node [below=-2pt] {\tiny{2}}; 
\draw (0.46,-0.46) -- (0.56,-0.56) to [out=315,in=270] (0.81,-0.56) to [out=90,in=90] (0,0.25) -- (0,0);
\fill[white] (0.66,-0.15) circle (1pt);
\draw (0,-0) circle (1pt) [fill=white];
\end{tikzpicture} - \begin{tikzpicture}
\draw (0,0) -- (0.18,-0.18);
\draw (0.22,-0.22) -- (0.40,-0.40) to [out=315,in=270] (0.65,-0.40) to [out=90,in=90] (0,0.25) -- (0,0);
\draw (0,0) -- (-0.18,-0.18);
\draw (-0.22,-0.22) -- (-0.45,-0.45) node [below=-2pt] {\tiny{2}};
\draw (-0.3,-0.3) -- (-0.15,-0.45) node [below=-2pt] {\tiny{3}};
\draw (0.3,-0.3) -- (0.15,-0.45) node [below=-2pt] {\tiny{1}};
\draw (0.3,-0.3) circle (1pt) [fill=white];
\fill[white] (0.54,-0.09) circle (1pt);
\end{tikzpicture}$ \\
\end{tabular} \\
\begin{tabular}{r l l l l l l l}
$H_1$: & 0 & $\leftarrow$ & $\begin{tikzpicture}
\draw (0,0) -- (0.40,-0.40) to [out=315,in=270] (0.65,-0.40) to [out=90,in=90] (0,0.25) -- (0,0);
\draw (0,0) -- (-0.18,-0.18);
\draw (-0.18,-0.18) -- (-0.45,-0.45) node [below=-2pt] {\tiny{1}};
\draw (-0.3,-0.3) -- (-0.15,-0.45) node [below=-2pt] {\tiny{2}};
\draw (0.3,-0.3) -- (0.15,-0.45) node [below=-2pt] {\tiny{3}};
\draw (0,0) circle (1pt) [fill=white];
\fill[white] (0.54,-0.09) circle (1pt);
\end{tikzpicture} - \begin{tikzpicture}
\draw (0,0) -- (0.40,-0.40) to [out=315,in=270] (0.65,-0.40) to [out=90,in=90] (0,0.25) -- (0,0);
\draw (0,0) -- (-0.18,-0.18);
\draw (-0.18,-0.18) -- (-0.45,-0.45) node [below=-2pt] {\tiny{2}};
\draw (-0.3,-0.3) -- (-0.15,-0.45) node [below=-2pt] {\tiny{3}};
\draw (0.3,-0.3) -- (0.15,-0.45) node [below=-2pt] {\tiny{1}};
\draw (0,0) circle (1pt) [fill=white];
\fill[white] (0.54,-0.09) circle (1pt);
\end{tikzpicture}$, & $\leftarrow$ & $\begin{tikzpicture}
\draw (0,0) -- (0.18,-0.18);
\draw (0.22,-0.22) -- (0.40,-0.40) to [out=315,in=270] (0.65,-0.40) to [out=90,in=90] (0,0.25) -- (0,0);
\draw (0,0) -- (-0.45,-0.45) node [below=-2pt] {\tiny{1}};
\draw (-0.3,-0.3) -- (-0.15,-0.45) node [below=-2pt] {\tiny{2}};
\draw (0.3,-0.3) -- (0.15,-0.45) node [below=-2pt] {\tiny{3}};
\draw (0,0) circle (1pt) [fill=white];
\fill[white] (0.54,-0.09) circle (1pt);
\end{tikzpicture} + \begin{tikzpicture}
\draw (0,0) -- (0.18,-0.18);
\draw (0.22,-0.22) -- (0.40,-0.40) to [out=315,in=270] (0.65,-0.40) to [out=90,in=90] (0,0.25) -- (0,0);
\draw (0,0) -- (-0.45,-0.45) node [below=-2pt] {\tiny{2}};
\draw (-0.3,-0.3) -- (-0.15,-0.45) node [below=-2pt] {\tiny{3}};
\draw (0.3,-0.3) -- (0.15,-0.45) node [below=-2pt] {\tiny{1}};
\draw (0,0) circle (1pt) [fill=white];
\fill[white] (0.54,-0.09) circle (1pt);
\end{tikzpicture} + \begin{tikzpicture}
\draw (0,0) -- (0.18,-0.18);
\draw (0.22,-0.22) -- (0.40,-0.40) to [out=315,in=270] (0.65,-0.40) to [out=90,in=90] (0,0.25) -- (0,0);
\draw (0,0) -- (-0.45,-0.45) node [below=-2pt] {\tiny{3}};
\draw (-0.3,-0.3) -- (-0.15,-0.45) node [below=-2pt] {\tiny{1}};
\draw (0.3,-0.3) -- (0.15,-0.45) node [below=-2pt] {\tiny{2}};
\draw (0,0) circle (1pt) [fill=white];
\fill[white] (0.54,-0.09) circle (1pt);
\end{tikzpicture}$ & $\leftarrow$ & 0 \\
\end{tabular}

So in terms of dimensions, we could maximally arrive at: 

\begin{tabular}{r l l l l}
deg: & 0 & 1 & 2 & 3 \\
dim: & 0 & 2 & 4 & 3 \\
\end{tabular}

Now, we can eliminate: 

$$%
 $$
This means that $H_*(Cob\wh(\mPoissw)_w(3))$ is two-dimensional in degree 3, which is what we would expect. 

Then we are left with five elements in the filtration homology. We see that 
$$%
 $$
so we see that the homology we found in the degree 1 of $F_1$ lies in the image of $d^*$, and that
$$%
 \stackrel{d^*}{\longrightarrow} 0 $$
This element does not lie in the image, as is checked by a quick round of linear algebra, so we have found our Koszulness-breaking cycle.

\chapter*{Populare samenvatting}

In de wiskunde komen we vaak verzamelingen tegen die een binaire operatie (zeg $\circ$) hebben, zoals groepen en algebra's. Dit betekent dat we van twee elementen $a$ en $b$ in de algebra een nieuw element kunnen maken: $a \circ b$. Als je een element $a$ uit de algebra kiest, kan je met $\circ$ een functie van de algebra naar zichzelf maken: $(a \circ ): x \mapsto a \circ x$. Hiermee wordt elk element $a$ een afbeelding op de algebra, met \'{e}\'{e}n ``input", onze $x$. 

Een \textit{operad} breidt dit concept uit: in plaats van \'{e}\'{e}n verzameling, hebben we nu een rij van verzamelingen $\{ \mP(n) | n \geq 1\}$, waar elementen van de eerste verzameling $\mP(1)$ \'{e}\'{e}n input hebben, $\mP(2)$ bestaat uit elementen met twee inputs, et cetera. Dit betekent ook dat $\mP(n)$ $n$ operaties heeft, $\circ_1, ..., \circ_n$, \'{e}\'{e}n voor elke input. We noemen deze operaties \textit{composities}. 
Als $\nu \in \mP(n)$ en $\mu \in \mP(m)$, zal $\nu \circ_i \mu \in \mP(n+m-1)$, omdat je nog $n-1$ iputs over hebt van $\nu$ (allemaal, behalve de $i$de, waar je $\mu$ hebt gestopt) en alle $m$ inputs over hebt van $\mu$. 

Het algebra\"{i}sch opschrijven van dit soort composities wordt al snel verwarrend, vanwege de vele haakjes die je nodig hebt, en het feit dat bij elke compositie het aantel inputs omhoog gaat. Daarom visualiseren we onze elementen en hun composities in boomdiagrammen. Onze compositie ziet er bijvoorbeeld uit als: 
$$ 
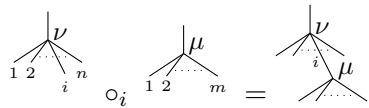
\begin{tikzpicture} [scale=1.5]
\draw (0,0.25) -- (0,0) node [above right= -4pt and -2pt] {\small{$\nu$}};
\draw (0,0) -- (-0.3,-0.2) node [below=-2pt] {\tiny{$1$}};
\draw (0,0) -- (-0.15,-0.2) node [below=-2pt] {\tiny{$2$}};
\draw (0,0) -- (0.15,-0.3) node [below=-1pt] {\tiny{$i$}};
\draw (-0.075,-0.15) -- (0.1875,-0.15) [dotted];
\draw (0,0) -- (0.3,-0.2) node [below=-0.5pt] {\tiny{$n$}};
\end{tikzpicture} \circ_i \begin{tikzpicture} [scale=1.5]
\draw (0,0.25) -- (0,0) node [above right= -4pt and -2pt] {\small{$\mu$}};
\draw (0,0) -- (-0.3,-0.2) node [below=-2pt] {\tiny{$1$}};
\draw (0,0) -- (-0.15,-0.2) node [below=-2pt] {\tiny{$2$}};
\draw (-0.075,-0.15) -- (0.1875,-0.15) [dotted];
\draw (0,0) -- (0.3,-0.2) node [below=-0.5pt] {\tiny{$m$}};
\end{tikzpicture}
=
\begin{tikzpicture} [scale=1.5]
\draw (0,0.25) -- (0,0) node [above right= -4pt and -2pt] {\small{$\nu$}};
\draw (0,0) -- (-0.3,-0.2);
\draw (0,0) -- (-0.15,-0.2);
\draw (0,0) -- (0.2,-0.4) node [above right= -4pt and -2pt] {\small{$\mu$}};
\path (0.125,-0.25) node [left=-2pt] {\tiny{$i$}};
\draw (-0.075,-0.15) -- (0.1875,-0.15) [dotted];
\draw (0,0) -- (0.3,-0.2);
\draw (0.2,-0.4) -- (-0.1,-0.6);
\draw (0.2,-0.4) -- (0.05,-0.6);
\draw (0.125,-0.55) -- (0.3875,-0.55) [dotted];
\draw (0.2,-0.4) -- (0.5,-0.6);
\end{tikzpicture}$$
(onze bomen zijn altijd naar boven gericht: de inputs zitten aan de onderkant)

Met dit idee in ons achterhoofd kunnen we het concept van operaden nog wat uitbreiden, tot operaden met een wiel, waarvan de elementen er bijvoorbeeld zo uit zien: 
$$
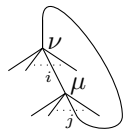
\begin{tikzpicture} [scale=1.5]
\draw (0,0.25) -- (0,0) node [above right= -4pt and -2pt] {\small{$\nu$}};
\draw (0,0) -- (-0.3,-0.2);
\draw (0,0) -- (-0.15,-0.2);
\draw (0,0) -- (0.2,-0.4) node [above right= -4pt and -2pt] {\small{$\mu$}};
\path (0.125,-0.25) node [left=-2pt] {\tiny{$i$}};
\draw (-0.075,-0.15) -- (0.1875,-0.15) [dotted];
\draw (0,0) -- (0.3,-0.2);
\draw (0.2,-0.4) -- (-0.1,-0.6);
\draw (0.2,-0.4) -- (0.05,-0.6);
\draw (0.2,-0.4) -- (0.3,-0.6) to [out=300,in=270] (0.7,-0.6) to [out=90,in=90] (0,0.25) -- (0,-0.00);
\path (0.325,-0.65) node [left=-2pt] {\tiny{$j$}};
\draw (0.125,-0.55) -- (0.3875,-0.55) [dotted];
\draw (0.2,-0.4) -- (0.5,-0.6);
\end{tikzpicture}$$
Deze operaden-met-wiel hebben een extra eigenschap: je kan een deel van je element, dat aan de onderkant van je wiel zit, ``door het wiel heen halen", naar boven. In het simpelste geval betekent dit:
$$
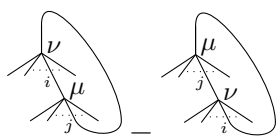
\begin{tikzpicture} [scale=1.5]
\draw (0,0.25) -- (0,0) node [above right= -4pt and -2pt] {\small{$\nu$}};
\draw (0,0) -- (-0.3,-0.2);
\draw (0,0) -- (-0.15,-0.2);
\draw (0,0) -- (0.2,-0.4) node [above right= -4pt and -2pt] {\small{$\mu$}};
\path (0.125,-0.25) node [left=-2pt] {\tiny{$i$}};
\draw (-0.075,-0.15) -- (0.1875,-0.15) [dotted];
\draw (0,0) -- (0.3,-0.2);
\draw (0.2,-0.4) -- (-0.1,-0.6);
\draw (0.2,-0.4) -- (0.05,-0.6);
\draw (0.2,-0.4) -- (0.3,-0.6) to [out=300,in=270] (0.7,-0.6) to [out=90,in=90] (0,0.25) -- (0,-0.00);
\path (0.325,-0.65) node [left=-2pt] {\tiny{$j$}};
\draw (0.125,-0.55) -- (0.3875,-0.55) [dotted];
\draw (0.2,-0.4) -- (0.5,-0.6);
\end{tikzpicture} = \begin{tikzpicture} [scale=1.5]
\draw (0,0.25) -- (0,0) node [above right= -4pt and -2pt] {\small{$\mu$}};
\draw (0,0) -- (-0.3,-0.2);
\draw (0,0) -- (-0.15,-0.2);
\draw (0,0) -- (0.2,-0.4) node [above right= -4pt and -2pt] {\small{$\nu$}};
\path (0.125,-0.25) node [left=-2pt] {\tiny{$j$}};
\draw (-0.075,-0.15) -- (0.1875,-0.15) [dotted];
\draw (0,0) -- (0.3,-0.2);
\draw (0.2,-0.4) -- (-0.1,-0.6);
\draw (0.2,-0.4) -- (0.05,-0.6);
\draw (0.2,-0.4) -- (0.3,-0.6) to [out=300,in=270] (0.7,-0.6) to [out=90,in=90] (0,0.25) -- (0,-0.00);
\path (0.325,-0.65) node [left=-2pt] {\tiny{$i$}};
\draw (0.125,-0.55) -- (0.3875,-0.55) [dotted];
\draw (0.2,-0.4) -- (0.5,-0.6);
\end{tikzpicture}$$
En het blijkt dat operaden met zo'n wiel, vanwege deze eigenschap, niet zomaar alles overnemen van hun neven zonder wiel: zo wordt in deze scriptie laten zien dat de Poisson operad met wiel geen Koszul operad is, terwijl de gewone Poisson operad dat wel is.


\begin{thebibliography}{6}

\bibitem{BDK} M. Bershtein, V. Dotsenko, A. Khoroshkin, Quadratic algebras related to the bihamiltonian operad, Int Math Res Notices Vol. 2007: article ID rnm122, 30 pages

\bibitem{GK} V. Ginzburg, M. Kapranov, Koszul duality for Operads, Duke Math. J. 76 (1994), no. 1, 203--272

\bibitem{Mar} M. Markl, Distributive Laws and the Koszulness, arXiv: hep-th/9409192v1

\bibitem{MMS} M. Markl, S.A. Merkulov and S. Shadrin, Wheeled PROPs, graph complexes and the master equation, J. Pure Appl. Algebra 213 (2009), no. 4, 496--535

\bibitem{Mer} S.A. Merkulov, PROP profile of deformation quantization and graph complexes with loops and wheels, preprint arXiv: math.QA/04122157

\bibitem{MSS} M. Markl, S. Shnider, and J. D. Stasheff. Operads in Algebra, Topology and Physics, volume 96 of Mathematical Surveys and Monographs. American Mathematical Society, Providence, Rhode Island, 2002.

\end{thebibliography}
\end{document}